\documentclass{article}
\usepackage[utf8]{inputenc}
\usepackage{amssymb}
\usepackage{amsfonts}
\usepackage{mathrsfs}
\usepackage{amsmath}
\usepackage{graphicx}
\usepackage{dsfont}
\usepackage{float}
\usepackage{diagbox} 

\usepackage{graphics}
\usepackage[all]{xy}
\usepackage[margin=1.05in,marginparwidth=1in,marginparsep=.05in]{geometry}
\PassOptionsToPackage{hyphens}{url}
\usepackage{tikz-cd}
\usepackage{titlesec}
\usepackage{braket}
\usepackage{xcolor,graphicx}
\usepackage{verbatim}
\usepackage{mathrsfs}
\usepackage{extarrows}
\usepackage{tikz}
\usetikzlibrary{shapes,arrows,positioning}
\usepackage{mathrsfs}
\usepackage{mathtools}

\usepackage{subfig}
\usepackage{hyperref}
\usepackage[percent]{overpic}

\usetikzlibrary{arrows}

\makeatletter
\newcommand*{\rom}[1]{\expandafter\@slowromancap\romannumeral #1@}

\titlespacing*{\section}{0pt}{\baselineskip}{\baselineskip}
\titleformat{\subsection}{\normalfont\bfseries}{\thesubsection.}{1em}{}
\titleformat{\subsubsection}{\normalfont}{\thesubsubsection.}{1em}{\itshape}
\newtheorem{theorem}{Theorem}[section]

\newtheorem{lemma}[theorem]{Lemma}

\newtheorem{remark}[theorem]{Remark}

\numberwithin{equation}{section}

\newenvironment{proof}{\paragraph{Proof:}}{\hfill$\square$}

\newcommand{\R}{\mathbb R}
\newcommand{\bA}{\mathbf A}

\newcommand{\bH}{\mathbf H}
\newcommand{\bI}{\mathbf I}

\newcommand{\bP}{\mathbf P}

\newcommand{\bV}{\mathbf V}

\newcommand{\bg}{\mathbf g}

\newcommand{\blf}{\mathbf f}
\newcommand{\bn}{\mathbf n}
\newcommand{\be}{\mathbf e}
\newcommand{\bp}{\mathbf p}

\newcommand{\bu}{\mathbf u}

\newcommand{\bv}{\mathbf v}
\newcommand{\bw}{\mathbf w}

\newcommand{\bx}{\mathbf x}
\newcommand{\by}{\mathbf y}

\newcommand{\bbf}{\mathbf f}

\newcommand{\T}{\mathcal T}

\newcommand{\divG}{{\mathop{\,\rm div}}_{\Gamma}}
\newcommand{\gradG}{\nabla_{\Gamma}}
\newcommand{\nablaG}{\nabla_{\Gamma}}
\newcommand{\laplG}{\Delta_{\Gamma}}

\newcommand{\OGamma}{\Omega^\Gamma_h}

\renewcommand{\div}{\textrm{div}\ \!}

\newcommand{\tr}{{\rm tr}}

\DeclareGraphicsExtensions{.pdf,.eps,.ps,.eps.gz,.ps.gz,.eps.Y}

\def\cl {\nonumber \\}
\def\el {\nonumber }

\newtheorem{assumption}{Assumption}[section]

\newcommand{\compositenorm}[1]{{\left\vert\kern-0.25ex\left\vert\kern-0.25ex\left\vert #1
    \right\vert\kern-0.25ex\right\vert\kern-0.25ex\right\vert}}

\usepackage{accents}

\title{A decoupled, stable, and linear FEM for a phase-field model of variable density two-phase incompressible surface flow}
\author{
Yerbol Palzhanov 
\and
Alexander Zhiliakov
\and
Annalisa Quaini
\and
Maxim Olshanskii
}

\date{\small Department of Mathematics, University of Houston, Houston, Texas 77204\\ \tt \{ypalzhanov\}\{azhiliakov\}\{aquaini\}\{maolshanskiy\}@uh.edu}

\begin{document}

\maketitle

\begin{abstract} The paper considers a thermodynamically consistent phase-field model of a two-phase flow of incompressible viscous fluids. The model allows for a  non-linear dependence of fluid density on the phase-field order parameter. Driven by applications in biomembrane studies, the model is written for tangential flows of fluids constrained to a surface and consists of (surface) Navier--Stokes--Cahn--Hilliard type equations.
We apply an unfitted finite element method to discretize the system and introduce a fully discrete time-stepping scheme with the following properties: (i) the scheme decouples the fluid and phase-field equation solvers at each time step,
(ii) the resulting two algebraic systems are linear, and (iii) the numerical solution satisfies the same stability bound as the solution of the original system under some restrictions on the discretization parameters. Numerical examples are provided to demonstrate the stability, accuracy, and overall efficiency of the approach. Our computational study of several  two-phase surface flows reveals
some interesting dependencies of flow statistics on the geometry.
\end{abstract}


\section{Introduction}\label{sec:intro}

Due to many applications in the physical and biological sciences,
binary (two-phase) flows have attracted increasing interest in the past
decades. The phase-field method has emerged as a powerful and effective technique for
modeling multi-phase problems in general. This method describes the physical system using a set of
field variables that are continuous across the interfacial regions separating the phases.
The main reasons for the success of the phase-field
methodology are the following: it replaces sharp interfaces by thin transition regions called diffuse interfaces, making
front-tracking unnecessary, and it is based on rigorous mathematics. The classical diffuse-interface
description of binary incompressible fluid with density-matched components is the so-called
\emph{Model H} \cite{RevModPhys.49.435}, which is a Navier--Stokes--Cahn--Hilliard (NSCH) system.
Several generalization to the case of variable density components
have been presented \cite{Abels2012,Aki2014,BOYER200241,DING20072078,DONG20125788,GONG201720,Gong2018,Lowengrub1998,shokrpour2018diffuse}.
Some of these models relax the incompressibility constraint to quasi-compressibility and not all models
satisfy Galilean invariance, local mass conservation or thermodynamic consistency.
In this paper, we focus on the generalization of Model H first presented in
\cite{Abels2012}. This model is thermodynamically consistent when the density of the mixture depends linearly
on the phase-field variable (concentration or volume/surface fraction). We present an extension
of the model in \cite{Abels2012} that is thermodynamically consistent for a general
monotone relation of density and phase-field variable.

Driven by applications in the analysis of lateral organization of plasma membranes (see \cite{Yushutin_IJNMBE2019,zhiliakov2021experimental} for more context),
we adopt this new model to simulate two-phase flow dynamics on arbitrary-shaped
closed smooth surfaces. While there exist many computational studies of multi-component fluid flows
in planar and volumetric domains (see, e.g., \cite{GUO2017144,LIU2020108948,NOCHETTO2016497,YANG2019435,ZHU2020614} for some recent works), the number of papers where NSCH systems are treated on manifolds
is limited. This can be explained by the fact that solving equations numerically on
general surfaces poses additional difficulties that are related to the discretization of tangential differential operators and
the approximate recovery of (possibly) complex shapes. In \cite{nitschke2012finite}, the authors resort to
a stream function formulation, decouple the surface Navier--Stokes problem from the surface Cahn-Hilliad problem,
and use a parametric finite element approach. Yang and Kim \cite{YANG2020113382} experimented with a
finite difference method on staggered marker-and-cell meshes for a NSCH system embedded in the mesh.

In the present paper, we study a \emph{geometrically unfitted} finite element method
for the simulation of two-phase incompressible flow on surfaces.
Our approach builds on earlier work on a unfitted FEM for elliptic PDEs posed on surfaces~\cite{ORG09}
called TraceFEM. Unlike some other geometrically unfitted methods for surface PDEs,
TraceFEM employs sharp surface representation. The surface can be defined implicitly
and no knowledge of the surface parametrization is required.
This method allows to solve for a scalar quantity or a vector field on a surface, for which a
 parametrization or triangulation is not required.
TaceFEM has been extended to the surface Stokes problem in \cite{olshanskii2018finite,ORZ} and the surface Cahn--Hilliard problem in
\cite{Yushutin_IJNMBE2019}. One of the main advantages of TraceFEM is that it
works well for PDEs posed on evolving surfaces, including cases with topological changes
\cite{lehrenfeld2018stabilized}. Although we will not benefit from this advantage in this paper
(as the surfaces are fixed), it will become handy in a future modeling of membrane deformation and fusion for  multicomponent lipid bilayers.

Although technical, the numerical analysis of diffuse-interface models for two-phase fluids with \emph{matching densities}
can be largely built on well-established analyses for the incompressible Navier--Stokes equations and Cahn--Hilliad
equations alone.
See, for example,  \cite{feng2006fully} for the convergence study of a coupled implicit FE method,
\cite{kay2008finite} for the analysis of a decoupled FE scheme, and
\cite{han2017numerical} for the analysis of a linear and second-order in time FE method.
The design of energy stable, efficient, and consistent discretizations of diffuse-interface  models for two-phase fluids
with \emph{non-matching densities} turns out to be more challenging.

Following the ideas from~\cite{guermond2000projection} for the incompressible Navier--Stokes equations with variable density,  in~\cite{shen2010phase}
Shen and Yang suggested a stable time-stepping method for a NSCH system with
the momentum equation modified based on an inconsistent mass conservation equation.
In~\cite{shen2015decoupled}, the same authors introduced a time-stepping scheme
for the thermodynamically consistent model from~\cite{Abels2012}. This scheme (which is
first order and linear, and it decouples fluid and phase equations) introduces some extra terms
in the momentum equation and modifies the advection flux for the order parameter
to achieve energy stability.
In the same paper, the authors present the stability analysis for a semi-discrete case.
In \cite{zhu2019efficient}, the approach from~\cite{shen2015decoupled} was extended with the application of
a scalar auxiliary variable method \cite{shen2018scalar} and the stability was studied for a finite difference scheme on uniform staggered grids.
A fully discrete FE method for the model from~\cite{Abels2012} was analyzed in  \cite{grun2013convergent},
subject to the coupled treatment of the phase and fluid equations and a modification of the momentum equation.
In \cite{garcke2016stable}, this analysis is extended to a posteriori estimates and adaptivity.
The question of building provably stable and computationally efficient methods
for thermodynamically consistent discretization models of two-phase fluids
with variable densities remains largely open.

The present study contributes to the field with a first order in time, linear, and decoupled
FE scheme for our generalization of the model from \cite{Abels2012}. One advantage of our scheme is that
it neither modifies the momentum equation nor alters the advection velocity.
We prove that our scheme is stable under a mild time-step restriction. Moreover, the analysis tracks the
dependence of the estimates on $\epsilon$, the critical model parameter that defines the width of the transition region between phases.
The fact that the model is posed on arbitrary-shaped
closed smooth surfaces and the use of an unfitted finite element method for spatial discretization
add some further technical difficulties to our analysis. However, apart from
these extra technical details, the general line of arguments can be followed
for the simplified to Euclidian setting or extended to other spatial discretization techniques.

Numerical analysis is even less developed for quasi-incompressible diffuse-interface models due to their additional complexity,
e.g.~a more subtle nonlinearity. Here, we would like to mention the
analysis of a semi-discrete method (coupled, but resulting in linear systems at each time step) in
\cite{shokrpour2018diffuse}.

The rest of the paper is organized as follows. In Sec.~\ref{sec:problem}, we state our generalization of the model from \cite{Abels2012}
and its variational formulations. A decoupled linear FEM based on the
application of TraceFEM to our model is presented in Sec.~\ref{sec:method}, while the analysis
is reported in Sec.~\ref{sec:analysis_FE}.  In Sec.~\ref{sec:num_res}, we report several numerical results,
including a convergence test and the simulation of the Kevin--Helmholtz and Rayleigh--Taylor instabilities.
 Sec.~\ref{sec:concl} collects a few concluding remarks.

\section{Problem definition}\label{sec:problem}

On surface $\Gamma$ we consider a heterogeneous mixture of two species with
surface fractions $c_i = S_i/S$, $i = 1, 2$, where $S_i$
are the surface area occupied by the components and $S$ is the surface area of $\Gamma$.
Since $S = S_1 + S_2$, we have $c_1 + c_2 = 1$. Let $c_1$ be the representative surface fraction, i.e~$c = c_1$.
Moreover, let $m_i$ be the mass of component $i$ and $m$ is the total mass.
Notice that density of the mixture can be expressed as
$\rho= \frac{m}{S} = \frac{m_1}{S_1} \frac{S_1}{S} + \frac{m_2}{S_2} \frac{S_2}{S}$. Thus,
\begin{align}
\rho= \rho(c) = \rho_1 c+ \rho_2 (1-c). \label{rho_c}
\end{align}
Similarly, for the dynamic viscosity of the mixture we can write
\begin{align}
\eta=\eta(c)=\eta_1 c+\eta_2(1-c),~\label{nu_c}
\end{align}
where $\eta_1$ and $\eta_2$ are the dynamic viscosities of the two species.

In order to state the model posed on surface $\Gamma$, we need some preliminary notation.
Let  $\bP=\bP(\bx)\coloneqq\bI -\bn(\bx)\bn(\bx)^T$ for $\bx \in \Gamma$ be the orthogonal projection onto the tangent plane.  For a scalar function $p:\, \Gamma \to \mathbb{R}$ or a vector function $\bu:\, \Gamma \to \mathbb{R}^3$  we define
 $p^e\,:\,\mathcal{O}(\Gamma)\to\mathbb{R}$, $\bu^e\,:\,\mathcal{O}(\Gamma)\to\mathbb{R}^3$,
 extensions of $p$ and $\bu$ from $\Gamma$ to its neighborhood $\mathcal{O}(\Gamma)$.
 The surface gradient and covariant derivatives on $\Gamma$ are then defined as
 $\nablaG p=\bP\nabla p^e$ and  $\nabla_\Gamma \bu\coloneqq \bP \nabla \bu^e \bP$. These  definitions are  independent of a particular smooth extension of $p$ and $\bu$ off $\Gamma$.
On $\Gamma$ we consider the surface rate-of-strain tensor \cite{GurtinMurdoch75} given by
\begin{equation} \label{strain}
E_s(\bu) \coloneqq \frac12(\nabla_\Gamma \bu + (\nabla_\Gamma \bu)^T).
 \end{equation}
The surface divergence operators for a vector $\bg: \Gamma \to \R^3$ and
a tensor $\bA: \Gamma \to \mathbb{R}^{3\times 3}$ are defined as:
\[
 \divG \bg  \coloneqq \tr (\gradG \bg), \qquad
 \divG \bA  \coloneqq \left( \divG (\be_1^T \bA),\,
               \divG (\be_2^T \bA),\,
               \divG (\be_3^T \bA)\right)^T,
               \]
with $\be_i$ the $i$th standard basis vector in $\R^3$.

The classical phase-field model to describe the flow of two immiscible, incompressible, and Newtonian fluids
is the so-called \emph{Model H} \cite{RevModPhys.49.435}. One of the fundamental assumptions for Model H
is that the densities of both components are matching. Several extensions have been proposed
to account for the case of non-matching densities; see Sec.~\ref{sec:intro}. Here, we
restrict our attention to
a thermodynamically consistent generalization of Model H first presented in \cite{Abels2012}. For surface based quantities the model
reads:
\begin{align}
\rho\partial_t \bu + \rho(\nabla_\Gamma\bu)\bu - \bP\divG(2\eta E_s(\bu))+ \nabla_\Gamma p &=  -\sigma_\gamma  c \nablaG \mu+ {M \frac{d\rho}{dc}(\nabla_\Gamma\bu)\gradG \mu} + \bbf,  \label{gracke-1} \\
\divG \bu & =0 , \label{gracke-2}\\
\partial_t c +\divG(c\bu)-  \divG \left(M \gradG \mu \right)  &= 0  \label{gracke-3}, \\
\mu &= \frac{1}{\epsilon} f_0' - \epsilon \laplG c.  \label{gracke-4}
\end{align}
Here, $\bu$ is the surface averaged
tangential velocity $\bu = c \bu_1 + (1 - c) \bu_2$, density is given by \eqref{rho_c}
and viscosity by \eqref{nu_c}, $\sigma_\gamma$ is line tension and $\mu$ is the chemical potential defined in \eqref{gracke-4}.
Force vector $\bbf$, with $\bbf \cdot \bn = 0$, is given.
In \eqref{gracke-3} $M $ is the mobility coefficient, while in \eqref{gracke-4}
$f_0(c)$ is the specific free energy of a homogeneous phase and
$\epsilon$ is a parameter related to the thickness of the interface between the two phases.
We recall that in order to have phase separation, $f_0$ must be a non-convex function of $c$. For a comprehensive discussion of the related Navier--Stokes and Cahn--Hilliard equations on surfaces we refer, for example,  to  \cite{jankuhn2017incompressible,EILKS20089727,du2011finite}.

Problem \eqref{gracke-1}--\eqref{gracke-4} has only one additional term with respect to Model H:
the middle term at the right-hand side in eq.~\eqref{gracke-1}. Notice that this term vanishes in the case of matching densities
since $\frac{d\rho}{dc}= (\rho_1-\rho_2)$.
However, the term is crucial for thermodynamic consistency when the densities do not match
and it can be interpreted as  additional momentum flux due to diffusion of
the components driven by the gradient of the chemical potential.


In practice, we are interested in more general relations than \eqref{rho_c} between $\rho$ and $c$, since depending on the choice of $f_0(c)$ (and because of numerical errors while computing) the order parameter may not be constrained in $[0,1]$ and so $\rho$ and $\eta$ based on \eqref{rho_c} may take physically meaningless (even negative) values. Since we do not see how to show the thermodynamic consistency  of \eqref{gracke-1}--\eqref{gracke-4} for non-linear $\rho(c)$, we propose a further modification.
Without the loss of generality let $\rho_1\ge\rho_2$. 
For a general dependence of $\rho$ on $c$, it is reasonable to assume that $\rho$ is a smooth monotonic function of $c$, i.e. $\frac{d\rho}{dc}\ge 0$ (for $\rho_1\ge\rho_2$), and so we can set
\begin{equation}\label{monot}
\frac{d\rho}{dc} = \theta^2.
\end{equation}
Then, we replace \eqref{gracke-1} with
\begin{equation}\label{grache-1m}
  \rho\partial_t \bu + \rho(\nabla_\Gamma\bu)\bu - \bP\divG(2\eta E_s(\bu))+ \nabla_\Gamma p =  -\sigma_\gamma c  \nablaG \mu + {M \theta(\nabla_\Gamma(\theta\bu)\,)\gradG \mu} + \bbf.
\end{equation}
The updated model \eqref{grache-1m},\eqref{gracke-2}--\eqref{gracke-4} obviously
coincides with \eqref{gracke-1}--\eqref{gracke-4} for $\rho(c)$ from \eqref{rho_c},
but exhibits thermodynamic consistency for a general monotone $\rho$--$c$ relation as we show below. The consistency is preserved if $M$ is a non-negative function of $c$ rather than  a constant coefficient.
 Thermodynamic consistent extensions of  \eqref{gracke-1}--\eqref{gracke-4} for a generic smooth $\rho(c)$ (no monotonicity assumption) were also considered in \cite{abels2016weak,abels2019existence} for the purpose of well-posedness analysis. Those extensions introduce yet more term(s) in the momentum equation, so for computational needs and numerical analysis purpose we opt for \eqref{grache-1m}.

From now on, we will focus on problem \eqref{grache-1m}, \eqref{gracke-2}--\eqref{gracke-4}.
This is a Navier--Stokes--Cahn--Hilliard (NSCH) system that needs to be supplemented with the
definitions of mobility $M$ and free energy per unit area $f_0$. Following many of the existing analytic studies,
as well as numerical studies, we assume mobility to be constant (i.e., independent of $c$).
As for $f_0$, a classical choice  is the Ginzburg--Landau double-well potential
\begin{align}\label{eq:f0}
f_0(c) = \frac{\xi}{4} c^2(1 - c)^2,
\end{align}
where $\xi$ defines the barrier height, i.e. the local maximum at $c = 1/2$ \cite{Emmerich2011}.
We set  $\xi = 1$ for the rest of the paper.

For the numerical method, we need a weak (integral) formulation.
We define the spaces
\begin{equation}   \label{defVT}
 \bV_T\coloneqq \{\, \bu \in H^1(\Gamma)^3~|~ \bu\cdot \bn =0\,\},\quad E\coloneqq \{\, \bu \in \bV_T~|~ E_s(\bu)=\mathbf{0}\,\}.
\end{equation}
We define the Hilbert space $\bV_T^0$ as an orthogonal complement of  $E$ in $\bV_T$
(hence $\bV_T^0 \sim \bV_T/E$), and recall the surface Korn's inequality~\cite{jankuhn2017incompressible}:
\begin{equation}\label{Korn}
\|\bu\|_{H^1(\Gamma)}\le C_K \|E_s(\bu)\|,\quad\forall\,\bu\in\bV_T^0.
\end{equation}
In \eqref{Korn} and later we use short notation $\|\cdot\|$ for the $L^2$-norm on $\Gamma$.
Finally, we define $L_0^2(\Gamma)\coloneqq\{\, p \in L^2(\Gamma)~|~ \int_\Gamma p\,ds=0\,\}$.
To devise the weak formulation, one multiplies eq.~\eqref{grache-1m} by $\bv \in  \bV_T$,
eq.~\eqref{gracke-2} by $q \in L^2(\Gamma)$, eq.~\eqref{gracke-3} by $v\in H^1(\Gamma)$,
and eq.~\eqref{gracke-4} by $g\in H^1(\Gamma)$ and integrates all the equations over $\Gamma$.
For eq.~\eqref{gracke-1m} and \eqref{gracke-3}, one employs the integration by parts identity,
which for a closed smooth surface $\Gamma$ reads:
 \begin{equation}\label{int_parts}
 \int_{\Gamma} v\div_\Gamma\bg \, ds= -\int_{\Gamma} \bg\cdot\nabla_\Gamma v\, ds+\int_{\Gamma} \kappa v\bg\cdot\bn \, ds,\quad \text{for}~\bg\in H^1(\Gamma)^3,\, v\in H^1(\Gamma),
 \end{equation}
here $\kappa$ is the sum of principle curvatures.
Identity \eqref{int_parts} is applied to the second term in \eqref{gracke-3} (i.e.~$\bg=c \bu$), which
leads to no contribution from the curvature term since $\bu$ is tangential, and
to the third term in \eqref{gracke-3} (i.e.~$\bg=M \nabla_\Gamma \mu$), which makes the curvature term vanish
as well.
For a similar reason (component-wise), the curvature term vanishes also
when identity \eqref{int_parts} is applied to the diffusion term in \eqref{grache-1m}.

The weak (integral) formulation of the surface NSCH problem \eqref{grache-1m}, \eqref{gracke-2}--\eqref{gracke-4} reads:
Find $(\bu,p,c,\mu) \in \bV_T \times L_0^2(\Gamma) \times H^1(\Gamma) \times H^1(\Gamma)$ 
such that
\begin{align}
&\int_\Gamma \left( \rho \partial_t \bu\cdot\bv + \rho(\nabla_\Gamma\bu)\bu \cdot \bv + 2\eta  E_s(\bu):E_s(\bv)\right) \, ds - \int_\Gamma p\,\divG \bv \, ds =  -\int_\Gamma \sigma_\gamma  c  \nablaG \mu \cdot\bv \, ds  \cl
& \quad \quad + \int_\Gamma M (\nabla_\Gamma(\theta\bu)) (\gradG \mu) \cdot(\theta\bv)  \, ds + \int_\Gamma  \bbf \cdot\bv \, ds, \label{gracke-w1} \\
& \int_\Gamma q\,\divG \bu \, ds = 0, \label{gracke-w2} \\
&\int_\Gamma \partial_t c \,v \, ds - \int_\Gamma c \bu \cdot  \gradG v \, ds + \int_\Gamma M \gradG \mu \cdot \gradG v \, ds = 0, \label{gracke-w3} \\
&\int_\Gamma  \mu \,g \, ds =  \int_\Gamma \frac{1}{\epsilon} f_0'(c) \,g \, ds + \int_\Gamma \epsilon \gradG c \cdot \gradG g \, ds, \label{gracke-w4}
\end{align}
for all $ (\bv,q,v,g) \in \bV_T \times L^2(\Gamma) \times H^1(\Gamma) \times H^1(\Gamma)$.

\begin{remark}
By testing \eqref{gracke-w1} with $\bv = \bu$, \eqref{gracke-w2} with $q = p$,  \eqref{gracke-w3} with $v = \mu$, and
using \eqref{gracke-4}, one readily obtains the following energy equality:
\begin{equation}\label{aux271}
\frac{d}{dt}\int_\Gamma\left(\frac{\rho}2|\bu|^2 +\sigma_\gamma\left(\frac{1}{\epsilon}f_0 + \frac{\epsilon}2 |\nablaG c|^2\right)\right) ds
+ \int_\Gamma 2\eta |E_s(\bu)|^2 ds + \int_\Gamma \sigma_\gamma M |\gradG \mu|^2 ds= \int_\Gamma\blf\cdot\bu\,ds.
\end{equation}
In other words, model \eqref{grache-1m}, \eqref{gracke-2}--\eqref{gracke-4} is thermodynamically consistent,
i.e.~the system is dissipative for $\bbf = \boldsymbol{0}$.
\end{remark}

Indeed, the first two terms in eq.~\eqref{gracke-w1} tested with $\bv = \bu$ can be handled as follows:
\[
\begin{split}
   \int_\Gamma\left(\rho \partial_t\bu \cdot \bu + \rho(\nabla_\Gamma\bu)\bu \cdot \bu\right) ds & =
   \int_\Gamma\left(\frac{\rho}2 \partial_t |\bu|^2 - \frac{1}2\divG(\rho\bu)|\bu|^2\right) ds \\
     &=  \int_\Gamma\left(\frac{1}2\partial_t (\rho|\bu|^2) -\frac12\left( \partial_t \rho +\divG(\rho\bu)\right)|\bu|^2\right) ds.
\end{split}
\]
Using generic $\rho = \rho(c)$, \eqref{monot}, and \eqref{gracke-2} we get
\[
\partial_t \rho +\divG(\rho\bu)=\frac{d\rho}{dc} \left(\partial_t c +\divG(c\bu)\right)=\theta^2\left(\partial_t c +\divG(c\bu)\right).
\]
With the help of \eqref{gracke-3} , this yields
\begin{align}
\int_\Gamma\left( \partial_t \rho +\divG(\rho\bu)\right)|\bu|^2ds &=
-\int_\Gamma M (\gradG \mu)\cdot(\gradG|\theta\bu|^2) ds \cl
&=-2\int_\Gamma M(\gradG (\theta\bu))(\gradG \mu) \cdot (\theta\bu) ds, \el
\end{align}
which will cancel with the second term at the right-hand side in eq.~\eqref{gracke-w1} tested with $\bv = \bu$.
Thus, from eq.~\eqref{gracke-w1} tested with $\bv = \bu$ we obtain the following equality:
\begin{equation}\label{aux247}
\frac{1}2\frac{d}{dt}\int_\Gamma\rho|\bu|^2 ds + \int_\Gamma 2\eta |E_s(\bu)|^2 ds = -\int_\Gamma \sigma_\gamma c  \nablaG \mu \cdot \bu\,ds+\int_\Gamma\blf\cdot\bu\,ds,
\end{equation}
where we have also used eq.~\eqref{gracke-w2} tested with $q = p$.
From eq.~\eqref{gracke-w3} tested with $v = \mu$ and multiplied by $\sigma_\gamma$, we get:
\begin{align}
\int_\Gamma \sigma_\gamma (\bu \cdot  \gradG c) \, \mu\,ds = - \int_\Gamma \sigma_\gamma (\bu \cdot  \gradG \mu) \, c\,ds= - \int_\Gamma \sigma_\gamma \partial_t c \, \mu\,ds  - \int_\Gamma \sigma_\gamma M | \gradG \mu |^2 ds. \label{aux378}
\end{align}
The first term on the right-hand side can be handled using
\eqref{gracke-4} as follows:
\[
\int_\Gamma \partial_t c \, \mu\,ds = \int_\Gamma \partial_t c \left( \frac{1}{\epsilon}f_0' - \epsilon \laplG c\right) ds=
\frac{d}{dt}\int_\Gamma\left( \frac{1}{\epsilon}f_0 + \frac{\epsilon}2 |\nablaG c|^2\right) ds.
\]
Plugging this into \eqref{aux378} and then using what we get in \eqref{aux247}, we obtain the energy balance \eqref{aux271}.

\section{Numerical method}\label{sec:method}

For the discretization of the variational problem \eqref{gracke-w1}--\eqref{gracke-w4}
we apply an unfitted finite element method called trace finite element approach (TraceFEM).
To discretize surface equations, TraceFEM relies on a tessellation of a 3D bulk computational domain
$\Omega$ ($\Gamma\subset\Omega$ holds) into shape regular tetrahedra untangled to the position of $\Gamma$.

We start with required definitions to set up finite element spaces and variational form. A few auxiliary results will be proved.
Then we proceed to the fully-discrete method by introducing a splitting time discretization, which decouples
\eqref{gracke-w1}--\eqref{gracke-w4} into one linear fluid problem and one
phase-field problem per every time step.
\medskip


Surface $\Gamma$ is defined implicitly as the zero level set of a sufficiently smooth
(at least Lipschitz continuous) function $\phi$, i.e.~$\Gamma=\{\bx\in\Omega\,:\, \phi(\bx)=0\}$,
such that $|\nabla\phi|\ge c_0>0$ in a 3D neighborhood $U(\Gamma)$ of the surface.
The vector field $\bn=\nabla\phi/|\nabla\phi|$ is normal on $\Gamma$ and  defines  quasi-normal directions in $U(\Gamma)$.
Let $\T_h$ be  the collection of all tetrahedra such that $\overline{\Omega}=\cup_{T\in\T_h}\overline{T}$.
The subset of tetrahedra that have a \emph{nonzero intersection} with $\Gamma$ is denoted by $\T_h^\Gamma$.
The grid can be refined towards $\Gamma$, however the tetrahedra from $\T_h^\Gamma$ form a quasi-uniform tessellation with the characteristic tetrahedra  size  $h$.
The domain formed by all tetrahedra in $\T_h^\Gamma$ is denoted by $\OGamma$.

On $\T_h^\Gamma$ we use a standard finite element space of continuous functions that are piecewise-polynomials
of degree $k$.
This bulk (volumetric) finite element space is denoted by $V_h^k$:
\[
V_h^k=\{v\in C(\OGamma)\,:\, v\in P_k(T)~\text{for any}~T\in\T_h^{\Gamma}\},
\]
In the trace finite element method formulated below, the traces of functions from $V_h^k$ on $\Gamma$
 are used to approximate the surface fraction and the chemical potential.
Our bulk velocity and pressure finite element spaces are either Taylor--Hood elements on $\OGamma$,
\begin{equation}\label{TH}
 \bV_h = (V_h^{m+1})^3, \quad Q_h = V_h^m \cap L^2_0(\Gamma), 
 \end{equation}
or equal order velocity--pressure elements
\begin{equation}\label{EE}
 \bV_h = (V_h^{m})^3, \quad Q_h = V_h^m \cap L^2_0(\Gamma), \quad m \geq 1.
 \end{equation}
These spaces are employed to discretize the surface Navier--Stokes system.
Surface velocity and pressure will be represented by the traces  of functions from $ \bV_h$ and $Q_h$ on $\Gamma$. In general, approximation orders $k$ (for the phase-field problem) and $m$ (for the fluid problem)
can be chosen to be different.


 \begin{assumption}
 We assume that integrals over $\Gamma$ can be computed exactly, i.e.~we do not consider geometry errors.
 \end{assumption}

 In practice, $\Gamma$ has to be approximated by a (sufficiently accurate) ``discrete'' surface $\Gamma_h \approx \Gamma$
 in such a way that integrals over $\Gamma_h$ can be computed accurately and efficiently.
 For first order finite elements ($m=1$, $k=1$), a straightforward  polygonal approximation of $\Gamma$
 ensures that the geometric approximation error is consistent with the finite element
 interpolation error; see, e.g., \cite{ORG09}. For higher order elements, numerical approximation of $\Gamma$ based on, e.g.,
 isoparametric trace FE can be used to recover the optimal accuracy~\cite{grande2018analysis}
 in a practical situation when parametrization of $\Gamma$ is not available explicitly.

 There are two well-known issues related to the fact that we are dealing with surface and unfitted finite elements:
 \begin{enumerate}
 \item The numerical treatment of condition $\bu \cdot \bn$. Enforcing
the condition $\bu_h \cdot \bn = 0$ on $\Gamma$ for polynomial functions $\bu_h \in  \bV_h$ is inconvenient and may lead to
locking (i.e., only $\bu_h = 0$ satisfies it). Instead, we add a penalty term to the weak
formulation to enforce the tangential constraint weakly.
\item  Possible small cuts of tetrahedra from $\T_h^\Gamma$ by the surface.
For the standard choice of finite element basis functions, this may lead to poorly conditioned algebraic systems.
We recover algebraic stability by adding certain volumetric terms to the finite element formulation.
 \end{enumerate}

 To make the presentation of the finite element formulation more compact, we introduce the following finite element bilinear forms
 related to the Cahn--Hilliard part of the problem:
\begin{align}
& a_\mu(\mu, v) \coloneqq \int_{\Gamma} M \nabla_{\Gamma} \mu \cdot \nabla_{\Gamma}v\,ds +  \tau_\mu \int_{\OGamma} (\bn\cdot\nabla \mu) (\bn\cdot\nabla v)\,d\bx \label{eq:a_mu} \\
& a_c(c, g) \coloneqq  \epsilon\int_{\Gamma} \gradG c \cdot \gradG g \, ds + \tau_c \int_{\OGamma} ( \bn\cdot\nabla c) (\bn\cdot\nabla g) \,d\bx. \label{eq:a_c}
\end{align}
Forms \eqref{eq:a_mu}--\eqref{eq:a_c} are well defined for $\mu,v, c, g \in H^1(\OGamma)$.
The second term in \eqref{eq:a_mu} and \eqref{eq:a_c} take care of the issue of the small  element cuts
(the above issue 2)~\cite{grande2018analysis}.
These terms are consistent up to geometric errors related to the approximation of $\Gamma$ by $\Gamma_h$ and $\bn$ by $\bn_h$ in the following sense: any smooth solution $\mu$ and $c$ can be extended off the surface along (quasi)-normal directions
so that $\bn\cdot\nabla \mu=0$ and $\bn\cdot\nabla c=0$ in $\OGamma$.
We set the stabilization parameters as follows: 
\[
\tau_\mu=h,\quad \tau_c=\epsilon\,h^{-1}.
\]

For the stability of a numerical method, it is crucial that the computed density and viscosity stay positive.
The polynomial double-well potential does not enforce  $c$ to stay within $[0,1]$ interval and hence $\rho$ and
$\eta$  may eventually take negative values, if one adopts the linear relation between
$c$ and $\rho$ in \eqref{rho_c} or between $c$ and $\eta$ in \eqref{nu_c}. Numerical errors may be
another reason for the order parameter to depart significantly from $[0,1]$.
Therefore, assuming without loss of generality that
$\rho_1\ge\rho_2$ and $\eta_1\ge\eta_2$, we first replace \eqref{rho_c} and \eqref{nu_c}
with the following cut-off functions that ensure density and viscosity stay positive:
\begin{equation}\label{eq:cut-off}
\rho(c)=\left\{
\begin{array}{cc}
  \rho_2 & c\le 0 \\
  c\rho_1+(1-c)\rho_2 & c>0
\end{array}
\right.\qquad \eta(c)=\left\{
\begin{array}{cc}
  \eta_2 & c\le 0 \\
  c\eta_1+(1-c)\eta_2 & c>0
\end{array}
\right.
\end{equation}
Note that unlike some previous studies, we clip the linear dependence \eqref{rho_c} and \eqref{nu_c} only from below.
The resulting convexity  of $\rho(c)$ plays a role in the stability analysis later. Nevertheless, \eqref{eq:cut-off} is  not completely satisfactory since $\theta^2=\frac{\delta \rho}{\delta c}$ is discontinuous, while we need $\theta$ to be from $C^1$. To this end, we approximate $\rho(c)$ from \eqref{rho_c} by a smooth monotone convex and uniformly positive function. In our implementation we let
$\theta^2 = \frac{\rho_1-\rho_2}2\left(\tanh(c/\alpha)+1\right)$, with $\alpha=0.1$, and $\rho(c)=\int_0^c\theta^2(t)dt+\rho_2$.

\smallskip

Later we make use of the decomposition of a vector field on $\Gamma$ into its tangential and normal components:
$
\bu=\overline{\bu}+(\bu\cdot\bn)\bn.
$
For the Navier--Stokes part, we introduce the following forms:
\begin{align}
& a(\eta; \bu,\bv) \coloneqq \int_\Gamma 2\eta E_s( \overline{\bu}):  E_s( \overline{\bv})\, ds+\tau \int_{\Gamma}(\bn\cdot\bu) (\bn\cdot\bv) \, ds + \beta_u \int_{\OGamma} [(\bn\cdot\nabla) \bu] \cdot [(\bn\cdot\nabla) \bv]  \, d\bx, \label{eq:a} \\
& c(\rho; \bw, \bu,\bv)\coloneqq \int_\Gamma\rho\bv^T(\nabla_\Gamma\overline{\bu})\bw\, ds 
 +\frac12 \int_\Gamma\widehat{\rho}(\divG \overline{\bw})\overline{\bu}\cdot\overline{\bv}\, ds, \label{eq:c} \\
&b(\bu,q) =  \int_\Gamma \bu\cdot\nabla_\Gamma q \, ds, \label{eq:b} \\
&s(p,q) \coloneqq \beta_p  \int_{\OGamma} (\bn\cdot\nabla p)(\bn\cdot\nabla q) \,d\bx. \label{eq:s}
\end{align}
with $\widehat{\rho}=\rho-\frac{d\rho}{d\,c}c$.
Forms \eqref{eq:a}--\eqref{eq:s} are well defined for $p,q \in H^1(\OGamma)$, $\bu,\bv,\bw \in H^1(\OGamma)^3$.
In \eqref{eq:a}, $\tau>0$ is a penalty parameter to enforce the tangential constraint (i.e., to address the above issue 1),
while $\beta_u\ge0$ in \eqref{eq:a} and $\beta_p\ge0$ in \eqref{eq:s} are stabilization parameters
to deal with possible small cuts. They are set according to \cite{Jankuhn2020}:
\begin{equation} \label{param}
\tau=h^{-2},\quad \beta_p=h, \quad \beta_u=h^{-1}.
\end{equation}
If one uses equal order pressure-velocity trace elements instead of Taylor--Hood elements, then form \eqref{eq:s} should be replaced by
\begin{align}
s(p,q) \coloneqq \beta_p  \int_{\OGamma} \nabla p\cdot\nabla q \, d\bx. \el 
\end{align}
The second term in \eqref{eq:c} is consistent since the divergence of the true tangential velocity  is zero.
To avoid differentiation of  the projector operator, one may use the identity $\nabla_\Gamma\overline{\bu}=\nabla_\Gamma\bu-(\bu\cdot\bn)\bH$ to implement the $a$-form and $c$-form.
For the analysis, we need the identity for the form \eqref{eq:c} given in the following elementary lemma.
\begin{lemma}\label{prop1} For any $\bu,\bw \in H^1(\OGamma)^3$, it holds
\[
   c(\rho; \bw, \bu,\bu)=- \frac12 \int_\Gamma \theta^2 \divG (c\overline{\bw})|\overline{\bu}|^2\, ds.
\]
\end{lemma}
\begin{proof}
 Using the definition of the covariant gradient in terms of tangential operators, $\nablaG\overline{\bu} = \bP(\nabla \overline{\bu})\bP$, and the integration by parts \eqref{int_parts}, we compute for the first integral term in \eqref{eq:c}
  \[
  \begin{split}
  \int_\Gamma\rho\bu^T(\nabla_\Gamma\overline{\bu})\bw\, ds& = \int_\Gamma\rho\overline{\bu}^T(\nabla_\Gamma\overline{\bu})\overline{\bw}\, ds
  = -\int_\Gamma\divG(\rho\overline{\bu}\otimes\overline{\bw})\cdot\overline{\bu}\, ds\\
  &= -\int_\Gamma\rho \divG(\overline{\bw})|\overline{\bu}|^2\, ds-\int_\Gamma\overline{\bu}\nablaG(\rho\overline{\bu})\overline{\bw}\, ds\\
  &= -\int_\Gamma\rho \divG(\overline{\bw})|\overline{\bu}|^2\, ds- \int_\Gamma\rho\bu^T(\nabla_\Gamma\overline{\bu})\bw\, ds- \int_\Gamma(\overline{\bw}\cdot \nablaG\rho)|\overline{\bu}|^2\, ds
  \end{split}
  \]
From this equality and using  $\nablaG\rho=\frac{d\rho}{dc}\nablaG c$  we obtain:
\begin{equation}\label{aux767}
  \int_\Gamma\rho\bu^T(\nabla_\Gamma\overline{\bu})\bw\, ds=-\frac12\int_\Gamma\rho \divG(\overline{\bw})|\overline{\bu}|^2\, ds -\frac12\int_\Gamma\frac{d\rho}{d\,c} (\overline{\bw}\cdot\nablaG c)|\overline{\bu}|^2\, ds
\end{equation}
We recall that $\widehat{\rho}=\rho-\frac{d\rho}{d\,c}c$ and substitute \eqref{aux767} into
the definition of the form \eqref{eq:c}. After straightforward computations, we
arrive at the result:
\[
\begin{split}
c(\rho; \bw, &\bu,\bu)= \int_\Gamma\rho\bu^T(\nabla_\Gamma\overline{\bu})\bw\, ds
 +\frac12 \int_\Gamma(\rho-\frac{d\rho}{d\,c}c)(\divG \overline{\bw})|\overline{\bu}|^2\, ds\\
 &=-\frac12\int_\Gamma\rho \divG(\overline{\bw})|\overline{\bu}|^2\, ds -\frac12\int_\Gamma\frac{d\rho}{d\,c} (\overline{\bw}\cdot\nablaG c)|\overline{\bu}|^2\, ds
+\frac12 \int_\Gamma(\rho-\frac{d\rho}{d\,c}c)(\divG \overline{\bw})|\overline{\bu}|^2\, ds\\
&= -\frac12\int_\Gamma\frac{d\rho}{d\,c} (\overline{\bw}\cdot\nablaG c)|\overline{\bu}|^2\, ds
-\frac12 \int_\Gamma\frac{d\rho}{d\,c}c(\divG \overline{\bw})|\overline{\bu}|^2\, ds\\
&=- \frac12 \int_\Gamma \frac{d\rho}{d\,c} \divG (c\overline{\bw})|\overline{\bu}|^2\, ds
=- \frac12 \int_\Gamma \theta^2 \divG (c\overline{\bw})|\overline{\bu}|^2\, ds.
\end{split}
\]
\end{proof}

 For the numerical experiments in Sec.~\ref{sec:num_res}, we also add to bilinear form \eqref{eq:a} the grad-div stabilization term~\cite{olshanskii2002low},  $\widehat{\gamma}\,\int_\Gamma \divG \bu \, \divG \bv\,ds$. This term is not essential, in particular for the analysis in Sec.~\ref{sec:analysis_FE},
but we find it beneficial  for the performance of the iterative algebraic solver and for the
overall accuracy of the solution. We set the grad-div stabilization parameter $\widehat{\gamma}=1$.

At time instance $t^n=n\Delta t$, with time step $\Delta t=\frac{T}{N}$,  $\zeta^n$ denotes the approximation of generic
variable $\zeta(t^n, \bx)$. Further, we introduce the following notation  for a first order approximation of the
 time derivative:
 \begin{equation}\label{BDF1}
\left[\zeta\right]_t^{n} =\frac{\zeta^{n}- \zeta^{n-1}}{\Delta t}.
\end{equation}

The decoupled linear finite element method analyzed and tested in this paper reads as follows.
At time step $t^{n+1}$, perform:
\begin{itemize}
\item[-] \underline{Step 1}: Given $\bu^n_h\in \bV_h$ and $c^n_h\in V_h^k$, find $(c^{n+1}_h, \mu^{n+1}_h) \in V^k _h \times V^k _h$
such that:
\begin{align}
&\left(\left[c_h\right]_t^{n+1}, v_h\right) - \left(\bu^{n}_h  c^{n+1}_h, \nablaG v_h\right) + a_\mu(\mu_h^{n+1}, v_h) = 0,  \label{eq:CH_FE1} \\
&\left(  \mu_h^{n+1} - \frac{\gamma_c\Delta t}{\epsilon}\left[c_h\right]_{t}^{n+1}  - \frac{1}{\epsilon} f'_0(c_h^{n}),\,g_h\right)
- a_c(c_h^{n+1}, g_h) = 0,  \label{eq:CH_FE2}
  \end{align}
for all  $(v_h, g_h) \in V^k _h \times V^k _h$.
\item[-] \underline{Step 2}: Set 
$\theta^{n+1} =\sqrt{\frac{d\rho}{d c}(c^{n+1}_h)}$.
Find $(\bu_h^{n+1}, p_h^{n+1}) \in \bV_h \times Q_h$ such that
 \begin{align}
&  (\rho^n\left[\overline{\bu}_h\right]_t^{n+1},\bv_h)+ c(\rho^{n+1}; \bu^{n}_h, \bu^{n+1}_h,\bv_h) + a(\eta^{n+1};\bu_h^{n+1},\bv_h)  + b(\bv_h,p_h^{n+1}) \cl
& \quad \quad = -(\sigma_\gamma c^{n+1}_h \nablaG \mu^{n+1}_h, \bv_h) + M \left((\nabla_\Gamma(\theta^{n+1}\overline{\bu}_h^{n+1}))\gradG \mu^{n+1}_h, \theta^{n+1}\bv_h\right) + (\bbf_h^{n+1}, \bv_h)  \label{NSEh1} \\
 & b(\bu_h^{n+1},q_h)-s(p_h^{n+1},q_h)  = 0  \label{NSEh2}
 \end{align}
for all  $(\bv_h,q_h) \in \bV_h \times Q_h$.
\end{itemize}

At each sub-step of the scheme, a linear problem (Chan--Hilliard type system at step~1 and linearized Navier--Stokes system at step 2)
has to be solved. This allows us to achieve low computational costs, while the scheme is provably stable under relatively mild restrictions.   Moreover, the results of numerical experiments in Sec.~\ref{sec:num_res} do not show  that any restrictions on the discretization parameters are required in practice (Remark~\ref{Rem2} discusses what arguments in our analysis require these restrictions).

Before proceeding with analysis, we note that the above scheme does not modify the advection velocity in  \eqref{eq:CH_FE1} for the purpose of analysis, unlike some other linear decoupled schemes for the NSCH equations found in the literature.  We also avoid another
common helpful trick  to prove  energy stability of the variable density NSCH, namely the modification of the  momentum equation  based on a mass conservation condition of the form
$
\partial_t \rho +\bu\cdot\nablaG c -  \divG \left(M \gradG \mu \right)  = 0,
$
which follow from  \eqref{gracke-3} and \eqref{rho_c}. This modification, however, is not consistent if a non-linear relation between $\rho$ and $c$ is adopted; also it introduces several extra terms in the finite element formulation.

\section{Analysis of the decoupled Finite Element method}\label{sec:analysis_FE}

For the analysis in this section, we assume no forcing term, i.e.~$\bbf^{n+1} = \boldsymbol{0}$ for all $n$.
To avoid technical complications related to handling Killing vector fields (see, e.g. discussion in \cite{bonito2020divergence}),
we shall also assume that the surface does not support any tangential rigid motions, and so $\bV_T^0=\bV_T$.

We further use the $a\lesssim b$ notation if inequality $a\le c\,b$ holds between quantities $a$ and $b$ with a constant $c$ independent
of $h$, $\Delta t$, $\epsilon$, and the position of $\Gamma$ in the background mesh. We give similar meaning to $a\gtrsim b$, and $a\simeq b$ means that both $a\lesssim b$ and $a\gtrsim b$ hold.
\smallskip

 The following lemma will be useful in the proof of the main result, which is reported
 in Theorem \ref{Th:main}.

\begin{lemma}\label{Lemma1} It holds
\begin{equation}\label{eq:Linf}
\|v_h\|_{L^\infty(\Gamma)} \lesssim |\ln h|^{\frac12}\|v_h\|_{H^1(\Gamma)}+ h^{-\frac12}\left\|\bn\cdot\nabla v_h\right\|_{L^2(\Omega_h^\Gamma)}\qquad  \forall\,v_h\in V_h^k.
\end{equation}
For $v_h$ satisfying $\int_\Gamma v_h\,ds=0$, the factor $\|v_h\|_{H^1(\Gamma)}$ on the r.h.s.~can be replaced by $\|\nablaG v_h\|_{L^2(\Gamma)}$.
\end{lemma}
\begin{proof}
Denote by $\bp(\bx)\in \Gamma$, $\bx\in\Omega$, the closest point projection on $\Gamma$.
Since $\Gamma$ is smooth, there is a tubular neighborhood  of $\Gamma$
\[
U=\{\bx\in\mathbb{R}^3\,:\, \mbox{dist}(\bx,\Gamma)<d\}, \quad d>0,
\]
such that $\bx=\bp(\bx)+s\bn$, $\bx\in U$, defines
the local coordinate system $(s,\by)$, with $\by=\bp(\bx)$ and $|s|=\mbox{dist}(\bx,\Gamma)$.
We can always assume $h\le h_0=O(1)$, such that $\Omega_h^\Gamma\subset U$.
For $u\in H^1(U)$, we have in $U$
\begin{equation}\label{aux413}
u(\bx)= u(\by)+\int_{0}^{s}\bn\cdot\nabla u(r,\by)\,dr.
\end{equation}
Denote by $\widetilde\Omega_h^\Gamma$ a ``reachable from $\Gamma$'' part of $\Omega_h^\Gamma$ in the following sense: for any $\bx\in \widetilde\Omega_h^\Gamma$ the interval $(\bx,\bp(\bx))$ is completely inside $\Omega_h^\Gamma$.
Let
\[
g_{\pm}(\by)=\pm\sup\{s\in\mathbb{R}\,:\, (\by\pm s\bn,\by)\subset\Omega_h^\Gamma\},\quad \by\in\Gamma,
\]
where $g_{\pm}(\by)$ are piecewise smooth and $\|g_{\pm}\|_{L^\infty(\Gamma)}\lesssim h$.
Thanks to the co-area formula and the smoothness of $\Gamma$, it holds
\begin{equation}\label{aux423}
\int_{\widetilde\Omega_h^\Gamma}|f|\,d\bx\simeq  \int_{\Gamma}\int_{g_{-}}^{g_{+}}|f|\,d\by\,ds,\quad \mbox{for any}~f\in L^1(\widetilde\Omega_h^\Gamma).
\end{equation}

From \eqref{aux413} and \eqref{aux423}, we have:
\[
\|u(\bx)\|_{L^p(\widetilde\Omega_h^\Gamma)}\simeq
\left(\int_{\Gamma}\int_{g_{-}}^{g_{+}} \left|u(\by)+\int_{0}^{s}\bn\cdot\nabla u(r,\by)\,dr\right|^p\,ds\,d\by\right)^{\frac1p},
\]
 for any real exponent $p> 1$.
Triangle inequality and inequality:
\[
\left(\int_{\Gamma}\int_{g_{-}}^{g_{+}} \left|u(\by)\right|^p\, ds\,d\by\right)^{\frac1p}\le |g_{+}-g_{-}|^{\frac1p}\|u\|_{L^p(\Gamma)}
\]
yield
\begin{equation}\label{aux421}
\|u\|_{L^p(\widetilde\Omega_h^\Gamma)}\lesssim
h^{\frac1p}\|u\|_{L^p(\Gamma)}+ \left(\int_{\Gamma}\int_{g_{-}}^{g_{+}} \left|\int_{0}^{s}\bn\cdot\nabla u(r,\by)\,dr\right|^p\,ds\,d\by\right)^{\frac1p}.
\end{equation}
To handle the second term on the right-hand side, we apply H\"{o}lder's inequality  with $p'=p/(p-1)$:
\[
\begin{split}
\left(\int_{\Gamma}\int_{g_{-}}^{g_{+}} \left|\int_{0}^{s}\bn\cdot\nabla u(r,\by)\,dr\right|^p\,ds\,d\by\right)^{\frac1p}
&\le
\left(\int_{\Gamma}\int_{g_{-}}^{g_{+}}\left| \left(\int_{0}^{s}\left|\bn\cdot\nabla u(r,\by)\right|^p\,dr\right)^{\frac1p} \left(\int_{0}^{s}\,dr\right)^{\frac1{p'}}\right|^p\,ds\,d\by\right)^{\frac1p}\\
&=
\left(\int_{\Gamma}\int_{g_{-}}^{g_{+}}\left( \int_{0}^{s}\left|\bn\cdot\nabla u(r,\by)\right|^p\,dr \right) s^{p-1}\,ds\,d\by\right)^{\frac1p}\\
&\le
\left(\int_{\Gamma}\int_{g_{-}}^{g_{+}} \left|\bn\cdot\nabla u(r,\by)\right|^p\,dr\, d\by\right)^{\frac1p} \left(\int_{g_{-}}^{g_{+}} s^{p-1}\,ds\right)^{\frac1p}\\
&\lesssim
 h \left\|\bn\cdot\nabla u\right\|_{L^p(\widetilde\Omega_h^\Gamma)}
\end{split}
\]
Substituting this in \eqref{aux421} and using $\widetilde\Omega_h^\Gamma\subset\Omega_h^\Gamma$ we get
\begin{equation*}
\|u\|_{L^p(\widetilde\Omega_h^\Gamma)}\lesssim
h^{\frac1p}\|u(\bx)\|_{L^p(\Gamma)}+  h \left\|\bn\cdot\nabla u\right\|_{L^p(\Omega_h^\Gamma)} .
\end{equation*}
Letting $u=u_h\in V_h^k$ and applying FE inverse inequality to treat the last term, we arrive at
\begin{equation}\label{aux445}
\|u_h\|_{L^p(\widetilde\Omega_h^\Gamma))}\lesssim
h^{\frac1p}\|u_h\|_{L^p(\Gamma)}+  h^{\frac3p-\frac12} \left\|\bn\cdot\nabla u_h\right\|_{L^2(\Omega_h^\Gamma)} .
\end{equation}

Next, we need the following technical result from Lemma~7.9 in~\cite{grande2018analysis}: There is
$\delta\simeq h$, depending only on the shape regularity of the mesh $\T_h^\Gamma$
such that for any $T\in \T_h^\Gamma$ there exists a ball $B_\delta(T)\subset T\cap \widetilde\Omega_h^\Gamma$ of radius $\delta$. Since on every tetrahedron $u_h$ is polynomial function of a fixed degree by the standard norm equivalence argument, we have
\[
\left(\int_{T}|u_h|^p\,d\bx\right)^{\frac1p} \le  C\left(\int_{B_\delta(T)}|u_h|^p\,d\bx\right)^{\frac1p},
\]
with $C$ depending only on $\delta$ and the shape regularity of $T$. Raising both parts of this inequality to power $p$, summing over all $T\in \T_h^\Gamma$,  raising both  parts to power $1/p$ and using $B_\delta(T)\subset \widetilde\Omega_h^\Gamma$ we get
\[
\|u_h\|_{L^p(\Omega_h^\Gamma)}\lesssim \|u_h\|_{L^p(\widetilde\Omega_h^\Gamma)}
\]

We now use this in \eqref{aux445} and apply another FE inverse inequality to get
\begin{equation}\label{aux470}
\|u\|_{L^\infty(\Gamma)}\le \|u\|_{L^\infty(\Omega_h^\Gamma)}\lesssim h^{-\frac3p}\|u_h\|_{L^p(\Omega_h^\Gamma)}
\lesssim
h^{-\frac2p}\|u_h\|_{L^p(\Gamma)}+  h^{-\frac12} \left\|\bn\cdot\nabla u_h\right\|_{L^2(\Omega_h^\Gamma)}.
\end{equation}
For $u\in H^1(\Gamma)$, recall that the Sobolev embedding theorem implies $u\in L^p(\Gamma)$, $p\in[1,\infty)$, and
\[
\|u\|_{L^p(\Gamma)}\le c\,p^{\frac12} \|u\|_{H^1(\Gamma)}.
\]
This result follows from the corresponding embedding theorem in $\mathbb{R}^2$ by
standard arguments based on the surface local parametrization and partition of unity.
Using this in \eqref{aux470},  we obtain the estimate
\begin{equation}\label{aux485}
\|u\|_{L^\infty(\Gamma)}
\lesssim
h^{-\frac2p} p^{\frac12} \|u_h\|_{H^1(\Gamma)}+  h^{-\frac12} \left\|\bn\cdot\nabla u_h\right\|_{L^2(\Omega_h^\Gamma)}.
\end{equation}
Letting $p=|\ln h|$ proves the result in \eqref{eq:Linf}. Applying the Poincar\'e inequality for the function $u_h$ satisfying
$\int_\Gamma u_h\,ds=0 $ proves the second claim of the lemma.
\end{proof}

Following \cite{Shen_Yang2010}, we modify the double-well potential in \eqref{eq:f0} for the purpose of analysis so that it is $C^2$ smooth but has quadratic growth for large $|c|$. Straightforward calculations give us the following expression for  $f_0'(c)$ with sufficiently large but fixed  $\alpha > 1$:
\begin{align*}
{f}_0'(c)=
 \begin{cases}
\frac{3\alpha^2 - 1}{4} c - \left(\frac{\alpha^3}{4} + \frac{3}{8} \alpha^2 - \frac{1}{8} \right) ,\quad c>\frac{1+\alpha}{2},\\
\left( c^2 - \frac{3}{2} c + \frac{1}{2} \right)c, \quad c\in\left[\frac{1-\alpha}{2},\frac{1+\alpha}{2}\right],\\
\frac{3\alpha^2 - 1}{4} c + \left(\frac{\alpha^3}{4} - \frac{3}{8} \alpha^2 + \frac{1}{8} \right),\quad c<\frac{1-\alpha}{2}.
 \end{cases}
\end{align*}
Function $f_0'(x)$ satisfies the following Lipschitz  condition with $L=\frac{3\alpha^2 - 1}{4}$:
\begin{align}\label{variation}
-\frac{1}{4}\leq\frac{{f}_0'(x)-{f}_0'(y)}{ x-y}\leq{}L,\quad\forall x,y \in \mathbb{R},~  x\neq y,
\end{align}
and growth condition
	\begin{align}\label{growth}
	|{f}_0'(x)|\leq L |x|. 
	\end{align}

\begin{theorem}\label{Th:main}
\label{Th1} Assume  $h$ and $\Delta t$ satisfy $\Delta t \le c |\ln h|^{-1} \epsilon$ and  $h \le c |\ln h|^{-1} \min\{\Delta t, |\ln h|^{-\frac12} \epsilon |\Delta t|^{\frac12}\}$  for some sufficiently small $c>0$, independent of $h$, $\Delta t$, $\epsilon$ and position of $\Gamma$ in the background mesh. Then, it holds
 \begin{equation}
 \int_\Gamma\left(\rho^{N}|\overline{\bu}_h^{N}|^2 +\frac{\sigma_\gamma}{\epsilon}f_0(c^{N}_h)\right) ds  + a_c(c^{N}_h,c^{N}_h) + \sum_{n=1}^{N}\Delta t \left( a(\eta^{n}; \bu^{n}_h,\bu^{n}_h)+ a_\mu(\mu^{n}, \mu^{n})+ s_h(p_h^{n},p_h^{n})\right)
 \le
 K,  \label{stab_est}
\end{equation}
for all $N=1,2,\dots$, with $K=\int_\Gamma\left(\rho^{0}|\bu^{0}_h|^2 +\frac{\sigma_\gamma}{\epsilon}f_0(c^{0}_h)\right) ds + a_c(c^{0}_h,c^{0}_h)$.
\end{theorem}
\begin{proof}
We use induction in $N$ to prove \eqref{stab_est}. For $N=0$, the estimate is trivial and  provided by the initial condition.
For the induction step, assume that it holds with $N=n$.

Letting $v_h=\mu^{n+1}_h$ in \eqref{eq:CH_FE1} and $g_h=-\left[c_h\right]_{t}^{n+1}$ in \eqref{eq:CH_FE2} and
adding the two equations together, we get
\begin{multline}\label{aux549}
-(c^{n+1}_h \bu^{n}_h, \nablaG\mu^{n+1}_h)
+  a_\mu (\mu^{n+1}_h,\mu^{n+1}_h)+ \left(\frac{1}{\epsilon}\left[c_h\right]_{t}^{n+1}, {f}_0'(c^n_h)\right) \\
 + \frac{1}{2\Delta t}\left( a_c(c^{n+1}_h,c^{n+1}_h)-a_c(c^{n}_h,c^{n}_h)+ |\Delta t|^2 a_c(\left[c_h\right]_{t}^{n+1},\left[c_h\right]_{t}^{n+1}) \right)
+ \frac{\gamma_c\Delta t}{\epsilon}\|\left[c_h\right]_{t}^{n+1}\|^2=0.
\end{multline}
With  the truncated Taylor expansion 
$
f_0(c^{n+1}_h)=f_0(c^n_h)+|\Delta t|\left[c_h\right]_{t}^{n+1} {f}_0'(c^n_h)+ |\Delta t|^2 |\left[c_h\right]_{t}^{n+1}|^2 {f}_0''(\xi^n),
$
and \eqref{variation}  we get:
\[
\begin{split}
 \left(\frac{1}{\epsilon}\left[c_h\right]_{t}^{n+1}, {f}_0'(c^n_h)\right)&=
\frac{1}{\epsilon}\int_\Gamma\frac{f_0(c^{n+1}_h)-f_0(c^n_h)}{\Delta t}\,ds-
\frac{\Delta t}{2\epsilon}\int_\Gamma|\left[c_h\right]_{t}^{n+1}|^2{f}_0''(\xi^n) ds\\
&\ge
\frac{1}{\epsilon}\int_\Gamma\frac{f_0(c^{n+1}_h)-f_0(c^n_h)}{\Delta t}\,ds-
\frac{L\Delta t}{2\epsilon}\|\left[c_h\right]_{t}^{n+1}\|^2.
\end{split}
\]
Since  the second term at the right-hand side has a negative sign, we
let $\gamma_c$ be large enough in order to cancel it with the $\frac12$ of the last term on the left-hand side of \eqref{aux549}.
We obtain
\begin{multline}\label{aux573}
 a_\mu (\mu^{n+1}_h,\mu^{n+1}_h)+ \frac{1}{\epsilon}\int_\Gamma\frac{f_0(c^{n+1}_h)-f_0(c^n_h)}{\Delta t}\,ds
 + \frac{1}{2\Delta t}\left( a_c(c^{n+1}_h,c^{n+1}_h)-a_c(c^{n}_h,c^{n}_h)\right) \\
 + \Delta t\left( \frac{1}{2} a_c(\left[c_h\right]_{t}^{n+1},\left[c_h\right]_{t}^{n+1})
+ \frac{\gamma_c}{2\epsilon}\|\left[c_h\right]_{t}^{n+1}\|^2 \right) \le (c^{n+1}_h \bu^{n}_h, \nablaG\mu^{n+1}_h).
\end{multline}
After re-arranging terms, multiplying by $\Delta t$, and dropping some non-negative terms on the left-hand side,
we obtain the following inequality
\begin{multline}\label{aux573b}
 \frac{1}{\epsilon}\int_\Gamma f_0(c^{n+1}_h) ds +\Delta t a_\mu (\mu^{n+1}_h,\mu^{n+1}_h)
 + \frac{1}{2}a_c(c^{n+1}_h,c^{n+1}_h)
+\Delta t^2 \frac{\gamma_c}{2\epsilon}\|\left[c_h\right]_{t}^{n+1}\|^2
 \\
\le  \frac{1}{\epsilon}\int_\Gamma f_0(c^{n}_h) ds+\frac{1}{2} a_c(c^{n}_h,c^{n}_h) + \Delta t (c^{n+1}_h \bu^{n}_h, \nablaG\mu^{n+1}_h).
\end{multline}

The first two terms on the right-hand side of \eqref{aux573b} are bounded due to the induction assumption.
To handle the third term on the right-hand side, we let $c_0=|\Gamma|^{-1}\int_\Gamma c_h^{n+1}\,ds$ and
use Lemma~\ref{Lemma1} that yields $\|c^{n+1}_h-c_0\|_{L^\infty(\Gamma)}\lesssim |\ln h|^{\frac12}\epsilon^{-\frac12} |a_c(c^{n+1}_h, c^{n+1}_h)|^{\frac12} $ by the definition of the $a_c$ form in \eqref{eq:a_c}.
This, the Cauchy inequality, and the equality $\bu^{n}_h\cdot\nablaG\mu^{n+1}_h=\overline{\bu}_h^{n}\cdot\nablaG\mu^{n+1}_h$ help us with the following estimate:
\begin{equation}\label{aux613}
\begin{split}
\Delta t |(c^{n+1}_h& \bu^{n}_h, \nablaG\mu^{n+1}_h)|\le \Delta t |(\, (c^{n+1}_h-c_0) \bu^{n}_h, \nablaG\mu^{n+1}_h)| +\Delta t |(c_0 \bu^{n}_h, \nablaG\mu^{n+1}_h)|\\
&\le \Delta t\|\overline{\bu}_h^{n}\| \|c^{n+1}_h-c_0\|_{L^\infty(\Gamma)}\|\nablaG \mu^{n+1}_h\| + \Delta t \|\overline{\bu}_h^{n}\| |c_0| \|\nablaG \mu^{n+1}_h\|\\
&\le \frac{C\Delta t}{\sqrt{\epsilon}}\|\overline{\bu}_h^{n}\||\ln h|^{\frac12} |a_c(c^{n+1}_h, c^{n+1}_h)|^{\frac12}  \|\nablaG \mu^{n+1}_h\| + \Delta t \|\overline{\bu}_h^{n}\| |c_0| \|\nablaG \mu^{n+1}_h\|\\
&\le  \Delta t |a_\mu(\mu^{n+1}_h, \mu^{n+1}_h)|^{\frac12} \sqrt{\frac{K}{\rho_2M}} \left( \frac{C|\ln h|^{\frac12}}{\sqrt{\epsilon}} |a_c(c^{n+1}_h, c^{n+1}_h)|^{\frac12}   + |c_0| \right).
\end{split}
\end{equation}
In last inequality in \eqref{aux613}, we used the induction assumption to estimate $\|\bu^{n}_h\|$ and the fact
that $\rho^n\ge\rho_2$ by definition of the $\rho(c)$ function.  Letting $v_h=1$ in \eqref{eq:CH_FE1}  we have
\[
\int_\Gamma \left[c_h\right]_{t}^{n+1}\,ds=0\quad\Rightarrow\quad c_0 = |\Gamma|^{-1}\int_\Gamma c_h^{0}\,ds.
\]
We conclude that $|c_0|$ in \eqref{aux613} can be bounded by a constant depending only on the initial data.
Then, using Young's inequality in \eqref{aux613} we get the following bound
\[
\Delta t |(\bu^{n}_h\cdot\nablaG \mu^{n+1}_h, c^{n+1}_h)|
\le
 \frac14 a_c(c^{n+1}_h, c^{n+1}_h) +  |\Delta t|^2 \frac{C_1|\ln h| }{\epsilon}a_\mu(\mu^{n+1}_h, \mu^{n+1}_h) + C_2,
\]
with some constants $C_1$, $C_2$ independent of $h$, $\Delta t$, $\epsilon$, and position of $\Gamma$ in the
background mesh.
Using this back in \eqref{aux573b} with $\Delta t$ satisfying assumptions of the theorem,  and applying \eqref{stab_est} for $N=n$ for the remainder terms on the right-hand side of  \eqref{aux573b}, we get:
\begin{equation}\label{aux613b}
 \frac{1}{\epsilon}\int_\Gamma f_0(c^{n+1}_h) ds +\frac{\Delta t}{2}  a_\mu (\mu^{n+1}_h,\mu^{n+1}_h)
 + \frac{1}{4}a_c(c^{n+1}_h,c^{n+1}_h) +\frac{\Delta t^2 \gamma_c}{2\epsilon}\|\left[c_h\right]_{t}^{n+1}\|^2  \le C,
\end{equation}
with some $C>0$ independent of $h$, $\Delta t$, $\epsilon$ and the position of $\Gamma$ in the mesh. We need this bound later in the proof.

 Let  $\bv_h=\bu^{n+1}_h$ in \eqref{NSEh1} and $q_h=-p_h$ in \eqref{NSEh2}, and add the two equations together. We also apply Lemma~\ref{prop1}
 to handle the $c$-form. This brings us to the following equality:
 \begin{align}
&\frac{1}{2\Delta t}\left( \|(\rho^n)^{\frac12}\overline{\bu}_h^{n+1}\|^2-\|(\rho^n)^{\frac12}\overline{\bu}_h^n\|^2+ |\Delta t|^2\|(\rho^n)^{\frac12}\left[\overline{\bu}_h\right]_t^{n+1}\|^2\right) \cl
& \quad - \frac12 \int_\Gamma |\theta^{n+1}|^2\divG (c_h^{n+1}\overline{\bu}_h^{n})|\overline{\bu}_h^{n+1}|^2\, ds
+  a(\eta^{n+1};\bu^{n+1}_h,\bu^{n+1}_h) + s(p_h^{n+1},p_h^{n+1}) \cl
& \quad =  -(\sigma_\gamma c^{n+1}_h\nablaG  \mu^{n+1}_h , \bu^{n+1}_h)
+ M \left((\nabla_\Gamma(\theta^{n+1}\overline{\bu}_h^{n+1}))\gradG \mu^{n+1}_h, \theta^{n+1}\bu^{n+1}_h\right). \el
\end{align}
By adding $\pm\frac{1}{2\Delta t}\|(\rho^{n+1})^{\frac12}\overline{\bu}_h^{n+1}\|^2$ and re-arranging terms, we get
\begin{align}
&\frac{1}{2\Delta t}\left(\|(\rho^{n+1})^{\frac12}\overline{\bu}_h^{n+1}\|^2-\|(\rho^n)^{\frac12}\overline{\bu}_h^n\|^2+|\Delta t|^2\|(\rho^n)^{\frac12}\left[\overline{\bu}_h\right]_t^{n+1}\|^2\right) \cl
&\quad - \frac12\int_\Gamma\frac{\rho^{n+1}-\rho^n}{\Delta t}|\overline{\bu}_h^{n+1}|^2\,ds
- \frac12 \int_\Gamma |\theta^{n+1}|^2\divG (c_h^{n+1}\overline{\bu}_h^{n})|\overline{\bu}_h^{n+1}|^2\, ds \cl
&\quad +  a(\eta^{n+1};\bu^{n+1}_h,\bu^{n+1}_h) + s(p_h^{n+1},p_h^{n+1}) \cl
&\quad  =  -(\sigma_\gamma c^{n+1}_h\nablaG  \mu^{n+1}_h , \bu^{n+1}_h)
+  M \left((\nabla_\Gamma(\theta^{n+1}\overline{\bu}_h^{n+1}))\gradG \mu^{n+1}_h, \theta^{n+1}\bu^{n+1}_h\right). \label{aux334_1}
\end{align}
Denote by $\mathcal{P}_h\,:\,H^1(\Gamma)\to V^k_h$ an $H^1$-type projection into $V^k_h$ given by the $a_\mu$ bilinear form:
\[
a_\mu(\mathcal{P}_h(u),  v_h)=  M \left(\gradG u,  \gradG v_h\right)\quad\forall\, v_h\in V_h^k.
\]
By standard analysis of the TraceFEM for the Laplace-Beltrami problem, e.g. \cite{reusken2015analysis}, we   have
\begin{equation}\label{approx}
  \|u-\mathcal{P}_h(u)\|\lesssim h  \|\gradG u\|.
\end{equation}
Since  $\rho$ is a convex function of $c$ and $\frac{d \rho}{dc}\ge0$, we have
\begin{equation}
\rho^{n+1}-\rho^{n}\le \frac{d\rho}{d c}\Big|_{c^{n+1}_h}(c^{n+1}_h-c^{n}_h)= |\theta^{n+1}|^2(c^{n+1}_h-c^{n}_h). \label{convex_rho}
\end{equation}
Using \eqref{convex_rho} and  eq.~\eqref{eq:CH_FE1} with $v_h=\mathcal{P}_h\left(|\theta^{n+1}\overline{\bu}_h^{n+1}|^2\right)$,
for the terms on the second line of \eqref{aux334_1}
we obtain
\begin{align}
&\int_\Gamma\frac{\rho^{n+1}-\rho^n}{\Delta t}|\overline{\bu}_h^{n+1}|^2
+ |\theta^{n+1}|^2\divG (c_h^{n+1}\overline{\bu}_h^{n})|\overline{\bu}_h^{n+1}|^2\, ds \cl
& \quad \le \int_\Gamma \left(\frac{c^{n+1}-c^n}{\Delta t} + \divG (c_h^{n+1}\overline{\bu}_h^{n})\right)|\theta^{n+1}\overline{\bu}_h^{n+1}|^2\, ds\cl
& \quad = -\int_\Gamma M \gradG \mu^{n+1}_h \cdot \gradG |\theta^{n+1}\overline{\bu}_h^{n+1}|^2\,ds + I(e_h), \label{541}
\end{align}
with
\[
I(e_h)=\int_\Gamma \left(\frac{c^{n+1}_h-c^n_h}{\Delta t} + \divG (c_h^{n+1}\overline{\bu}_h^{n})\right)e_h\, ds,\quad e_h\coloneqq|\theta^{n+1}\overline{\bu}_h^{n+1}|^2-\mathcal{P}_h\left(|\theta^{n+1}\overline{\bu}_h^{n+1}|^2\right).
\]

Next, we use \eqref{541} to simplify \eqref{aux334_1} as follows: 
\begin{multline}
\frac{1}{2\Delta t}\left(\|(\rho^{n+1})^{\frac12}\overline{\bu}_h^{n+1}\|^2-\|(\rho^n)^{\frac12}\overline{\bu}_h^n\|^2+|\Delta t|^2\|(\rho^n)^{\frac12}\left[\overline{\bu}_h\right]_t^{n+1}\|^2\right)
 \\
 +  a(\eta^{n+1};\bu^{n+1}_h,\bu^{n+1}_h) + s(p_h^{n+1},p_h^{n+1}) \le    -(\sigma_\gamma c^{n+1}_h\nablaG  \mu^{n+1}_h , \bu^{n+1}_h) + I(e_h). \label{aux545}
\end{multline}
After adding \eqref{aux573} multiplied by $\sigma_\gamma$ to \eqref{aux545} and dropping some non-negative terms on the
left-hand side, we arrive at
\begin{align}
& \frac{1}{2\Delta t}\left(\|(\rho^{n+1})^{\frac12}\overline{\bu}_h^{n+1}\|^2-\|(\rho^{n})^{\frac12}\overline{\bu}_h^n\|^2 + |\Delta t|^2 \|(\rho^{n})^{\frac12}\left[\overline{\bu}_h\right]_t^{n+1}\|^2\right. \cl
& \quad \left. + \sigma_\gamma (a(c^{n+1}_h,c^{n+1}_h) -a(c^{n}_h,c^{n}_h))+\frac{2 \sigma_\gamma}{\epsilon}\int_\Gamma(f_0(c^{n+1}_h)-f_0(c^n_h)) \right) ds + s_h(p_h^{n+1},p_h^{n+1})\cl
& \quad + a(\eta^{n+1};\bu^{n+1}_h,\bu^{n+1}_h) +\sigma_\gamma a_\mu(\mu^{n+1}_h,\mu^{n+1}_h) \cl
& \quad \le - \Delta t(\sigma_\gamma c^{n+1}_h \nablaG \mu^{n+1}_h, \left[\bu_h\right]_t^{n+1})+I(e_h).  \label{aux568}
\end{align}

It remains to estimate the terms on the right-hand side of \eqref{aux568}.
We handle the first term by invoking the result of Lemma~\ref{Lemma1}  as follows:
\begin{align}
\Delta t|(\sigma_\gamma &c^{n+1}_h \nablaG \mu^{n+1}_h, \left[\bu_h\right]_t^{n+1})| \le
 \sigma_\gamma \Delta t |((c^{n+1}_h-c_0)\nablaG \mu^{n+1}_h, \left[\bu_h\right]_t^{n+1})|+  \sigma_\gamma \Delta t |(c_0\nablaG \mu^{n+1}_h, \left[\bu_h\right]_t^{n+1})|\cl
               &\le  \sigma_\gamma \Delta t\|c^{n+1}_h-c_0\|_{L^\infty(\Gamma)}\|\nablaG \mu^{n+1}_h\| \|\left[\overline{\bu}_h\right]_t^{n+1}\|+  \sigma_\gamma \Delta t  |c_0| \|\nablaG \mu^{n+1}_h\| \|\left[\overline{\bu}_h\right]_t^{n+1}\| \cl
               &\le \frac{C}{\rho_2^{\frac12}\epsilon^{\frac12}} |\ln h|^{\frac12} \sigma_\gamma \Delta t |a_c(c^{n+1}_h, c^{n+1}_h)|^{\frac12} \|\nablaG \mu^{n+1}_h\| \|(\rho^{n})^{\frac12}\left[\overline{\bu}_h\right]_t^{n+1}\| \cl
               &\qquad+   \frac{ \sigma_\gamma \Delta t}{\rho_2^{\frac12}}  |c_0| \|\nablaG \mu^{n+1}_h\| \|(\rho^{n})^{\frac12}\left[\overline{\bu}_h\right]_t^{n+1}\|. \label{aux583}
\end{align}
Thanks to the  a priori bound $a_c(c^{n+1}_h, c^{n+1}_h)\le C$ from \eqref{aux613b},
estimate \eqref{aux583} yields
\begin{equation}\label{aux643}
\Delta t|( \sigma_\gamma c^{n+1}_h \nablaG \mu^{n+1}_h, \left[\bu_h\right]_t^{n+1})| \le
  C  \sigma_\gamma \Delta t ({\epsilon^{-\frac12}} |\ln h|^{\frac12} +1) \|(\rho^{n})^{\frac12}\left[\overline{\bu}_h\right]_t^{n+1}\|  |a_\mu(\mu^{n+1}_h, \mu^{n+1}_h)|^{\frac12}.
\end{equation}
With the help of $\Delta t\le C|\ln h|^{-1} \epsilon$ for sufficiently small $C$ and $\epsilon\lesssim 1$ we obtain
\begin{align}
\Delta t|(\sigma_\gamma c^{n+1}_h \nablaG \mu^{n+1}_h, \left[\bu_h\right]_t^{n+1})|
&\le
 \frac{\Delta t}{2}\|(\rho^{n})^{\frac12}\left[\overline{\bu}_h\right]_t^{n+1}\|^2+ C \sigma_\gamma \Delta t({\epsilon^{-\frac12}} |\ln h|^{\frac12} +1)^2 a_\mu(\mu^{n+1}_h, \mu^{n+1}_h)\cl
 &\le
 \frac{\Delta t}{2}\|(\rho^{n})^{\frac12}\left[\overline{\bu}_h\right]_t^{n+1}\|^2+ \frac12 \sigma_\gamma a_\mu(\mu^{n+1}_h, \mu^{n+1}_h).
 \label{aux649}
 \end{align}

Now, we proceed with the terms in $I_h$. For the first terms in $I_h$,
the Cauchy-Schwarz inequality and estimate \eqref{approx} for the $L^2(\Gamma)$ norm of $e_h$ gives:
\begin{equation}\label{aux978}
\int_\Gamma \left(\frac{c_h^{n+1}-c_h^n}{\Delta t}\right)e_h\, ds\le \|\left[c_h\right]_t^{n+1}\|
\|e_h\|
\lesssim h \|\left[c_h\right]_t^{n+1}\|
\|\nablaG(|\theta^{n+1}\overline{\bu}_h^{n+1}|^2)\|
\end{equation}
To estimate the right-most factor, we need the inequalities
\[
\|\overline{\bu}_h^{n+1}\|^2_{L^\infty(\Gamma)}\lesssim |\ln h|\|\overline{\bu}_h^{n+1}\|^2_{H^1(\Gamma)}+ h^{-1}\|\bn\cdot\nabla|\overline{\bu}_h^{n+1}\|^2_{L^2(\Omega_h^\Gamma)} \lesssim |\ln h|a(\eta^{n+1},\bu_h^{n+1},\bu_h^{n+1}),
\]
which follow by applying Lemma~\ref{Lemma1} componentwise and then using the Korn inequality \eqref{Korn} and the fact that $\eta^{n+1}$ is uniformly bounded from below (see the definition in \eqref{eq:cut-off}).
Thanks to $0\le\theta^2\lesssim1$ and $0\le\frac{d\theta^2}{dc}=\frac{d^2\rho}{dc^2}\lesssim1$,  we have
\begin{align}
\|\nablaG(|\theta^{n+1}\overline{\bu}_h^{n+1}|^2) \| &\lesssim
\||\overline{\bu}_h^{n+1}|^2\nablaG c^{n+1}_h\|+ \|\nablaG |\overline{\bu}_h^{n+1}|^2\| \cl
&\lesssim
\|\overline{\bu}_h^{n+1}\|^2_{L^\infty(\Gamma)}\|\nablaG c^{n+1}_h\|+ \|\overline{\bu}_h^{n+1}\|_{L^\infty(\Gamma)}\|\nablaG \overline{\bu}_h^{n+1}\| \cl
&\lesssim
|\ln h| a(\eta^{n+1},\bu_h^{n+1},\bu_h^{n+1}) \|\nablaG c^{n+1}_h\|+ |\ln h|^{\frac12} a(\eta^{n+1},\bu_h^{n+1},\bu_h^{n+1}) \cl
&\lesssim
(|\ln h|\epsilon^{-\frac12} + |\ln h|^{\frac12})a(\eta^{n+1},\bu_h^{n+1},\bu_h^{n+1}) \cl
&\lesssim
|\ln h|\epsilon^{-\frac12}a(\eta^{n+1},\bu_h^{n+1},\bu_h^{n+1}). \el
\end{align}
Using this together with the a priori bound for $\|\left[c_h\right]_t^{n+1}\|$ from  \eqref{aux613b} in \eqref{aux978}
and using the assumption  $h \le c |\ln h|^{-1} \Delta t$ for sufficiently small $c$, we have
\begin{equation}\label{aux998}
\int_\Gamma \left(\frac{c_h^{n+1}-c_h^n}{\Delta t}\right)e_h\, ds\le c\, h|\ln h| \Delta t^{-1} a(\eta^{n+1},\bu_h^{n+1},\bu_h^{n+1})
\le \frac14 a(\eta^{n+1},\bu_h^{n+1},\bu_h^{n+1}).
\end{equation}

Next, we treat the second term in $I(e_h)$. Using Cauchy--Schwarz inequality and  estimating $\|e_h\|$ as above we have
\begin{equation}\label{aux1003}
\int_\Gamma \divG (c_h^{n+1}\overline{\bu}_h^{n})e_h\, ds \lesssim
h|\ln h|\epsilon^{-\frac12}\| \divG (c_h^{n+1}\overline{\bu}_h^{n})\| a(\eta^{n+1},\bu_h^{n+1},\bu_h^{n+1}).
\end{equation}
By the triangle inequality
\[
\| \divG (c_h^{n+1}\overline{\bu}_h^{n})\|\le
\| c_h^{n+1} \divG (\overline{\bu}_h^{n})\| + \|\bu^{n}_h \cdot \nablaG c_h^{n+1} \|.
\]
Each term on the right hand side can be treated individually by invoking Lemma~\ref{Lemma1}, a priori bound from  \eqref{aux613b},
and induction assumption:
\[
\begin{split}
\| c_h^{n+1} \divG (\overline{\bu}_h^{n})\|&\le \| c_h^{n+1}\|_{L^\infty(\Gamma)} \| \divG (\overline{\bu}_h^{n})\|\\
&\lesssim |\ln h|^{\frac12}\| c_h^{n+1}\|_{H^1(\Gamma)} \| \divG (\overline{\bu}_h^{n})\|\\
&\lesssim |\ln h|^{\frac12}\epsilon^{-\frac12}\| \divG (\overline{\bu}_h^{n}) \|
\lesssim |\ln h|^{\frac12}\epsilon^{-\frac12}|a(\eta^{n},\bu_h^{n},\bu_h^{n})|^{\frac12}
\lesssim |\ln h|^{\frac12}\epsilon^{-\frac12}|\Delta t|^{-\frac12}.
\end{split}
\]
and
\[
\begin{split}
\|\bu^{n}_h\cdot\nablaG c_h^{n+1} \|&\le \|\overline{\bu}_h^{n}\|_{L^\infty(\Gamma)} \|\nablaG c_h^{n+1} \|\\
&\lesssim \epsilon^{-\frac12}\|\overline{\bu}_h^{n}\|_{L^\infty(\Gamma)} \lesssim
 \epsilon^{-\frac12}|\ln h|^{\frac12}\|\overline{\bu}_h^{n}\|_{H^1(\Gamma)}\\
  &\lesssim
 \epsilon^{-\frac12}|\ln h|^{\frac12}|a(\eta^{n},\bu_h^{n},\bu_h^{n})|^{\frac12}
 \lesssim |\ln h|^{\frac12}\epsilon^{-\frac12}|\Delta t|^{-\frac12}.
\end{split}
\]
Using these estimates back into \eqref{aux1003} and with
the assumption  $h \le c |\ln h|^{-\frac32} \epsilon |\Delta t|^{-\frac12}$ for sufficiently small $c$,
we arrive at the estimate
\begin{equation}\label{aux1030}
|I(e_h)|
\le \frac12 a(\eta^{n+1};\bu_h^{n+1},\bu_h^{n+1}).
\end{equation}
Finally, we substitute \eqref{aux649} and \eqref{aux1030}  in \eqref{aux568} to obtain
\begin{multline*}
\frac{1}{2\Delta t}\left(\|(\rho^{n+1})^{\frac12}\overline{\bu}_h^{n+1}\|^2  + \sigma_\gamma a(c^{n+1}_h,c^{n+1}_h) +\frac{2 \sigma_\gamma}{\epsilon}\int_\Gamma f_0(c^{n+1}_h) ds \right) + \frac12 a(\eta^{n+1};\bu^{n+1}_h,\bu^{n+1}_h)\\ + \frac12 \sigma_\gamma a_\mu(\mu^{n+1}_h,\mu^{n+1}_h)+ s_h(p_h^{n+1},p_h^{n+1})
 \le    \frac{1}{2\Delta t}\left( \|(\rho^{n})^{\frac12}\overline{\bu}_h^n\|^2
  + \sigma_\gamma a(c^{n}_h,c^{n}_h)+\frac{2 \sigma_\gamma}{\epsilon}\int_\Gamma f_0(c^n_h) ds\right). 
\end{multline*}
This completes the induction step.
\end{proof}

\begin{remark}\label{Rem2}\rm From the proof we see that the  assumption  $h \le c |\ln h|^{-1} \min\{\Delta t, |\ln h|^{-\frac12} \epsilon |\Delta t|^{\frac12}\}$ results from the fact that at  the discrete level we cannot test the transport equation for the order parameter with $v_h=|\bu_h^{n+1}|^2$, and so we have to project  $|\bu_h^{n+1}|^2$ (or $|\overline{\bu}_h^{n+1}|^2$ for surfaces) in the finite element space $V^k_h$ and handle the resulting inconsistency. If the finite element velocity space is such that $|\bv_h|^2\in V_h^k$ for $\bv_h\in\bV_h$, then the upper bound on $h$ is not needed in the analysis (in the surface case we still would need to handle $|\bu_h^{n+1}|^2-|\overline{\bu}_h^{n+1}|^2$, but this is possible due to the control over $\|\bn\cdot\bu_h^{n+1}\|^2$  that we have  thanks to penalty term in the TraceFEM formulation). An example, when  $|\bv_h|^2\in V_h^k$ holds, is $\bf P_1$--$P_1$ stabilized velocity--pressure element combined with $P_2$ element for the order parameter and chemical potential.
\end{remark}



\section{Numerical results}\label{sec:num_res}

After checking the convergence orders of the method described in
Sec.~\ref{sec:method}, we present a series of numerical tests to study well-known
two-phase fluid flows (the Kelvin--Helmholtz and Rayleigh--Taylor instabilities) on surfaces.
Thanks to our method, we can investigate the effect of line tension on such instabilities.
In addition, for the Rayleigh--Taylor instability we assess the effect of viscosity
and surface shape.

For all the tests, we choose $\bf P_2$--$P_1$ finite elements for fluid velocity and
pressure and $P_1$--$P_1$ finite elements for surface fraction and chemical potential.

\subsection{Convergence test}

We proceed with checking the spatial accuracy of the finite element
method described in Sec.~\ref{sec:method} with a benchmark test.
The aim is to validate our implementation of the method.
For this purpose, we consider the two-phase fluid system on the unit sphere
centered at the origin. The surface is characterized as the zero level set of function $\phi(\bx) = \|\bx\|_2 -1$
and is embedded in an outer cubic domain $\Omega=[-5/3,5/3]^3$.
We choose Van der Waals ``tanh'' exact solution for the surface fraction and solenoidal Killing vector field for velocity:
\[
 c^*(t, \bx) = \frac{1}{2} \left(1 + \tanh\frac{z \cos(\pi t) - y \sin(\pi t)} {2 \sqrt{2} \epsilon} \right), \quad \bu^*(t, \bx) = \pi\,(0, -z ,y)^T.
\]
Nonzero CH equation forcing term is computed from~\eqref{gracke-3}. We set $\epsilon = 0.05$. Fig.~\ref{fig:rotation} (leftmost panel) displays $c(0, \bx)$.
For this test, we consider fluids with matching densities and viscosities:
$\rho_1 = \rho_2 = 1$ and $\eta_1 =\eta_2 = 1$. In addition, we set
$M=0.05$ and $\sigma_\gamma = 0$. The time interval of interest is $t\in [0,1]$, during which
the initial configuration $c_0$ is rotated by $180^\circ$. See Fig.~\ref{fig:rotation}.
Notice that by setting $\sigma_\gamma = 0$
the NSCH system one-way coupled: phase-separation is affected by the fluid flow,
but not vice versa.

\begin{figure}[h]
\centering
\begin{overpic}[width=0.15\textwidth, grid=false]{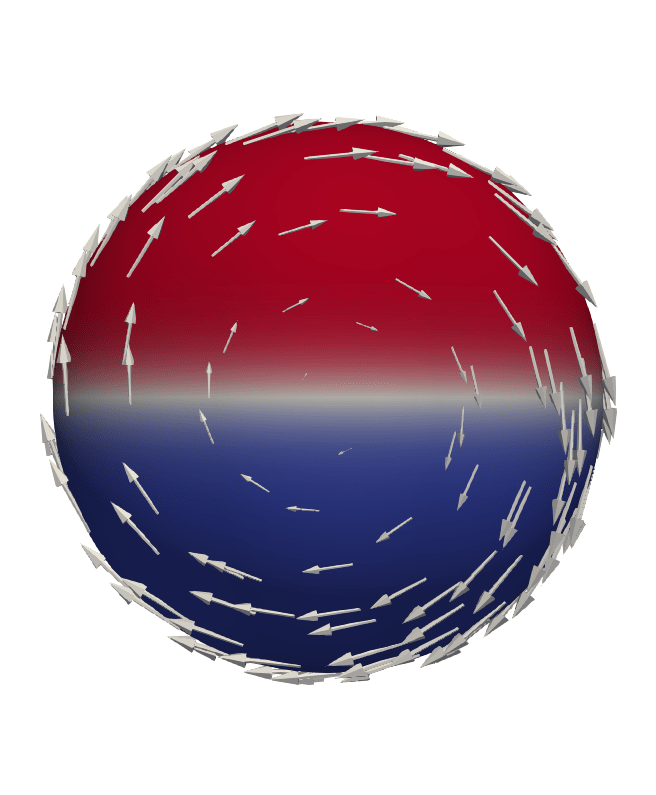}
\put(30,90){$t = 0$}
\end{overpic}
\begin{overpic}[width=0.15\textwidth, grid=false]{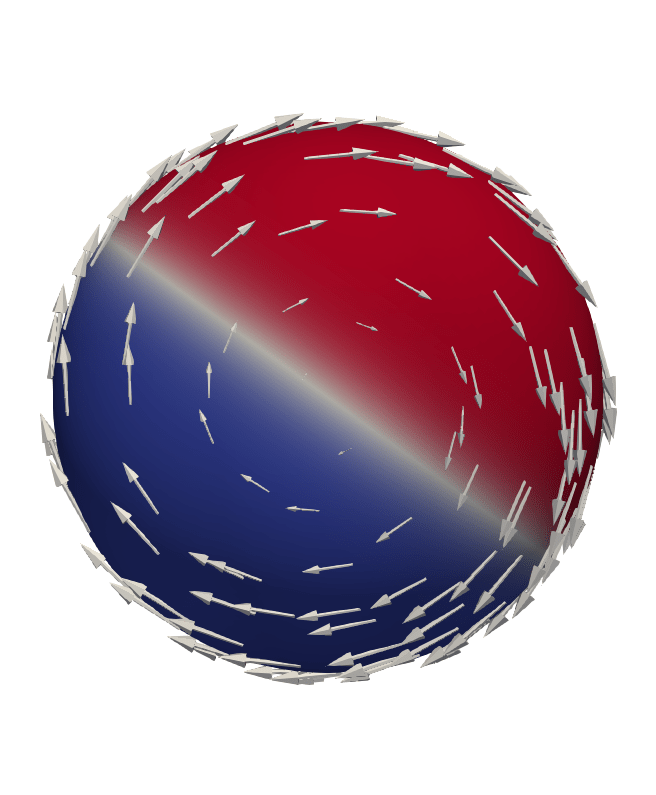}
\put(25,90){$t = 0.2$}
\end{overpic}
\begin{overpic}[width=0.15\textwidth, grid=false]{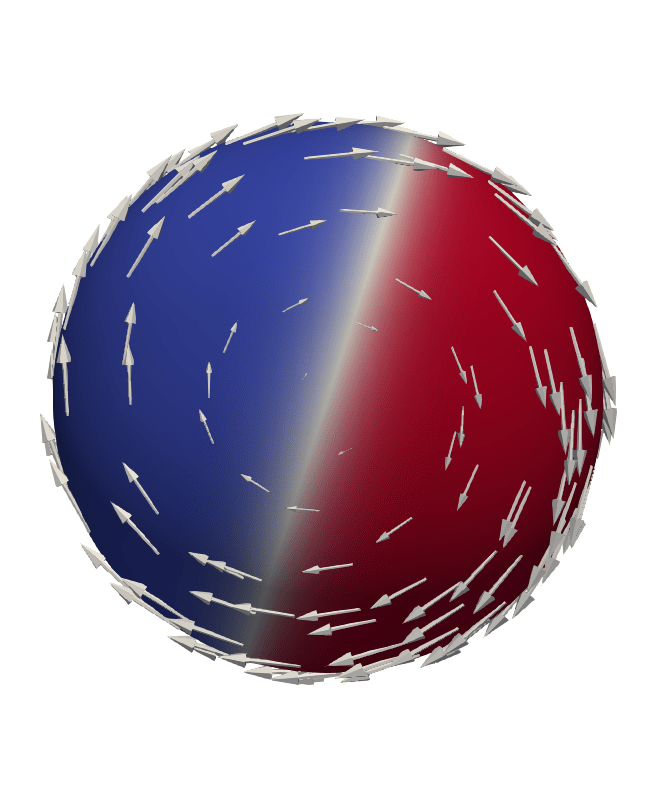}
\put(25,90){$t = 0.6$}
\end{overpic}
\begin{overpic}[width=0.15\textwidth, grid=false]{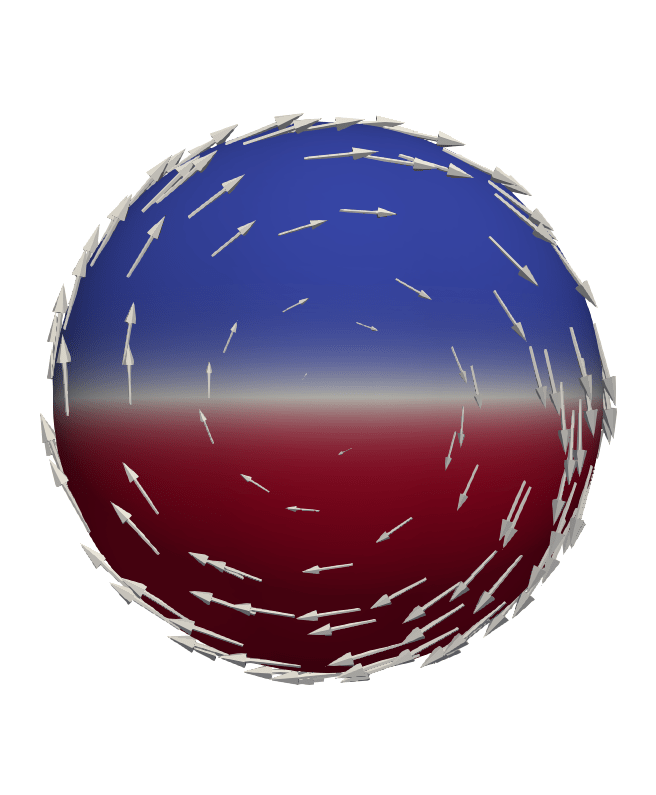}
\put(30,90){$t = 1$}
\end{overpic}
	\caption{Evolution of the surface fraction $c$ over time computed with mesh $\ell = 5$.}
	\label{fig:rotation}
\end{figure}

%
%

The initial triangulation $\T_{h_\ell}$ of $\Omega$ consists of eight sub-cubes,
where each of the sub-cubes is further subdivided into six tetrahedra.
Further, the mesh is refined towards the surface, and $\ell\in\Bbb{N}$ denotes the level of refinement, with the associated mesh size $h_\ell= \frac{10/3}{2^{\ell+1}}$. 
For the purpose of numerical integration, we apply several ``virtual'' levels of refinements for the tetrahedra cut by the mesh and integrate our bilinear forms over a piecewise planar approximation of $\Gamma$ on this virtual grid. This allowed us to apply a standard quadrature rule and reduce the geometric error in our convergence test.   The time step was refined together with the mesh size according to $\Delta t = 1/(25 \cdot 2^{\ell - 2})$. For this test, we used
BDF2 for the time discretization of the Cahn--Hilliard problem at Step 1, instead
of BDF1 as reported in \eqref{eq:CH_FE1}. This is why the time step is refined linearly.

Table \ref{tab:rotation} reports the $H_1$ and $L_2$ errors for the velocity and the $L_2$ error
for the order parameter at the end of the time interval, i.e.~$t = 1$, for levels $\ell = 3, 4, 5$. For each
mesh, Table \ref{tab:rotation} gives the number of sublevels (virtual levels) used for more accurate numerical integration.
We observe slightly better than expected convergence rates, which might be the effect of the interplay between interpolation and geometric error reduction.

\begin{table}[h]
	\centering
	\begin{tabular}{|c|c|c|c|c|c|c|c|}
		\hline
	$\ell$	& sublevels  & $||\bu^* - \bu_h||_{H^1(\Gamma)}$  & rate  &  $||\bu^* - \bu_h||_{L^2(\Gamma)}$  & rate  & $||c^* - c_h||_{L^2(\Gamma)}$  & rate   \\ \hline
	3	& 1  & 0.081485 &  &0.010026  &  &  0.311232 &  \\ \hline
	4	&  2 &  0.016800 & 2.42 & 0.000619 & 8.20 & 0.081597 & 1.91 \\ \hline
	5	&   4 & 0.003905  & 2.15 & 0.000046  & 6.73 &  0.015086& 2.70 \\ \hline
	\end{tabular}
	\caption{$H_1$ and $L_2$ errors for the velocity and the $L_2$ error
for surface fraction at $t = 1$ for levels $\ell = 3, 4, 5$ and rate of convergence.}\label{tab:rotation}
\end{table}

\subsection{The Kelvin--Helmholtz instability}

While the Kelvin--Helmholtz (KH) instability is a classical example of two-phase fluid flow in planar
or volumetric domains, the number of numerical studies on curved surfaces is limited.
The KH instability arises when there is a difference in velocity at the interface between the
two fluids and a perturbation is added to the interface. This perturbation eventually makes the interface curl up
and generates a vortex strip.
Here, we will simulate the KH instability on a sphere and investigate the effect
of varying line tension.

To design this experiment, we follow what done in \cite{Lederer2020,Jankuhn2020,olshanskii2020recycling}.
The initial velocity field is given by the counter-rotating upper and
lower hemispheres with speed approximately equal 1 closer to equator and
vanishing at the poles. The velocity field has a sharp transition layer along the equator, where
the perturbation is added. See, e.g., \cite{olshanskii2020recycling} for details on the perturbation.
The initial surface fraction is given by
\[
c_0 = \frac{1}{2} \left(1 + \tanh\frac{z} {2 \sqrt{2} \epsilon}\right),
\]
where $\epsilon=0.01$. Also for this test, we consider fluids with matching densities and viscosities:
$\rho_1 = \rho_2 = 1$ and $\eta_1 =\eta_2 = 10^{-5}$. In addition, we set
$M=0.01$. We consider time interval $[0,10]$.

We select mesh level $\ell = 6$ (see mesh description for the previous test).
We choose $\Delta t = 1/640$. Fig.~\ref{khch} and \ref{khwh} show the evolution of surface fraction and vorticity
for three different values of line tension: $\sigma_\gamma = 0, 0.01, 0.1$. The evolution of
both quantities does not vary significantly when going from $\sigma_\gamma = 0$ to $\sigma_\gamma = 0.01$,
although some differences can be noticed from $t =4.5$ on. Changing to $\sigma_\gamma = 0.1$ produces
more evident differences, starting already from $t =1.375$. When $\sigma_\gamma \neq 0$, the NS part of the
CHNS system is two-way coupled to the CH part. So, the differences are significant both for surface fraction and vorticity.

\vskip .3cm

\begin{figure}[htb]
\centering
\hskip .7cm
\begin{overpic}[width=0.11\textwidth,grid=false,tics=10]{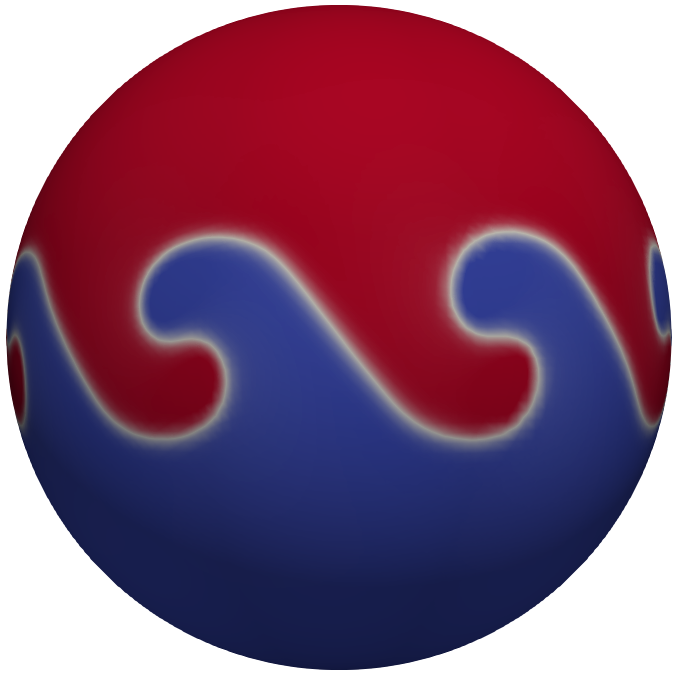}
\put(-95,45){$\sigma_\gamma = 0$}
\put(30,110){$t = 1$}
\end{overpic}~
\begin{overpic}[width=0.11\textwidth,grid=false,tics=10]{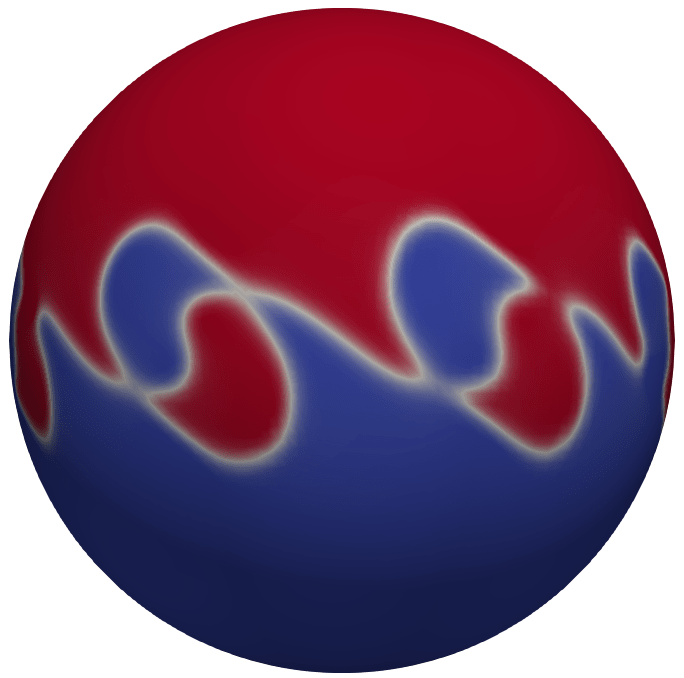}
\put(0,110){$t = 1.375$}
\end{overpic}~
\begin{overpic}[width=0.11\textwidth,grid=false,tics=10]{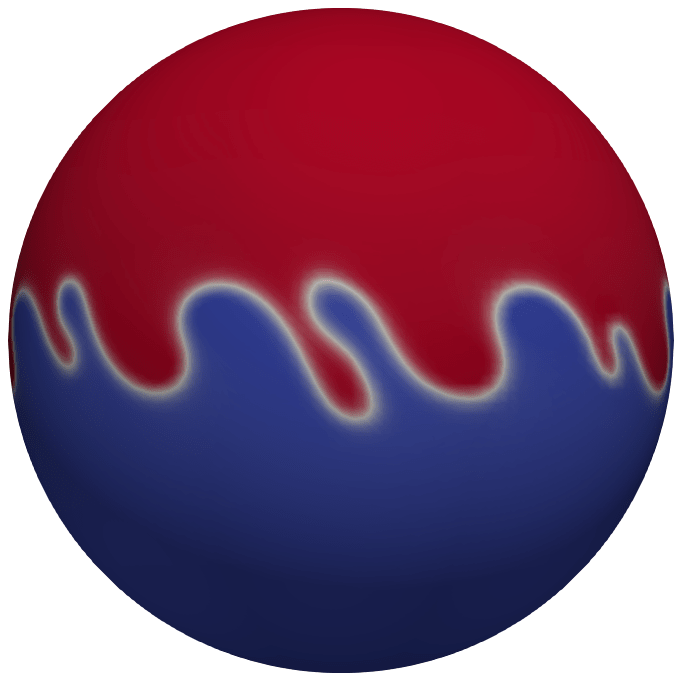}
\put(30,110){$t = 2$}
\end{overpic}~
\begin{overpic}[width=0.11\textwidth,grid=false,tics=10]{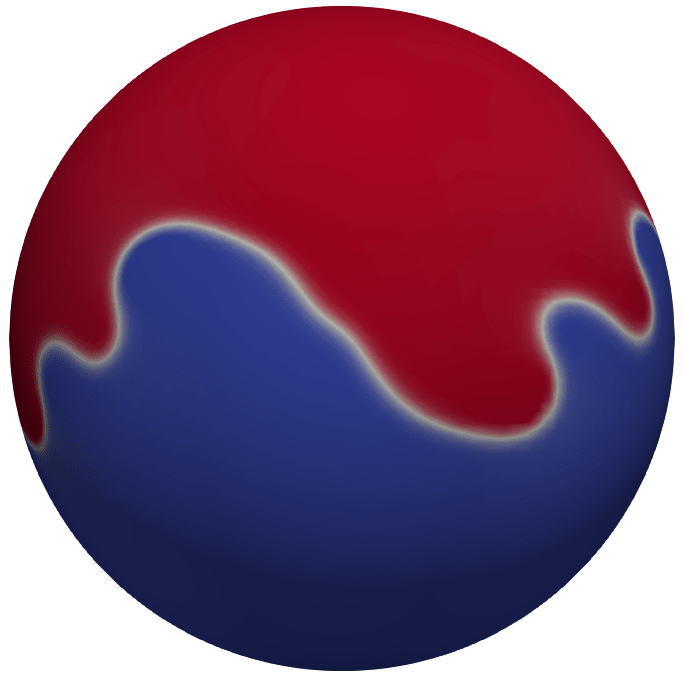}
\put(15,110){$t = 4.5$}
\end{overpic}~
\begin{overpic}[width=0.11\textwidth,grid=false,tics=10]{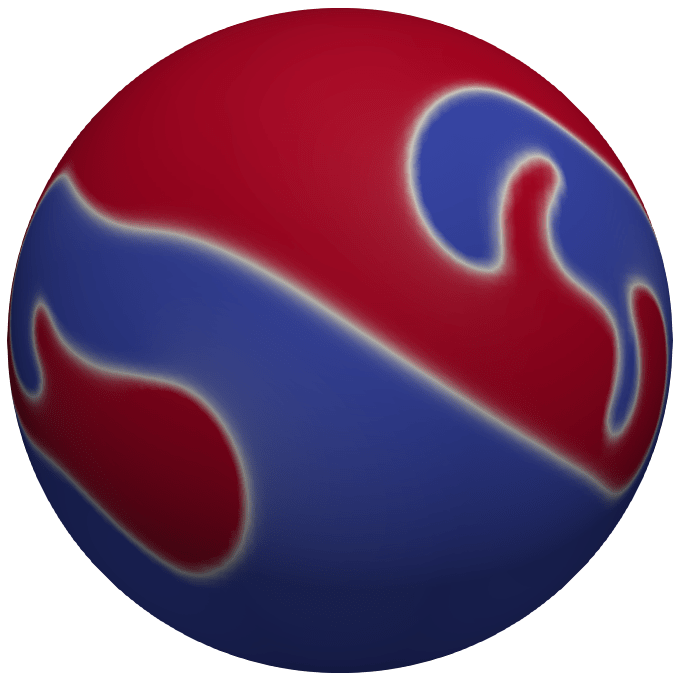}
\put(15,110){$t = 5.5$}
\end{overpic}~
\begin{overpic}[width=0.11\textwidth,grid=false,tics=10]{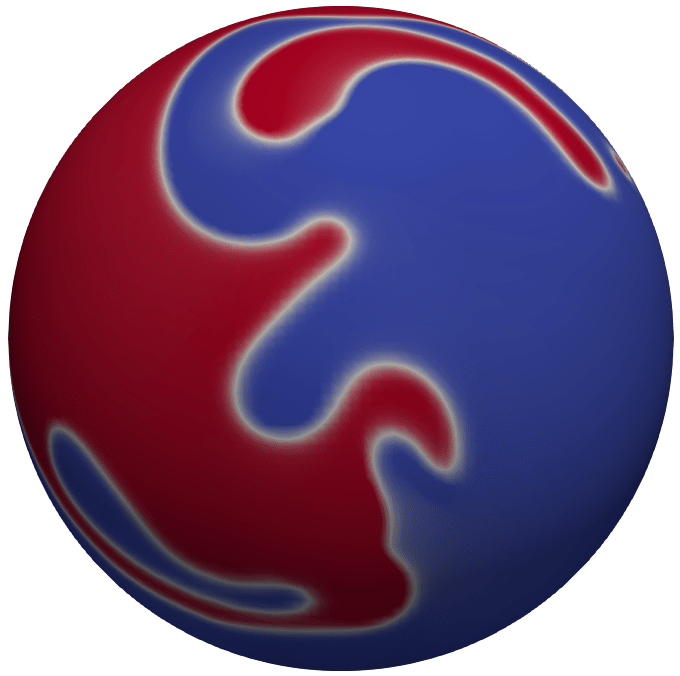}
\put(15,110){$t = 8.5$}
\end{overpic}~
\begin{overpic}[width=0.11\textwidth,grid=false,tics=10]{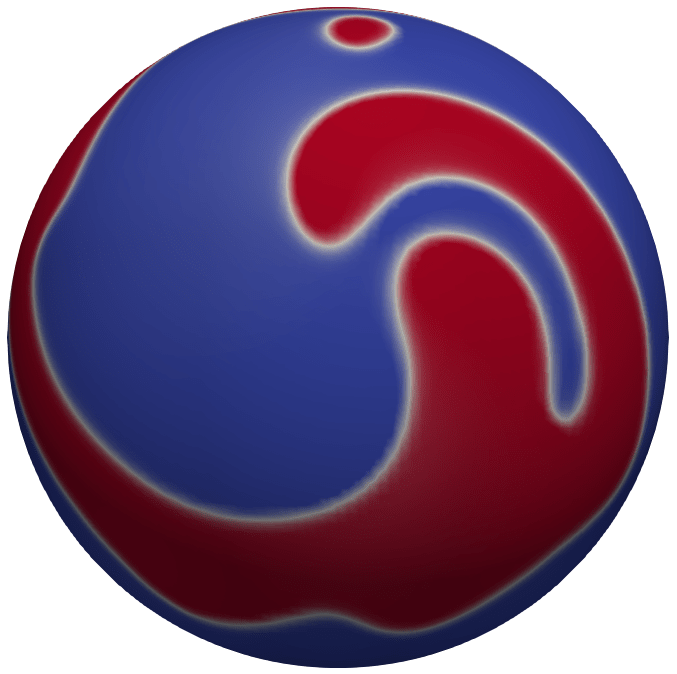}
\put(15,110){$t = 10$}
\end{overpic} \\
\hskip .7cm
\begin{overpic}[width=0.11\textwidth,grid=false,tics=10]{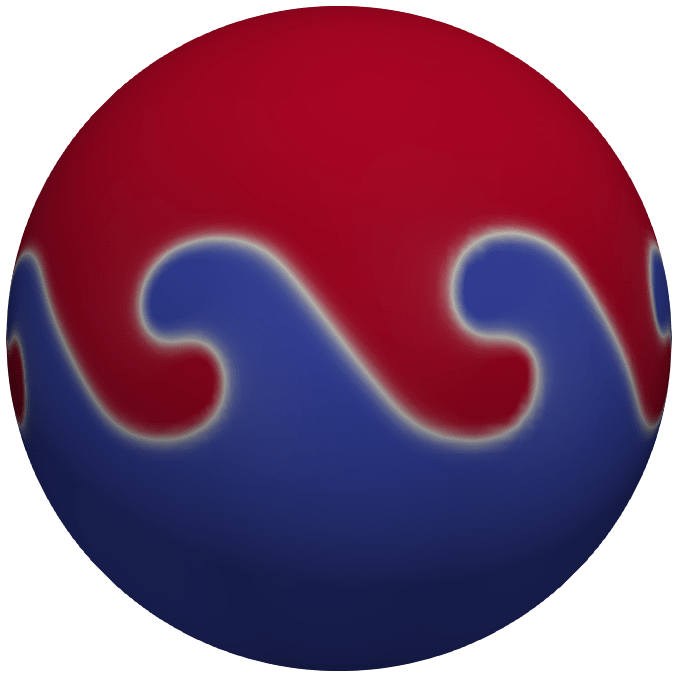}
\put(-95,45){$\sigma_\gamma = 0.01$}
\end{overpic}~
\begin{overpic}[width=0.11\textwidth,grid=false,tics=10]{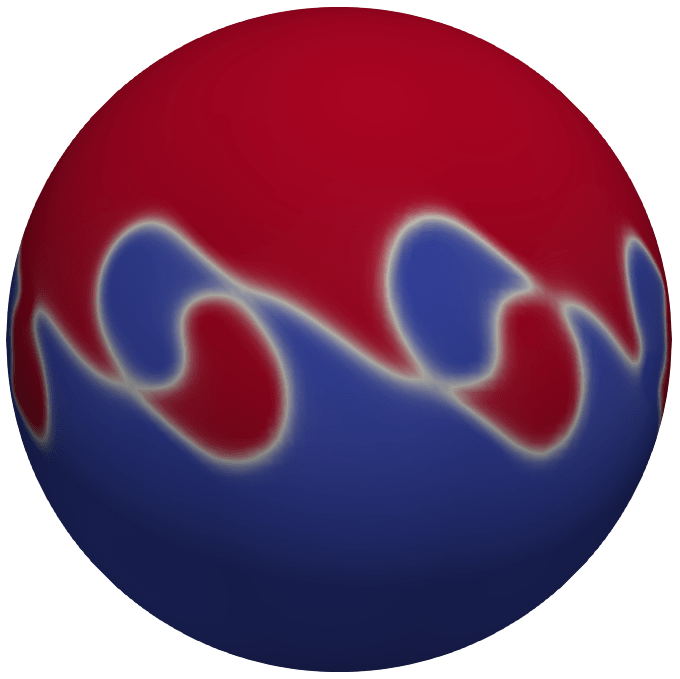}
\end{overpic}~
\begin{overpic}[width=0.11\textwidth,grid=false,tics=10]{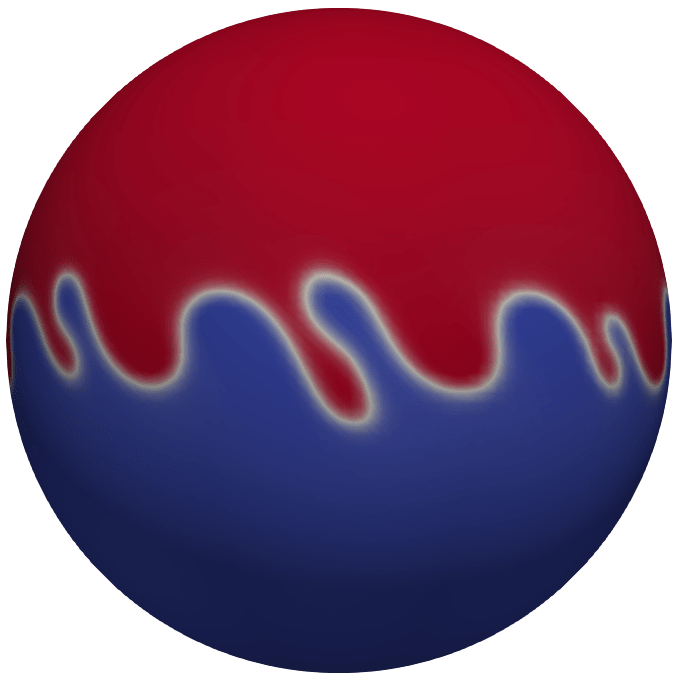}
\end{overpic}~
\begin{overpic}[width=0.11\textwidth,grid=false,tics=10]{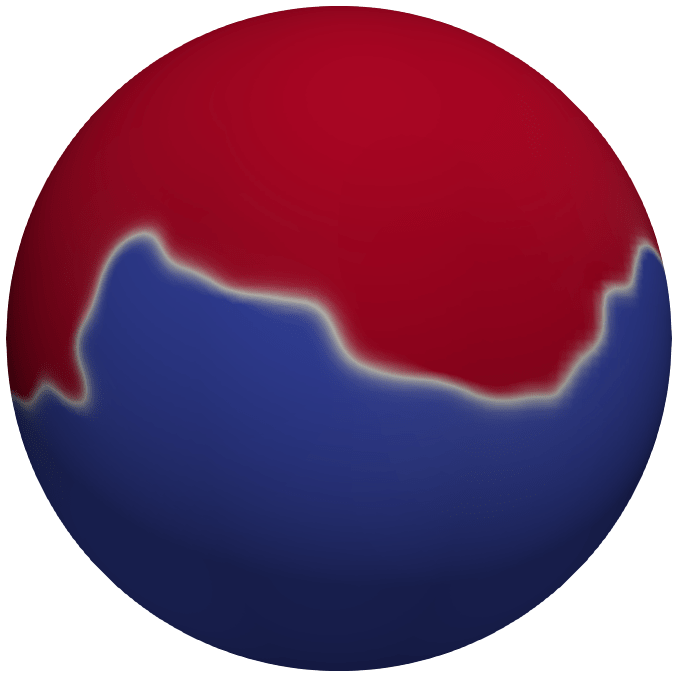}
\end{overpic}~
\begin{overpic}[width=0.11\textwidth,grid=false,tics=10]{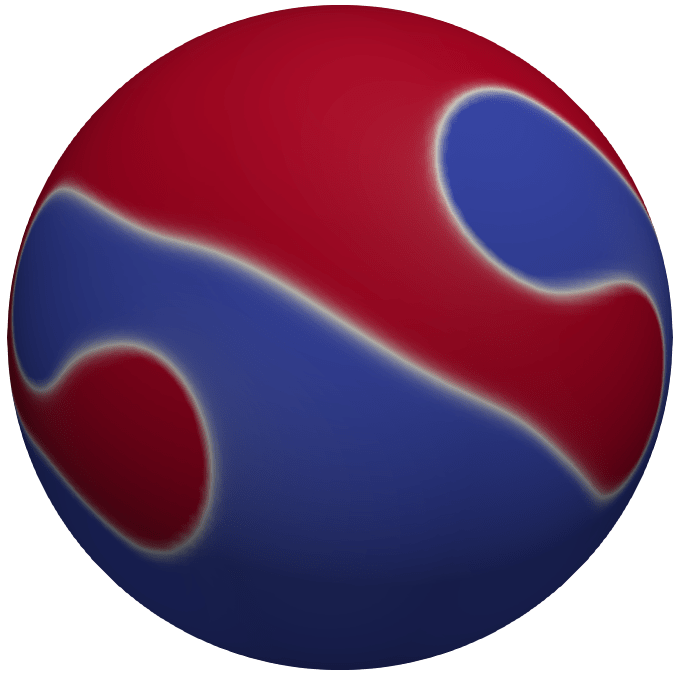}
\end{overpic}~
\begin{overpic}[width=0.11\textwidth,grid=false,tics=10]{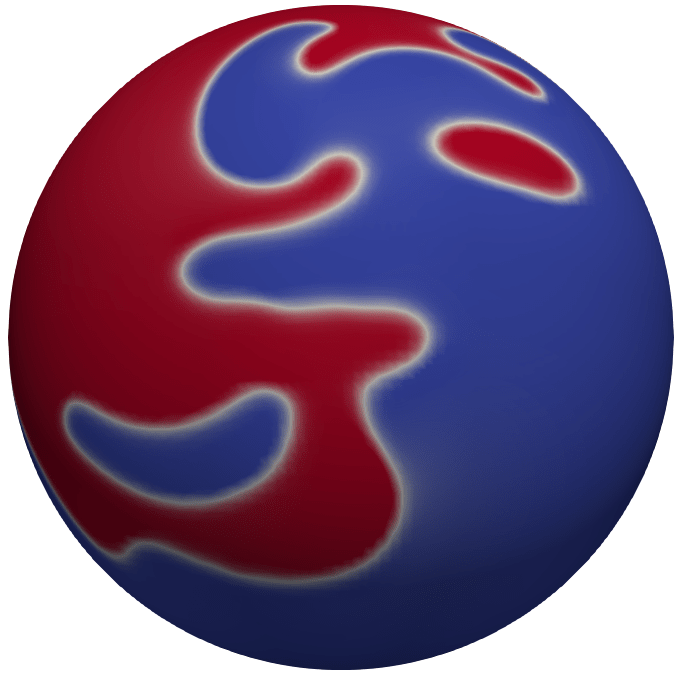}
\end{overpic}~
\begin{overpic}[width=0.11\textwidth,grid=false,tics=10]{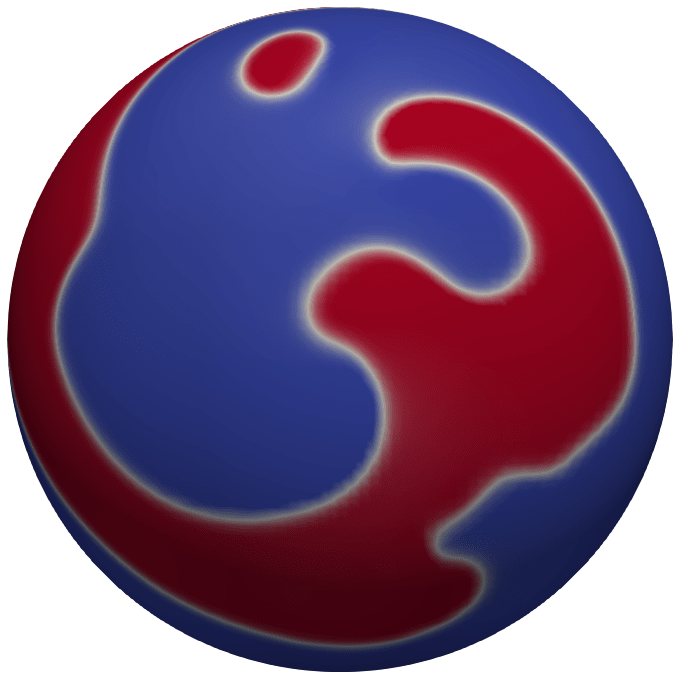}
\end{overpic} \\
\hskip .7cm
\begin{overpic}[width=0.11\textwidth,grid=false,tics=10]{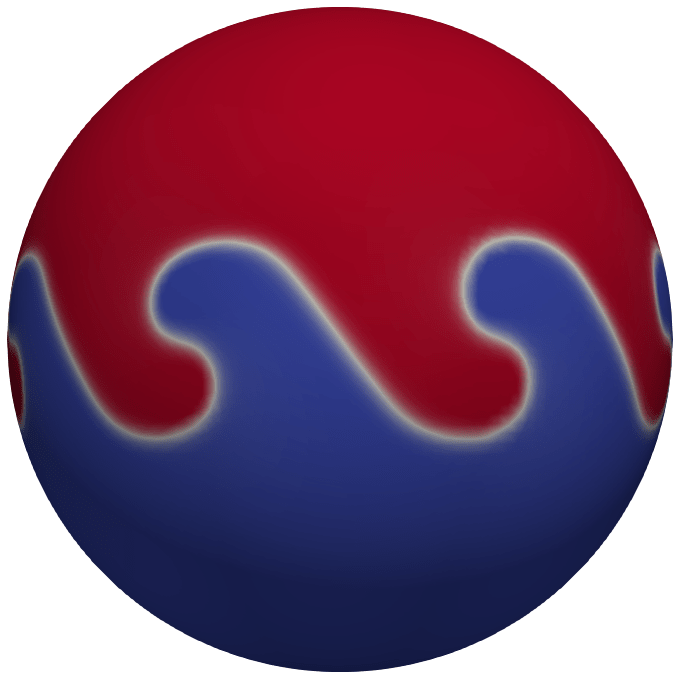}
\put(-95,45){$\sigma_\gamma = 0.1$}
\end{overpic}~
\begin{overpic}[width=0.11\textwidth,grid=false,tics=10]{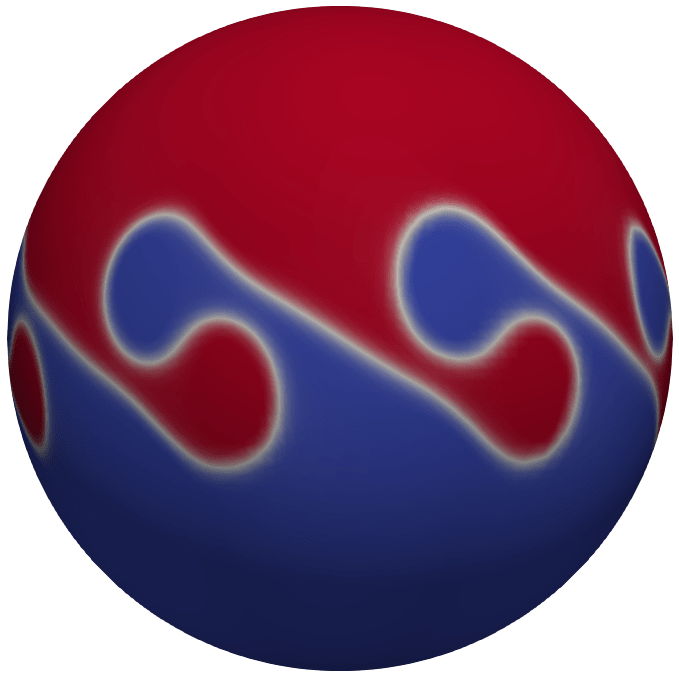}
\end{overpic}~
\begin{overpic}[width=0.11\textwidth,grid=false,tics=10]{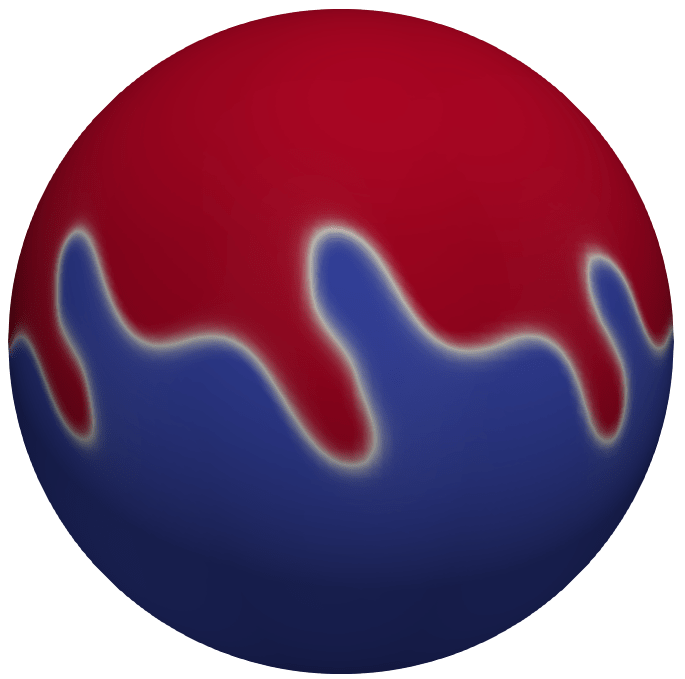}
\end{overpic}~
\begin{overpic}[width=0.11\textwidth,grid=false,tics=10]{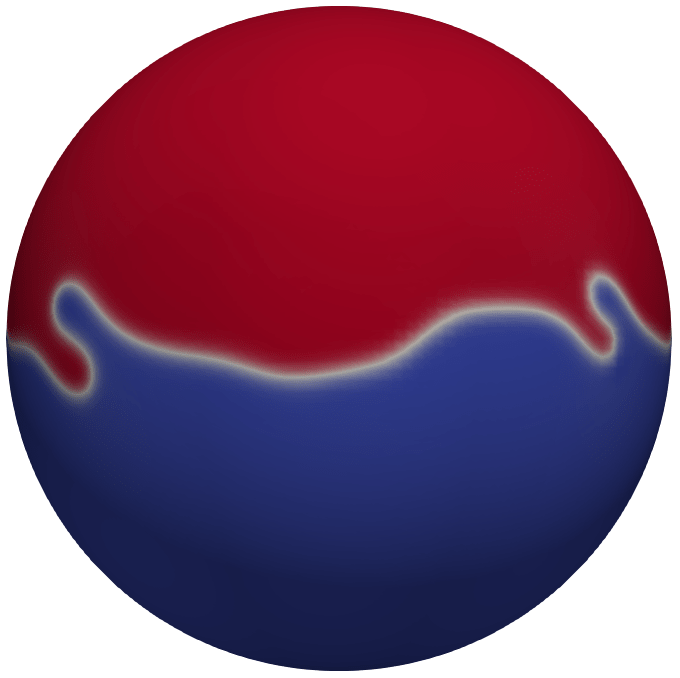}
\end{overpic}~
\begin{overpic}[width=0.11\textwidth,grid=false,tics=10]{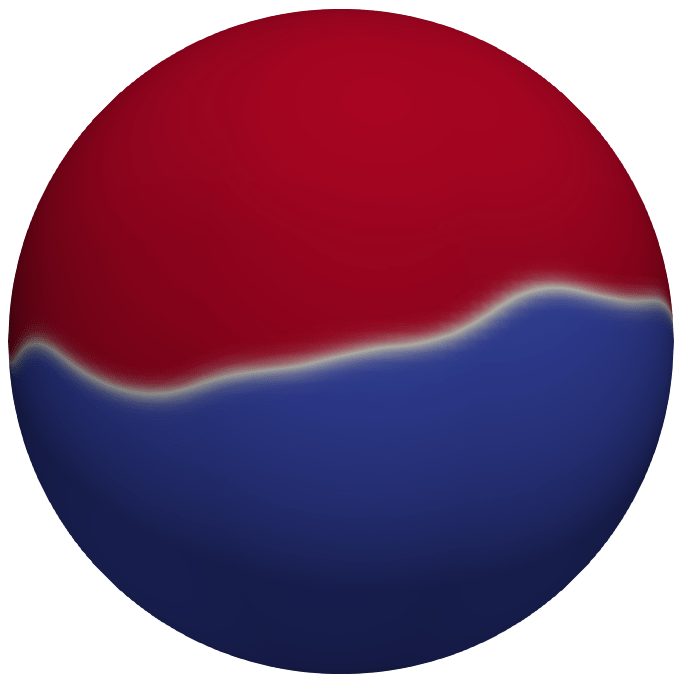}
\end{overpic}~
\begin{overpic}[width=0.11\textwidth,grid=false,tics=10]{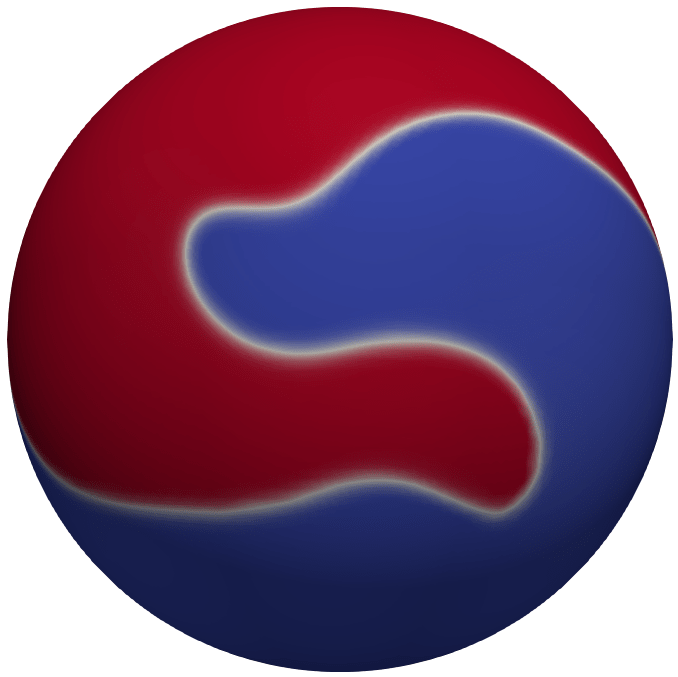}
\end{overpic}~
\begin{overpic}[width=0.11\textwidth,grid=false,tics=10]{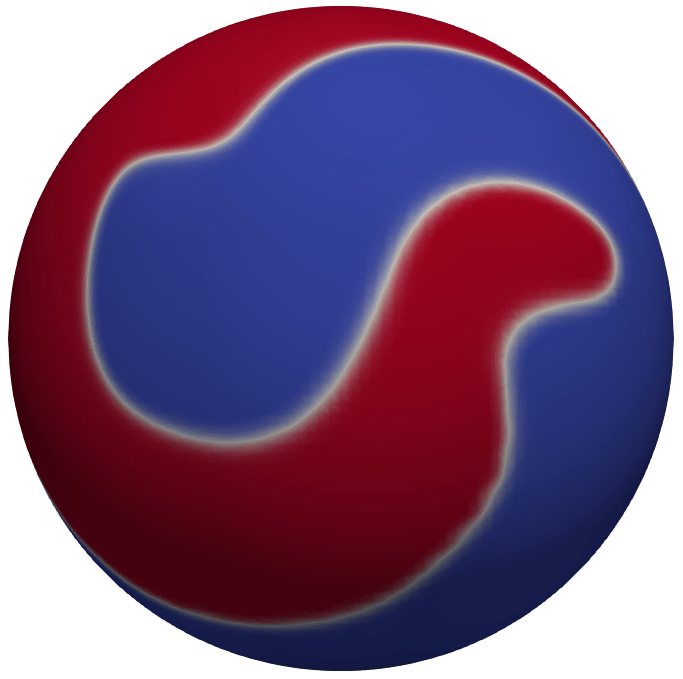}
\end{overpic} \\
\vskip .2cm
\begin{overpic}[width=0.5\textwidth,grid=false,tics=10]{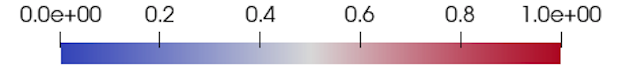}
\end{overpic}
	\caption{KH instability: evolution of order parameter for different values of line tension:
	$\sigma=0$ (top), $\sigma=0.01$ (center), and $\sigma = 0.1$ (bottom).
A full animation can be viewed following the link \href{https://youtu.be/C3_WLO1Wd7Y}{\underline{youtu.be/C$3_{-}$WLO1Wd7Y}}
}
	\label{khch}
\end{figure}


\begin{figure}[htb]
\centering
\hskip .7cm
\begin{overpic}[width=0.11\textwidth,grid=false,tics=10]{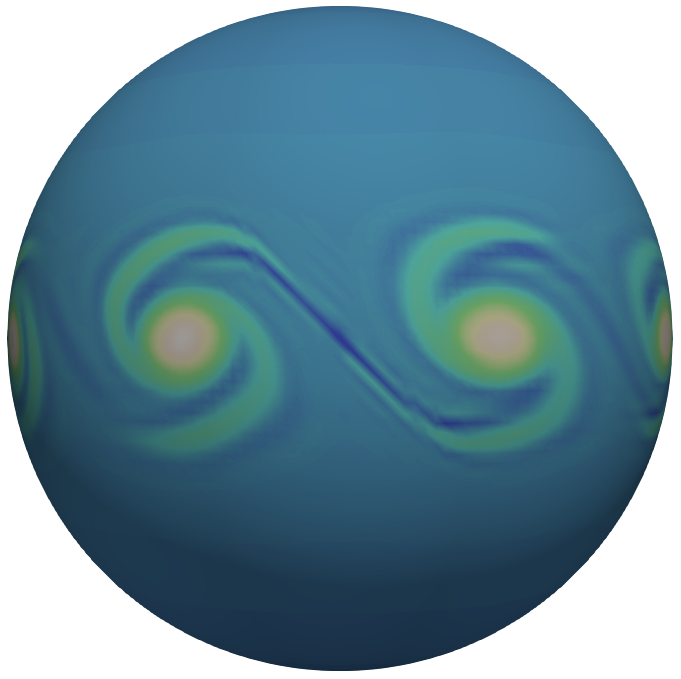}
\put(-95,45){$\sigma_\gamma = 0$}
\put(28,105){$t = 1$}
\end{overpic}~
\begin{overpic}[width=0.11\textwidth,grid=false,tics=10]{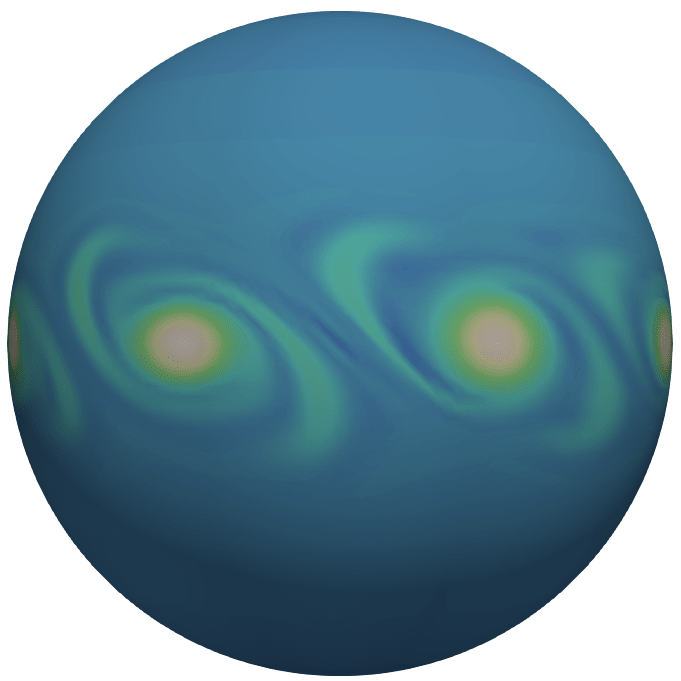}
\put(0,105){$t = 1.375$}
\end{overpic}~
\begin{overpic}[width=0.11\textwidth,grid=false,tics=10]{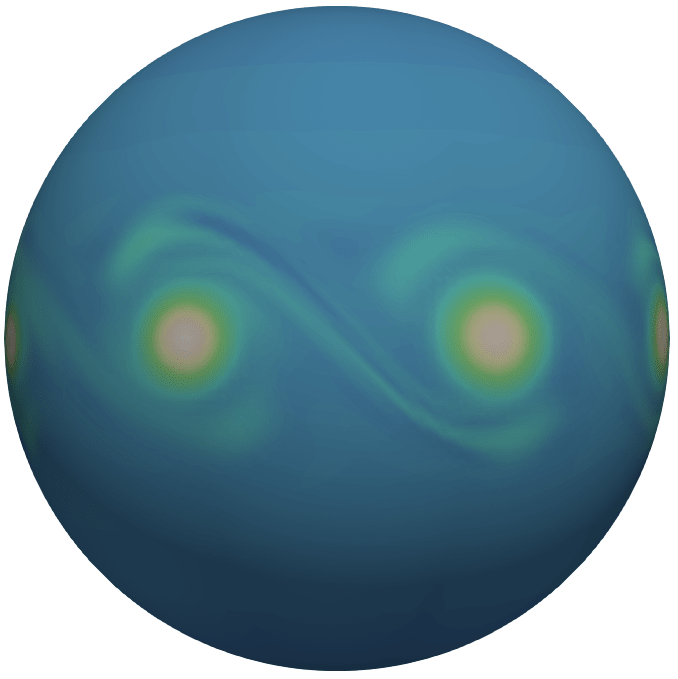}
\put(28,105){$t = 2$}
\end{overpic}~
\begin{overpic}[width=0.11\textwidth,grid=false,tics=10]{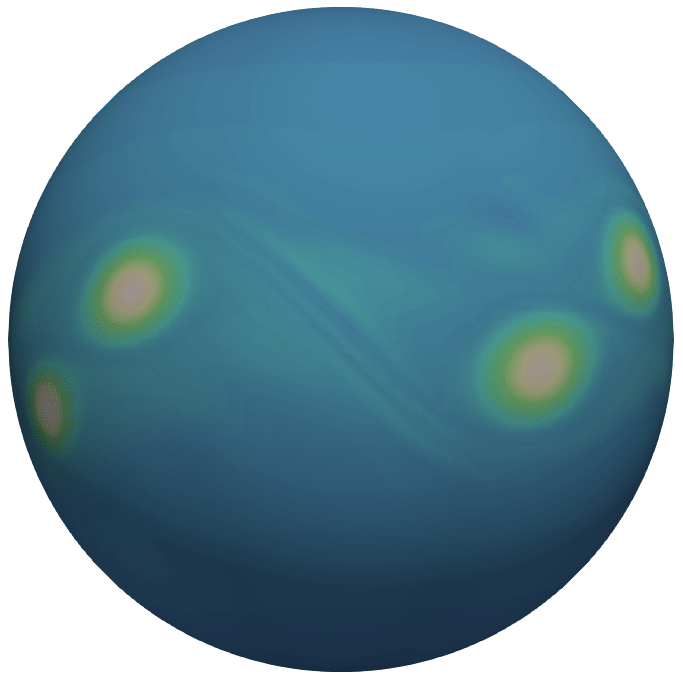}
\put(18,105){$t = 4.5$}
\end{overpic}~
\begin{overpic}[width=0.11\textwidth,grid=false,tics=10]{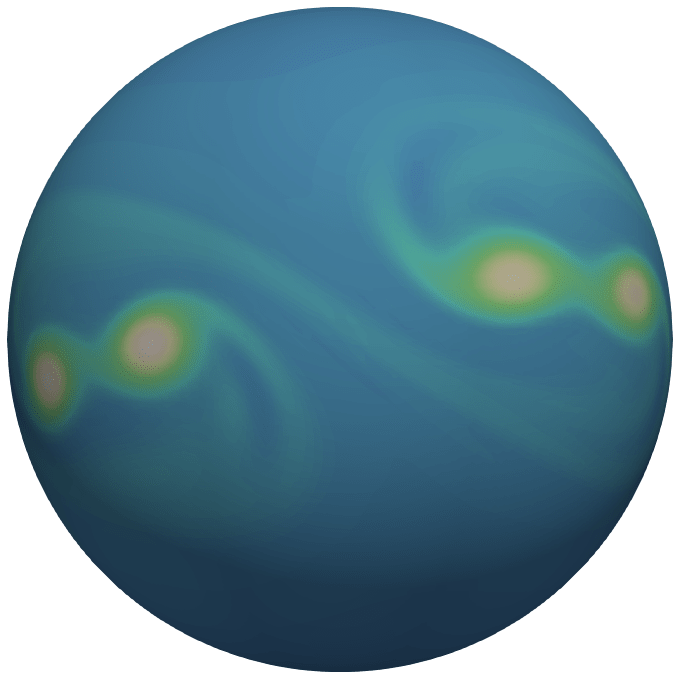}
\put(18,105){$t = 5.5$}
\end{overpic}~
\begin{overpic}[width=0.11\textwidth,grid=false,tics=10]{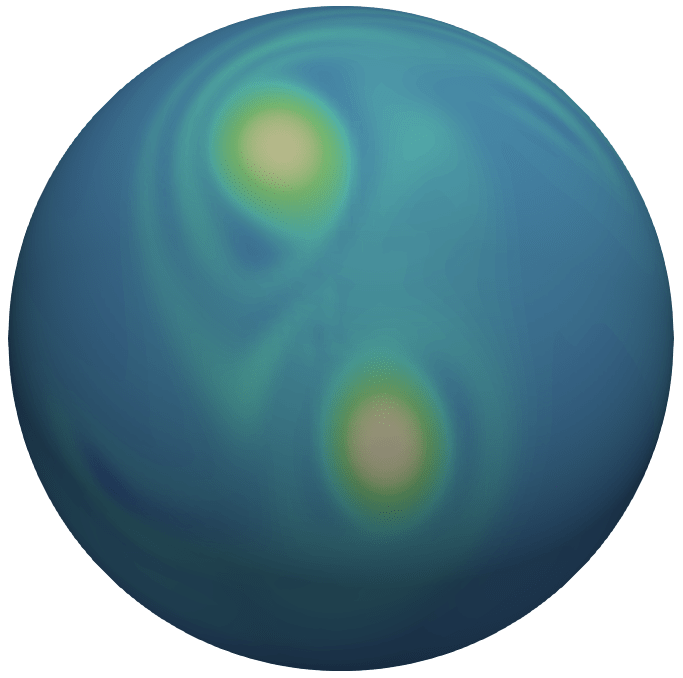}
\put(18,105){$t = 8.5$}
\end{overpic}~
\begin{overpic}[width=0.11\textwidth,grid=false,tics=10]{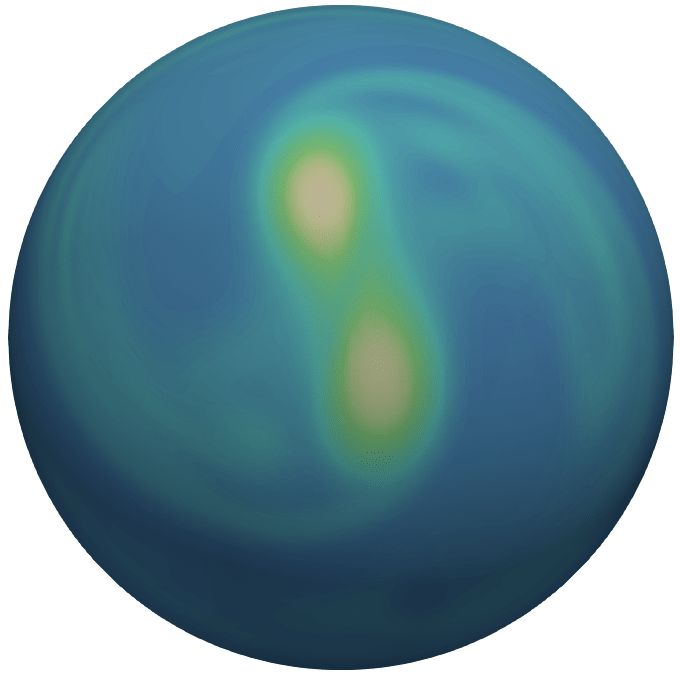}
\put(23,105){$t = 10$}
\end{overpic} \\
\hskip .7cm
\begin{overpic}[width=0.11\textwidth,grid=false,tics=10]{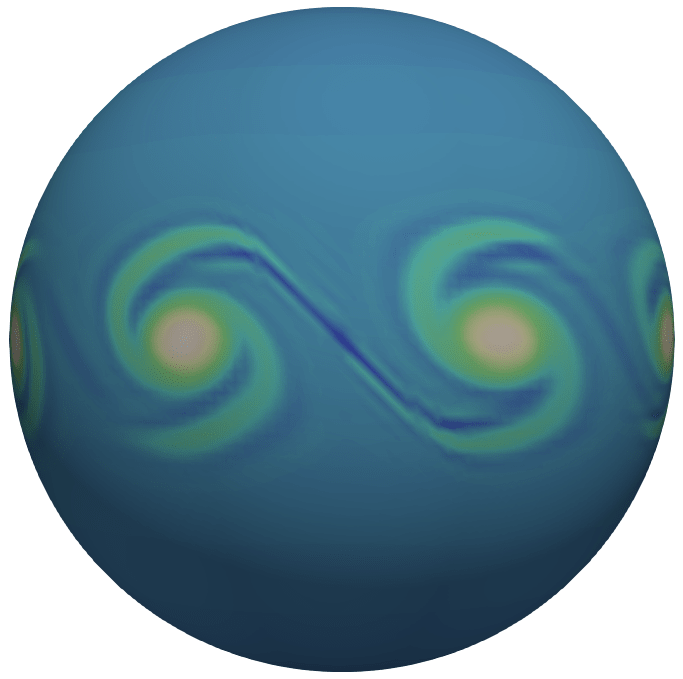}
\put(-95,45){$\sigma_\gamma = 0.01$}
\end{overpic}~
\begin{overpic}[width=0.11\textwidth,grid=false,tics=10]{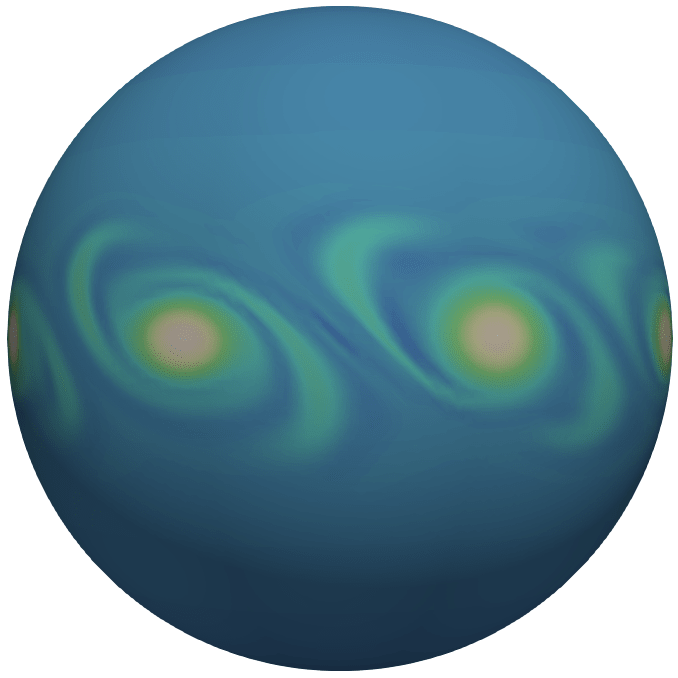}
\end{overpic}~
\begin{overpic}[width=0.11\textwidth,grid=false,tics=10]{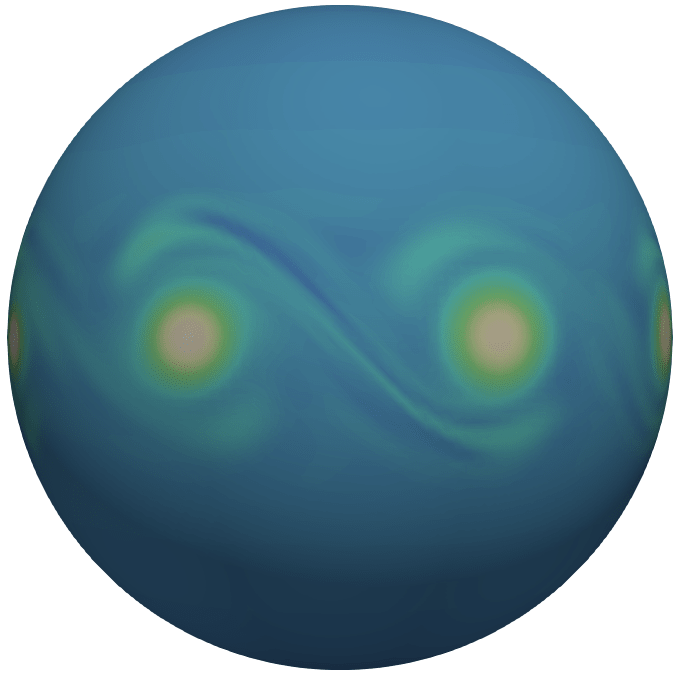}
\end{overpic}~
\begin{overpic}[width=0.11\textwidth,grid=false,tics=10]{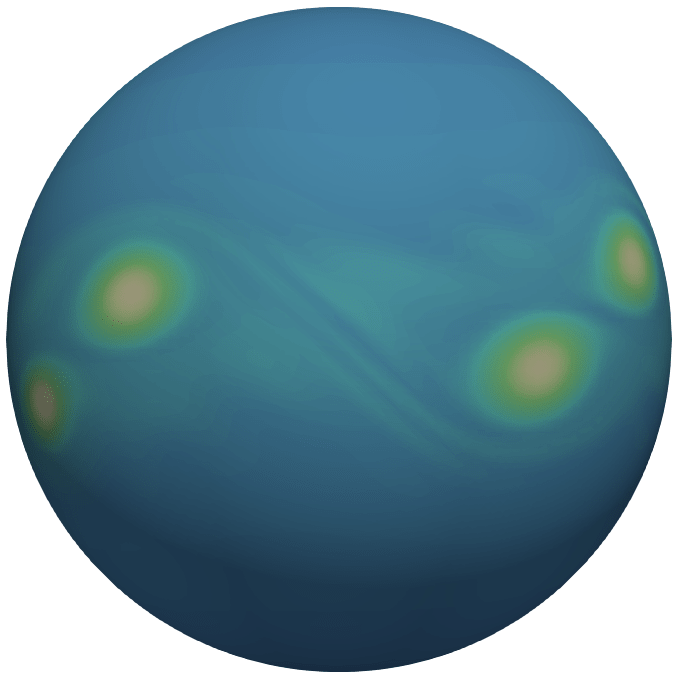}
\end{overpic}~
\begin{overpic}[width=0.11\textwidth,grid=false,tics=10]{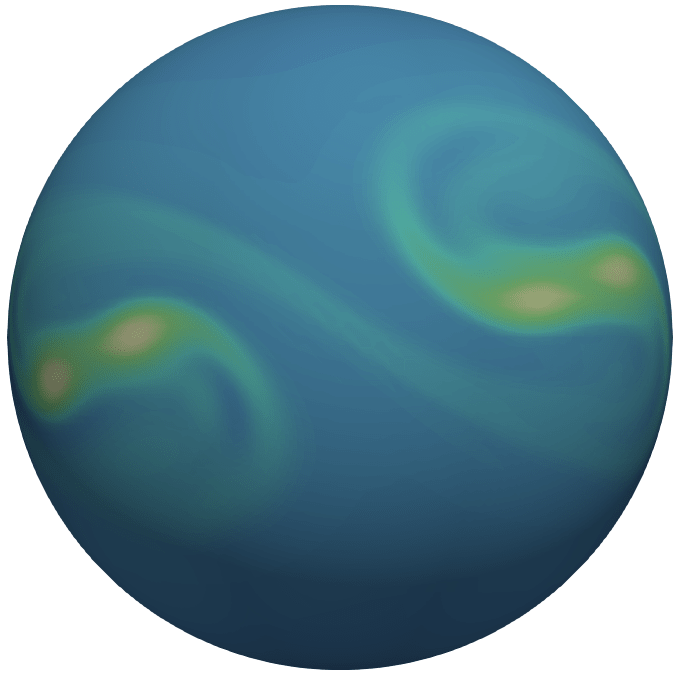}
\end{overpic}~
\begin{overpic}[width=0.11\textwidth,grid=false,tics=10]{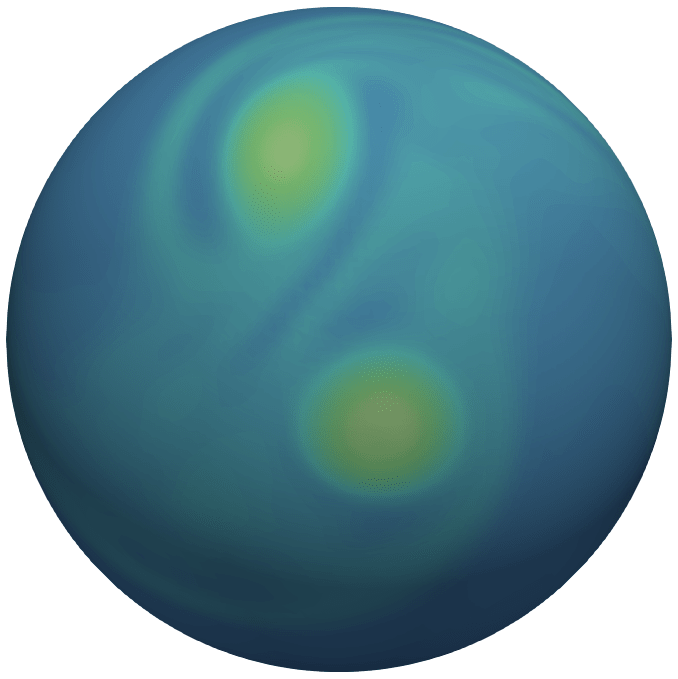}
\end{overpic}~
\begin{overpic}[width=0.11\textwidth,grid=false,tics=10]{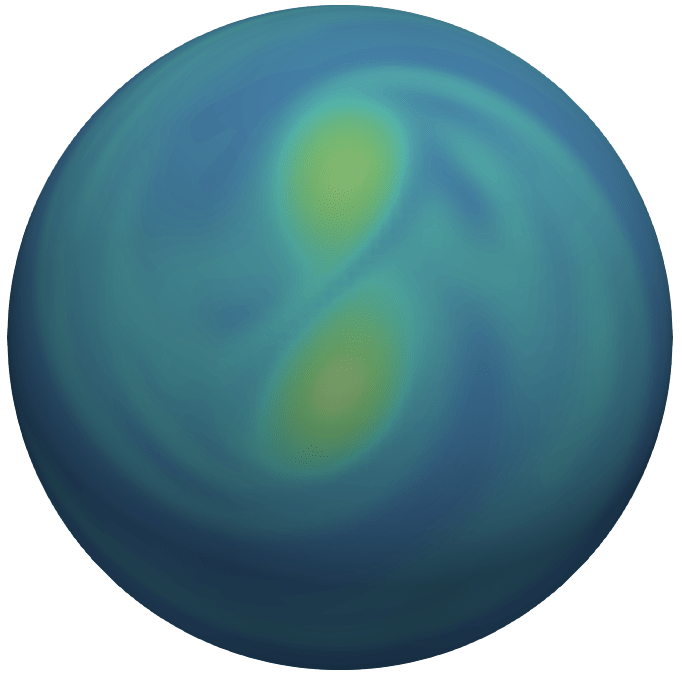}
\end{overpic} \\
\hskip .7cm
\begin{overpic}[width=0.11\textwidth,grid=false,tics=10]{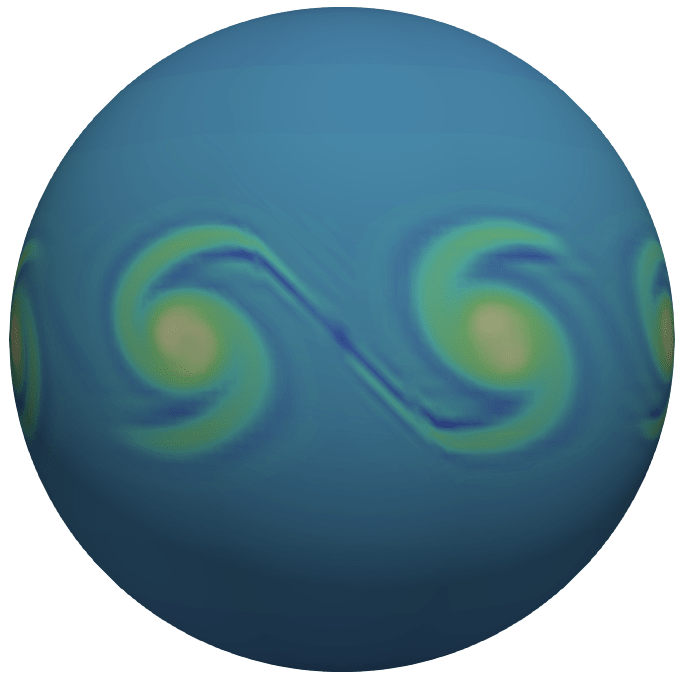}
\put(-95,45){$\sigma_\gamma = 0.1$}
\end{overpic}~
\begin{overpic}[width=0.11\textwidth,grid=false,tics=10]{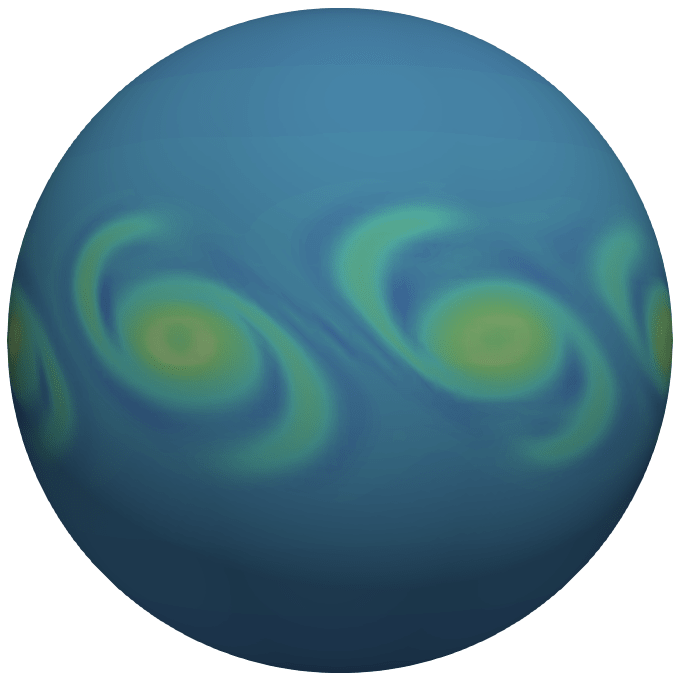}
\end{overpic}~
\begin{overpic}[width=0.11\textwidth,grid=false,tics=10]{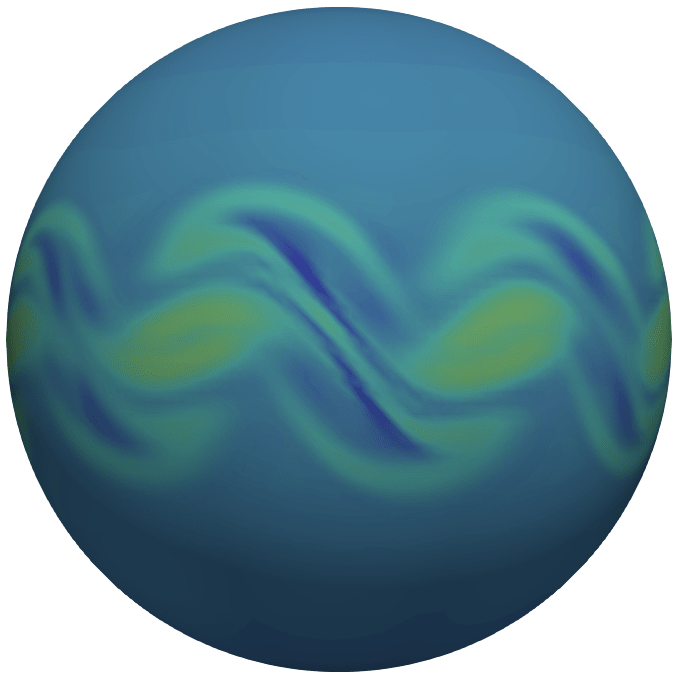}
\end{overpic}~
\begin{overpic}[width=0.11\textwidth,grid=false,tics=10]{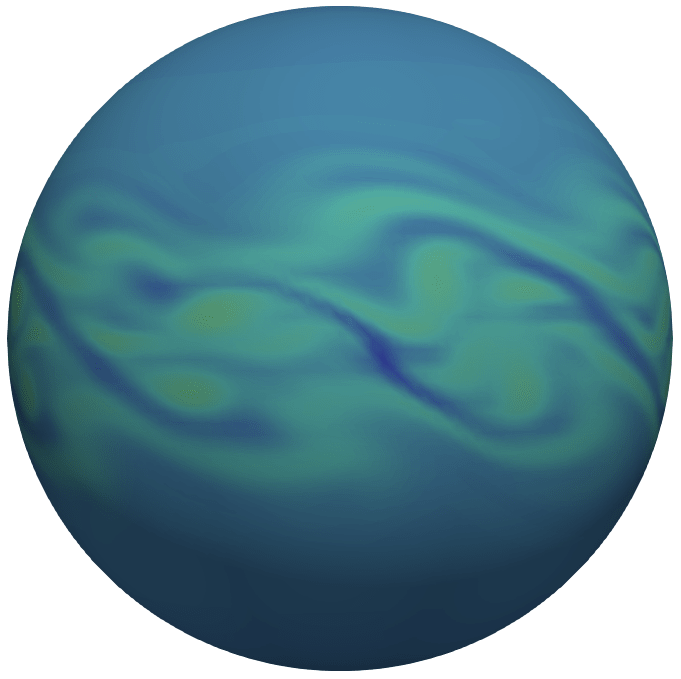}
\end{overpic}~
\begin{overpic}[width=0.11\textwidth,grid=false,tics=10]{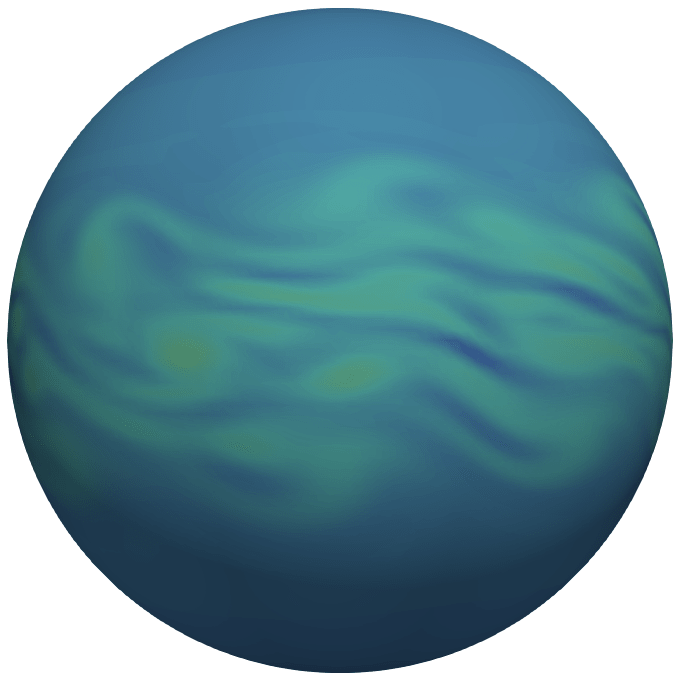}
\end{overpic}~
\begin{overpic}[width=0.11\textwidth,grid=false,tics=10]{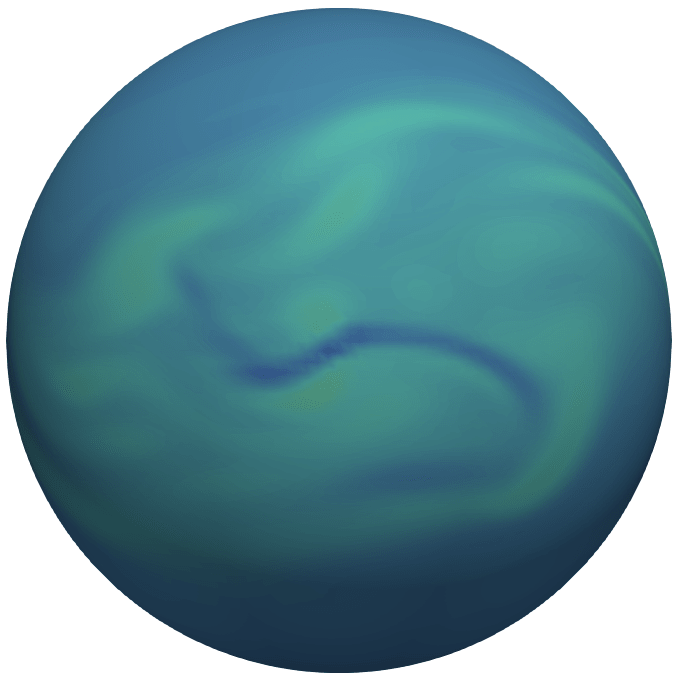}
\end{overpic}~
\begin{overpic}[width=0.11\textwidth,grid=false,tics=10]{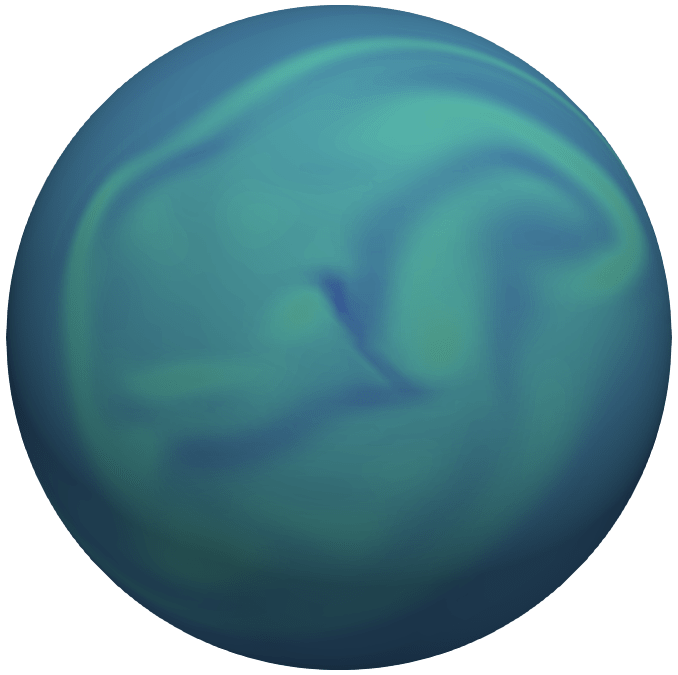}
\end{overpic} \\
\vskip .2cm
\begin{overpic}[width=0.5\textwidth,grid=false,tics=10]{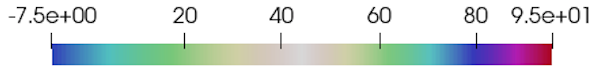}
\end{overpic}
	\caption{KH instability: evolution of the vorticity for different values of line tension:
	$\sigma_\gamma=0$ (top), $\sigma_\gamma=0.01$ (center), and $\sigma_\gamma = 0.1$ (bottom).
A full animation can be viewed following the link \href{https://youtu.be/FdMznBuMJPE}{\underline{youtu.be/FdMznBuMJPE}}
}
	\label{khwh}
\end{figure}

\subsection{The Rayleigh--Taylor instability}

The Rayleigh--Taylor (RT) instability occurs when a gravity force is acting on
a heavier fluid that lies above a lighter fluid. As the RT instability develops, ``plumes'' of the lighter
fluid flow upwards (with respect to the gravitational field) and ``spikes'' of the heavier fluid
fall downwards. We will simulate the RT instability on a sphere and on a torus with the aim of
investigating the effect of the geometry. In addition, we will vary line tension and fluid
viscosity.

We take two fluids with densities $\rho_2 = 3$, $\rho_1 = 1$ and matching viscosities
$\eta_1 =\eta_2 = \eta$, which will be specified for each test.
The initial surface fraction is given by
\[
c_0 =  \frac{1}{2} \left(1 + \tanh\frac{z + z_{rand}} {2 \sqrt{2} \epsilon}\right) ,
\]
where $\epsilon = 0.025$ and $z_{rand}$ is a uniformly generated random number from the range $(-0.1\epsilon, 0.1\epsilon)$.
The role of the perturbation generated by $z_{rand}$ is to onset the RT instability. We set  $M = 0.0025$.

Let us start with the sphere. We select mesh level $\ell = 5$
(see mesh description for the convergence test) and set $\Delta t = 0.1$. Fig.~\ref{rt_torus} shows the evolution of the surface fractions and velocity field
for $\eta = 10^{-2}$ and two values of line tension: $\sigma_\gamma=0 $  and $\sigma_\gamma=0.025$. At time $t = 7$, for $\sigma_\gamma=0.025$
we observe the characteristic flow structures of the RT instability. Instead, for $\sigma_\gamma = 0$
such structures have already broken up at $t = 7$. The effect of line tension is also seen at $t = 30$:
for $\sigma_\gamma = 0.025$ we observe that the heavier fluid has already settled at the bottom of the sphere,
while for $\sigma_\gamma = 0$ that has not happened yet. It takes till $t = 55$ to have the heavier fluid
at the bottom in the absence of line tension. After the revolution, the fluid phases  do not achieve steady state quickly but the waves
keep traveling along the equator.

\vskip .3cm
\begin{figure}[htb]
\centering
\hskip 1.2cm
\begin{overpic}[width=0.14\textwidth,grid=false,tics=10]{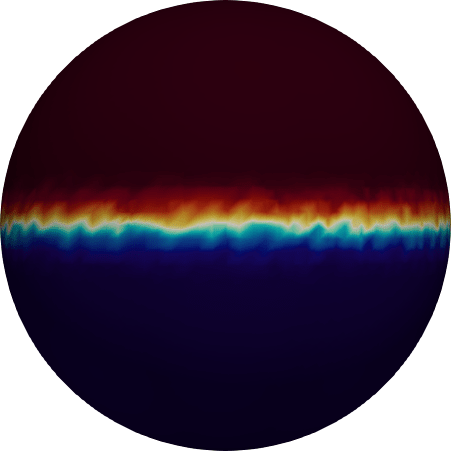}
\put(-85,45){$\sigma_\gamma = 0$}
\put(30,105){$t = 0$}
\end{overpic}~
\begin{overpic}[width=0.14\textwidth,grid=false,tics=10]{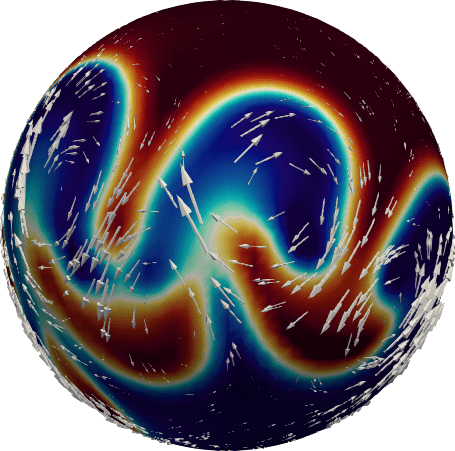}
\put(30,105){$t = 7$}
\end{overpic}~
\begin{overpic}[width=0.14\textwidth,grid=false,tics=10]{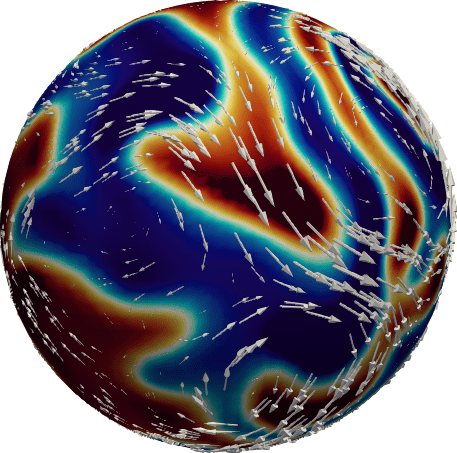}
\put(27,105){$t = 14$}
\end{overpic}~
\begin{overpic}[width=0.14\textwidth,grid=false,tics=10]{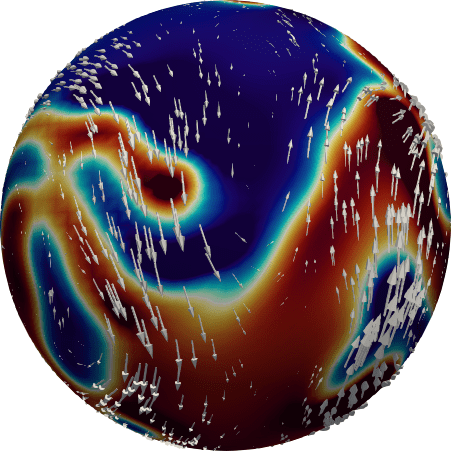}
\put(27,105){$t = 20$}
\end{overpic}~
\begin{overpic}[width=0.14\textwidth,grid=false,tics=10]{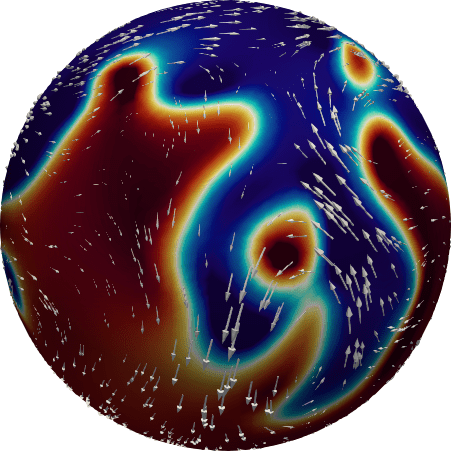}
\put(27,105){$t = 30$}
\end{overpic}~
\begin{overpic}[width=0.14\textwidth,grid=false,tics=10]{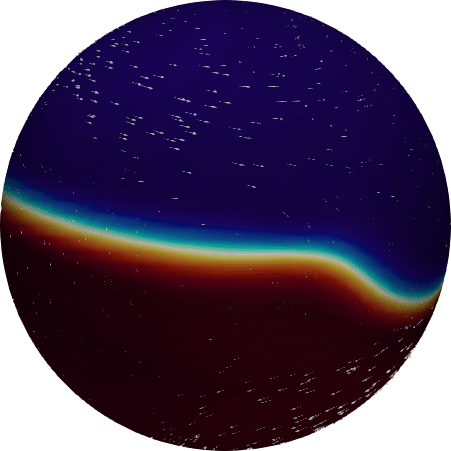}
\put(27,105){$t = 55$}
\end{overpic} \\
\vskip .1cm
\hskip 1.2cm
\begin{overpic}[width=0.14\textwidth,grid=false,tics=10]{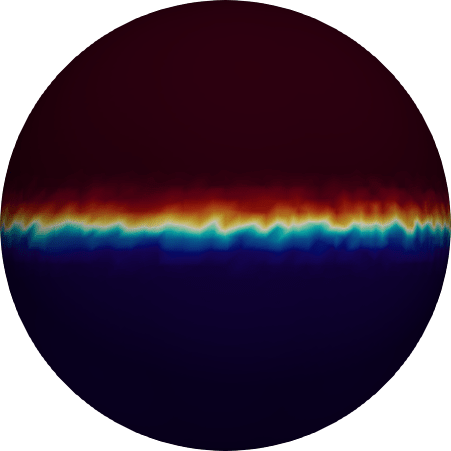}
\put(-85,45){$\sigma_\gamma = 0.025$}
\end{overpic}~
\begin{overpic}[width=0.14\textwidth,grid=false,tics=10]{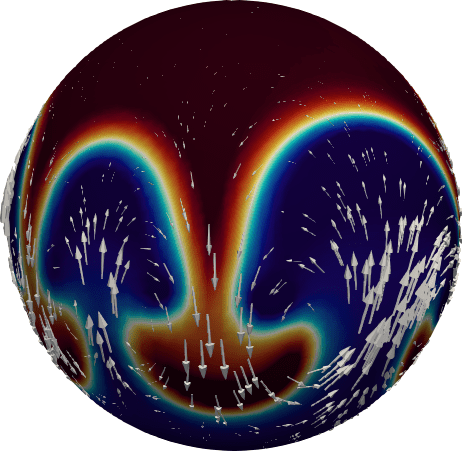}
\end{overpic}~
\begin{overpic}[width=0.14\textwidth,grid=false,tics=10]{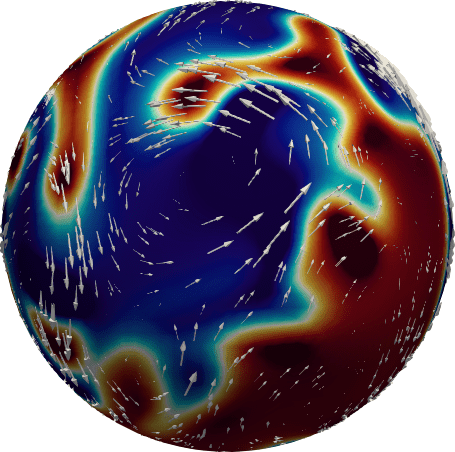}
\end{overpic}~
\begin{overpic}[width=0.14\textwidth,grid=false,tics=10]{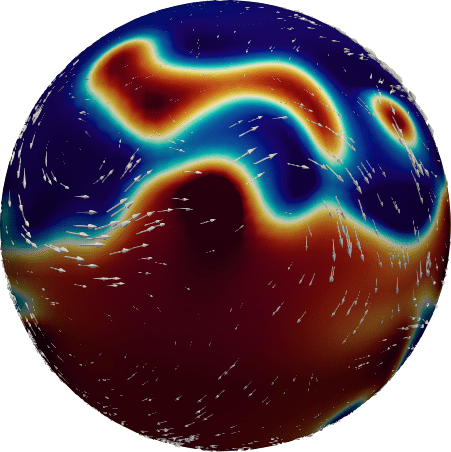}
\end{overpic}~
\begin{overpic}[width=0.14\textwidth,grid=false,tics=10]{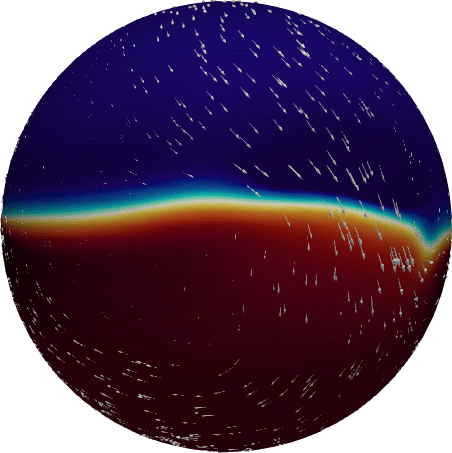}
\end{overpic}~
\begin{overpic}[width=0.14\textwidth,grid=false,tics=10]{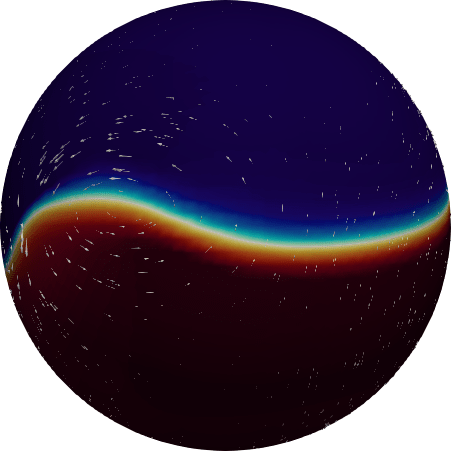}
\end{overpic} \\
\vskip .2cm
\begin{overpic}[width=0.5\textwidth,grid=false,tics=10]{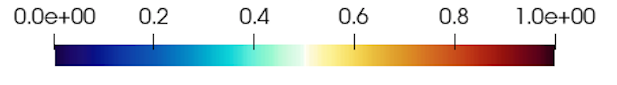}
\end{overpic}
	\caption{RT instability on the sphere: evolution of the order parameter (color) and velocity field (arrows)
	for $\eta = 10^{-2}$ and different values of line tension:
	$\sigma_\gamma=0$ (top) and $\sigma_\gamma = 0.025$ (bottom).
A full animation can be viewed following the link \href{https://youtu.be/_OmR_-qxvAI}{\underline{youtu.be/${}_{-}$OmR${}_{-}$-qxvAI}}
}
	\label{rt_sphere}
\end{figure}

Next, we consider an assymetric torus with constant distant from the center of the tube to the origin
$R=1$ and variable radius of the tube: $r_{min}=0.3\leq r(x,y) \leq r_{max}=0.6$, with
$r(x,y) =r_{min} + 0.5 (r_{max} - r_{min})  (1 - \frac{x}{\sqrt{x^2 + y^2}})$.
We characterize the torus surface as the zero level set of function $\phi = (x^2 + y^2 + z^2 + R^2 - r(x,y)^2)^2 - 4 R^2 (x^2 + y^2)$.
The torus is embedded in an outer domain $\Omega=[-5/3,5/3]^3$, just like the sphere.
We also selected same mesh level ($l=5$) and same time step ($\Delta t =0.1$) as for the sphere.
We set the line tension to $\sigma = 0.025$ and vary the viscosity: $\eta = 10^{-2}, 10^{-1}, 1$.
Fig.~\ref{rt_torus} displays the evolution of the surface fractions for these three values of viscosity.
First, we observe that in all cases the instability develops more slowly on the ``skinny'' side of the torus.
See second column in Fig.~\ref{rt_torus}. The fact that geometry has a considerable effect on the surface
RT instability is also clear when one compares the results on the sphere and the torus for
the same values of $\sigma_\gamma$ and $\eta$, i.e.~the top row in Fig.~\ref{rt_torus}
with the bottom row in Fig.~\ref{rt_sphere}. In particular, notice that while the heavier fluid
reaches the bottom of the sphere around $t = 30$ (Fig.~\ref{rt_sphere}, bottom second-last panel),
the two fluids are still very much mixed on the torus at $t = 160$ (Fig.~\ref{rt_torus}, top left panel).
We need to increase the viscosity value to 1 to be able to see most of the heavier fluid at the bottom
of the torus at $t = 160$ (Fig.~\ref{rt_torus}, bottom left panel), although that is still far from being settled.

\begin{figure}[htb]
\centering
\hskip .7cm
\begin{overpic}[width=0.14\textwidth,grid=false,tics=10]{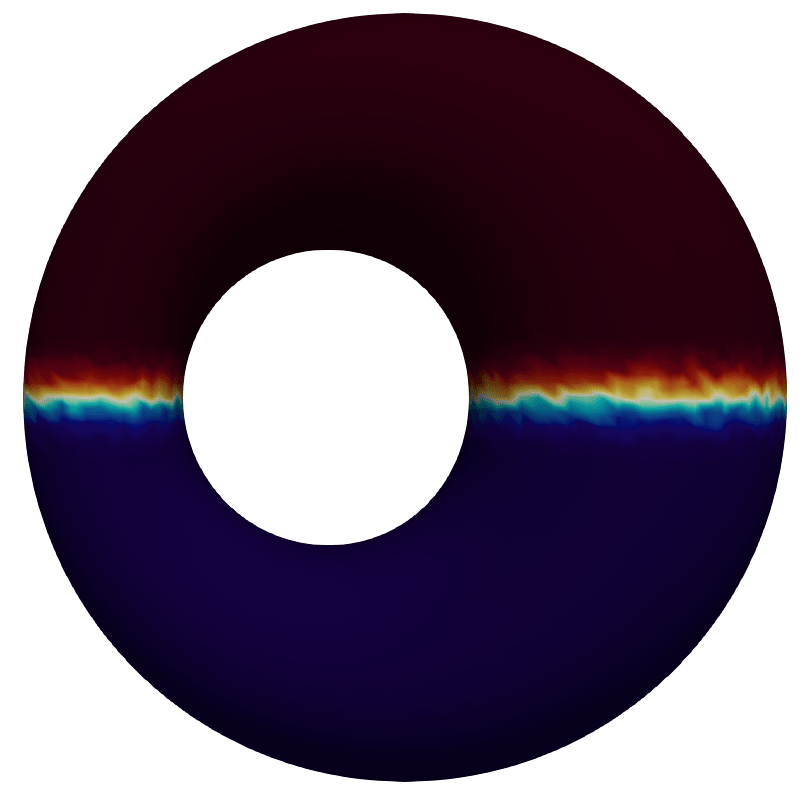}
\put(-70,45){$\eta = 10^{-2}$}
\put(30,100){$t = 0$}
\end{overpic}~
\begin{overpic}[width=0.14\textwidth,grid=false,tics=10]{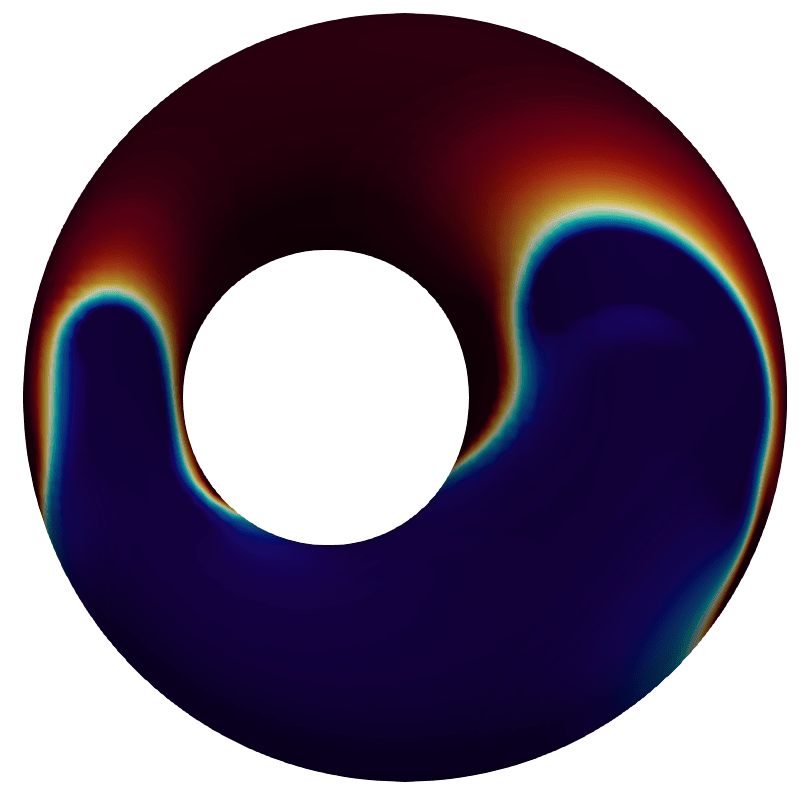}
\put(30,100){$t = 2$}
\end{overpic}~
\begin{overpic}[width=0.14\textwidth,grid=false,tics=10]{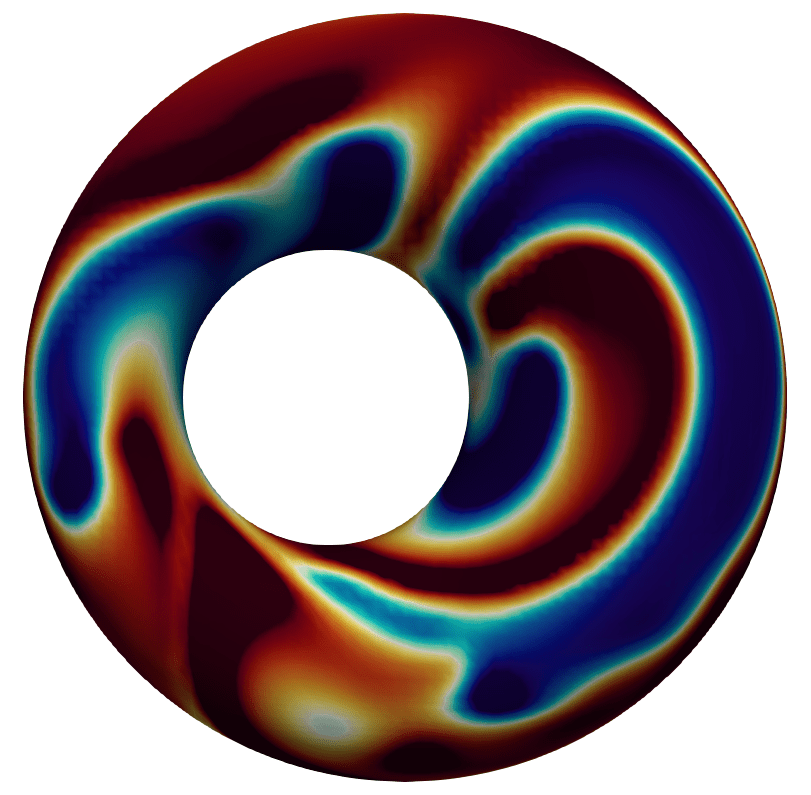}
\put(30,100){$t = 4$}
\end{overpic}~
\begin{overpic}[width=0.14\textwidth,grid=false,tics=10]{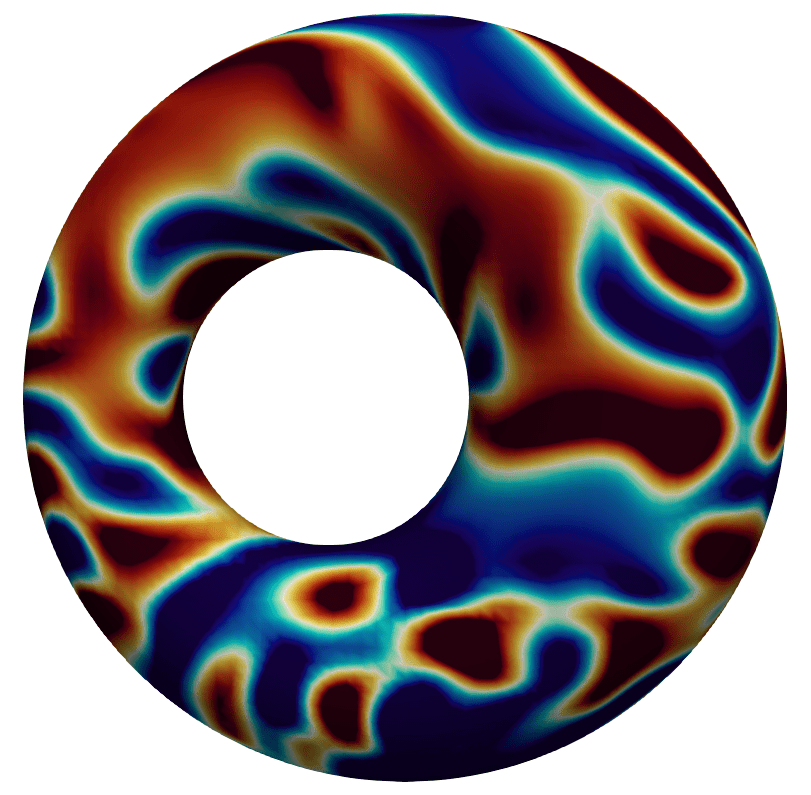}
\put(27,100){$t = 10$}
\end{overpic}~
\begin{overpic}[width=0.14\textwidth,grid=false,tics=10]{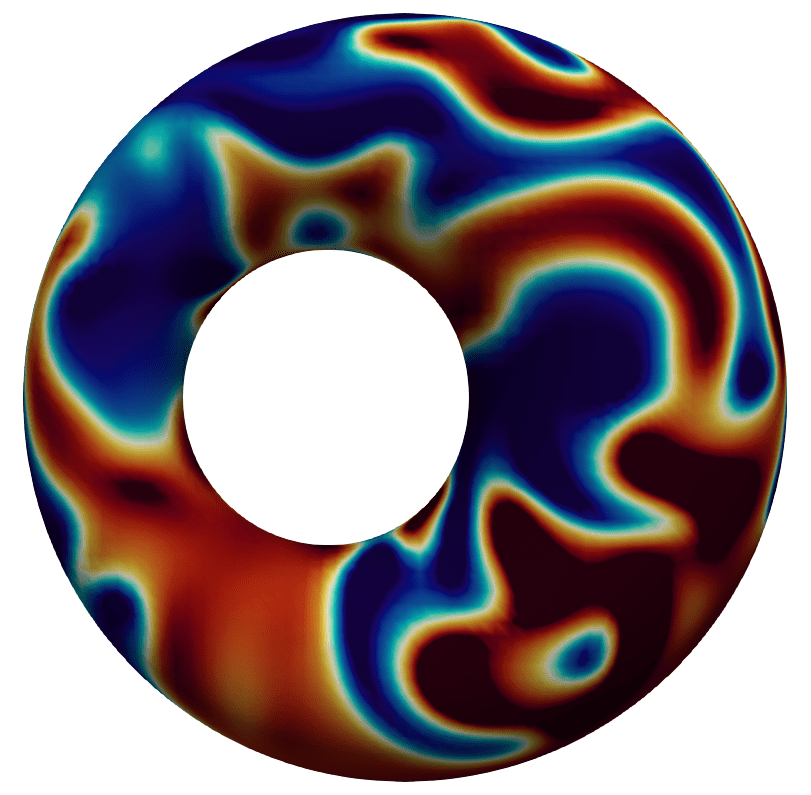}
\put(27,100){$t = 40$}
\end{overpic}~
\begin{overpic}[width=0.14\textwidth,grid=false,tics=10]{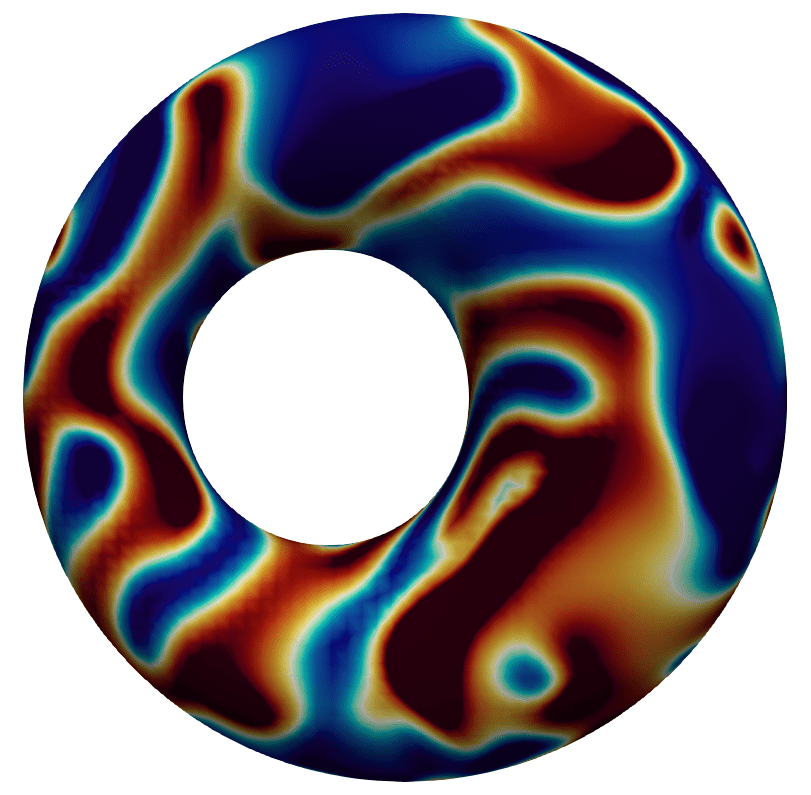}
\put(25,100){$t = 160$}
\end{overpic}\\
\hskip .7cm
\begin{overpic}[width=0.14\textwidth,grid=false,tics=10]{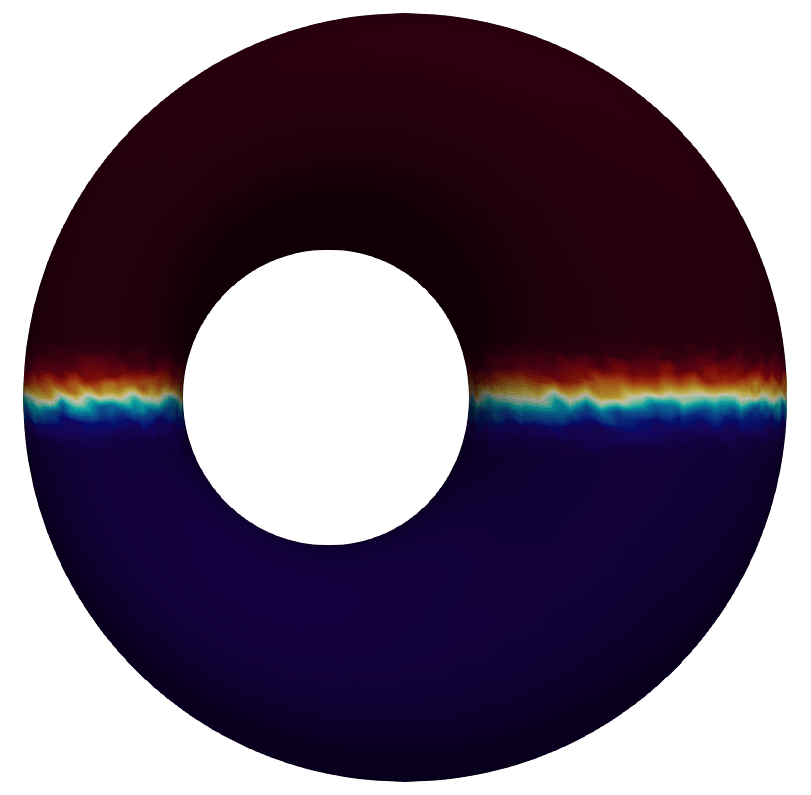}
\put(-70,45){$\eta = 10^{-1}$}
\end{overpic}~
\begin{overpic}[width=0.14\textwidth,grid=false,tics=10]{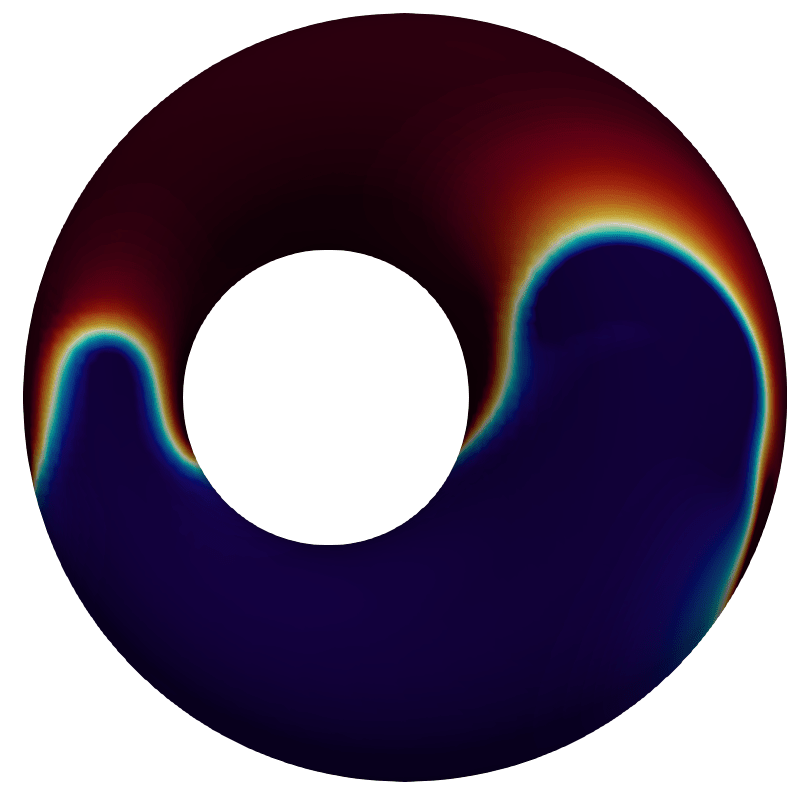}
\end{overpic}~
\begin{overpic}[width=0.14\textwidth,grid=false,tics=10]{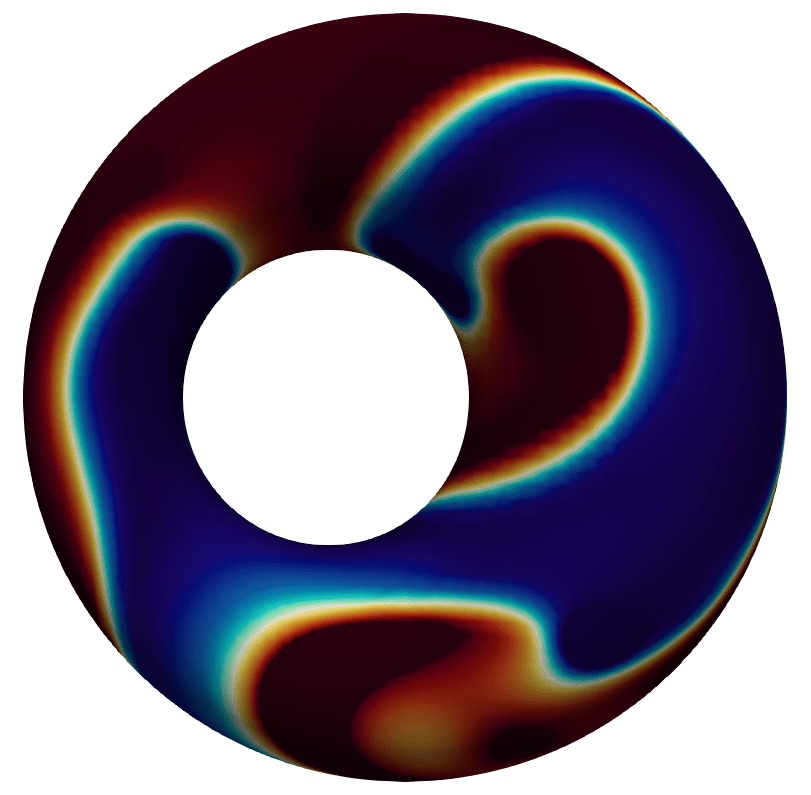}
\end{overpic}~
\begin{overpic}[width=0.14\textwidth,grid=false,tics=10]{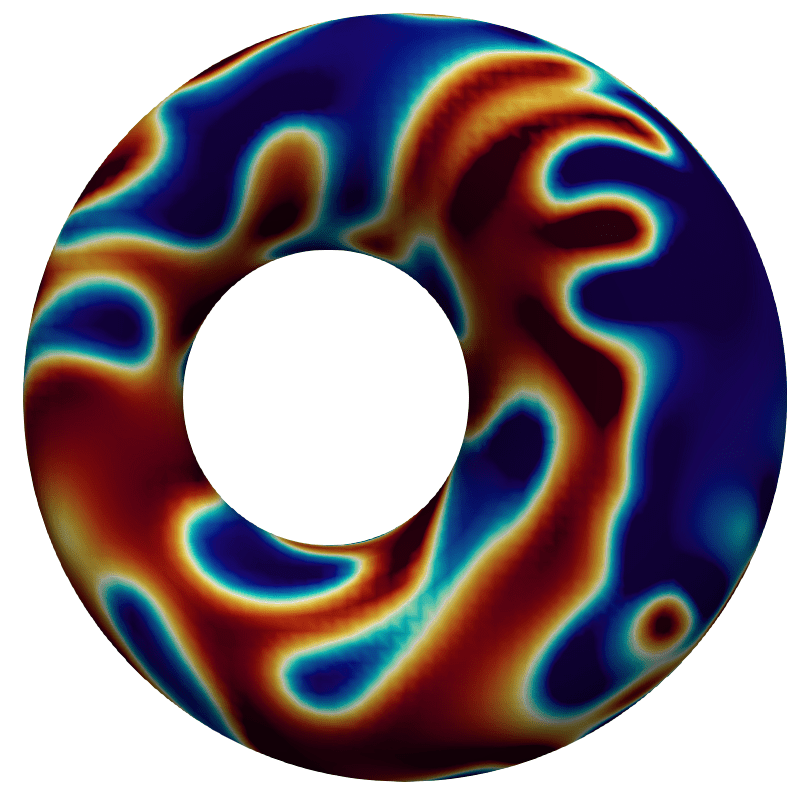}
\end{overpic}~
\begin{overpic}[width=0.14\textwidth,grid=false,tics=10]{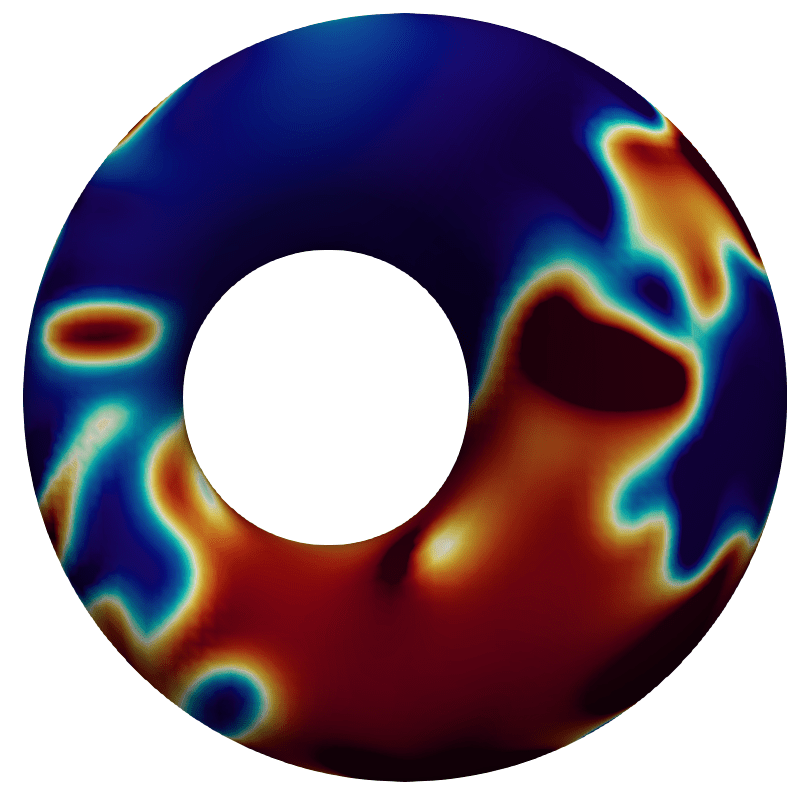}
\end{overpic}~
\begin{overpic}[width=0.14\textwidth,grid=false,tics=10]{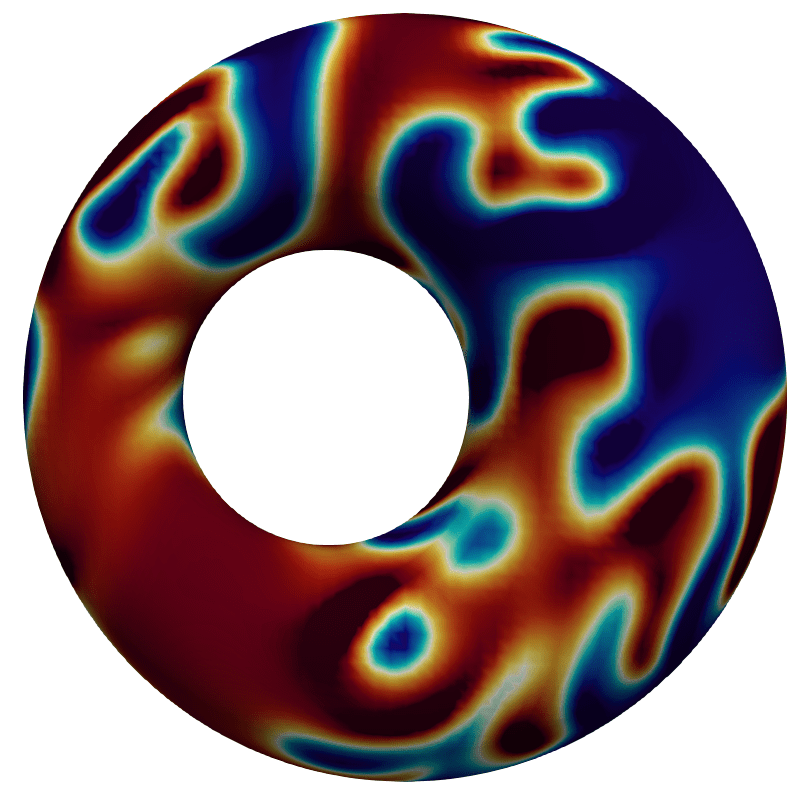}
\end{overpic}\\
\hskip .7cm
\begin{overpic}[width=0.14\textwidth,grid=false,tics=10]{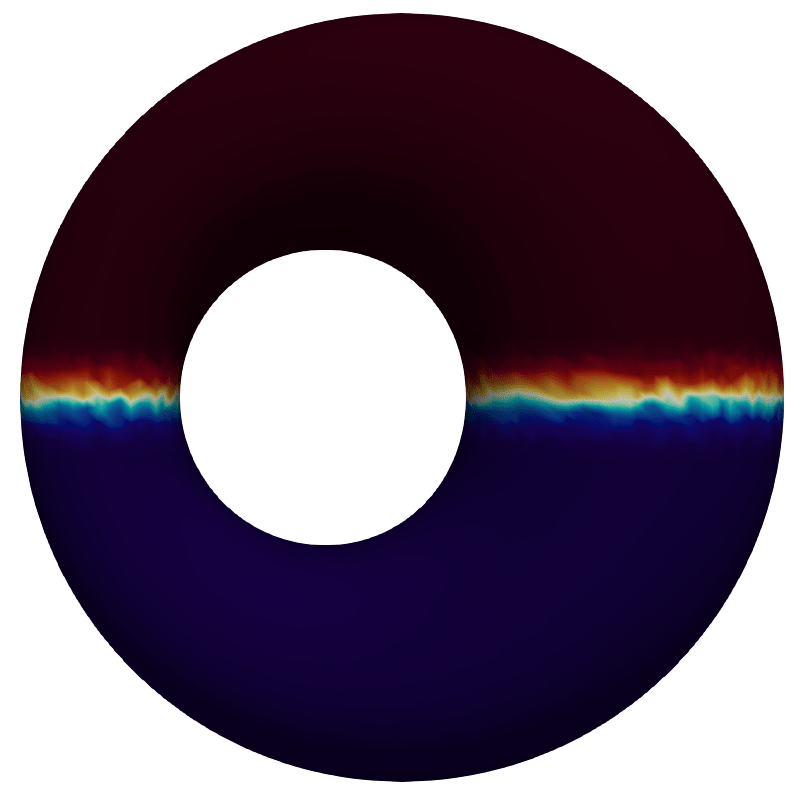}
\put(-70,45){$\eta = 1$}
\end{overpic}~
\begin{overpic}[width=0.14\textwidth,grid=false,tics=10]{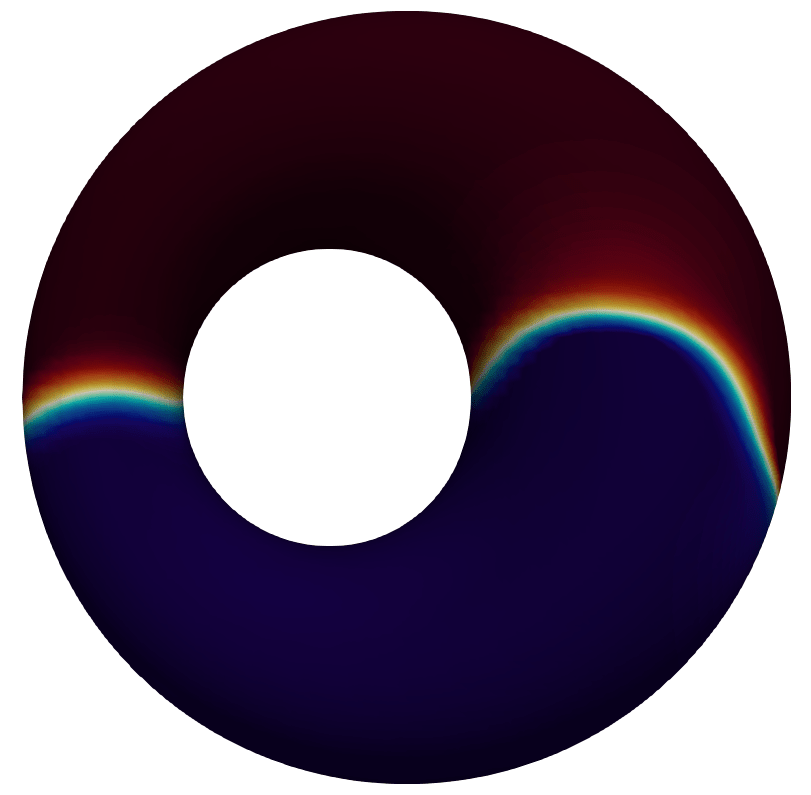}
\end{overpic}~
\begin{overpic}[width=0.14\textwidth,grid=false,tics=10]{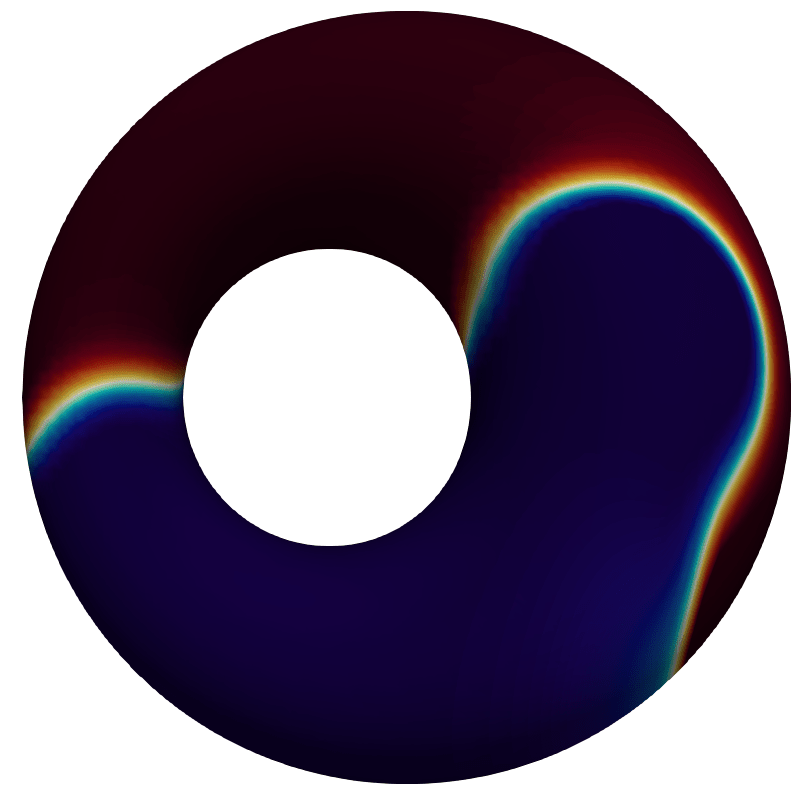}
\end{overpic}~
\begin{overpic}[width=0.14\textwidth,grid=false,tics=10]{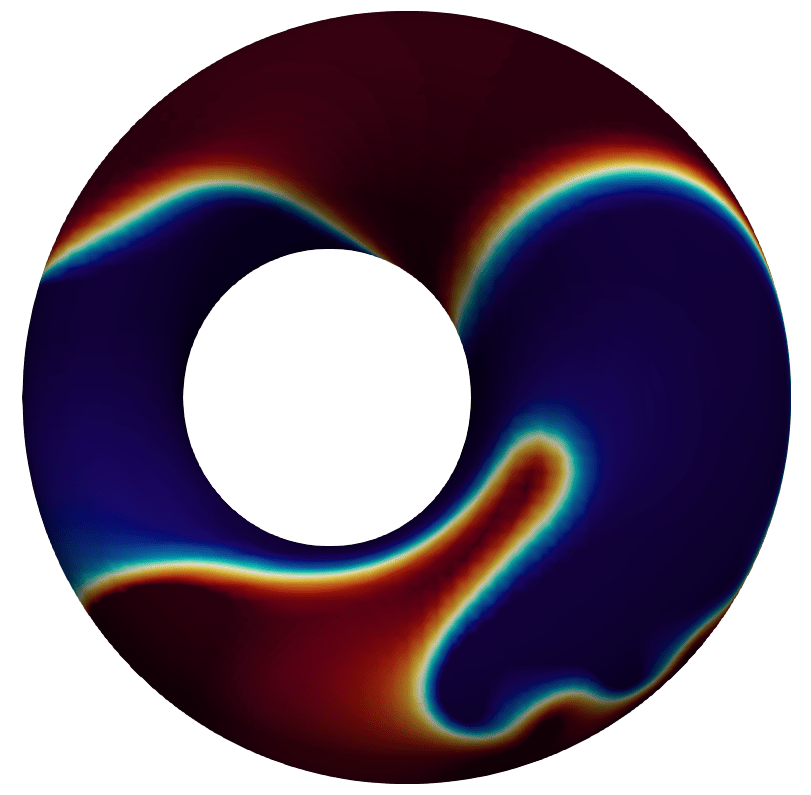}
\end{overpic}~
\begin{overpic}[width=0.14\textwidth,grid=false,tics=10]{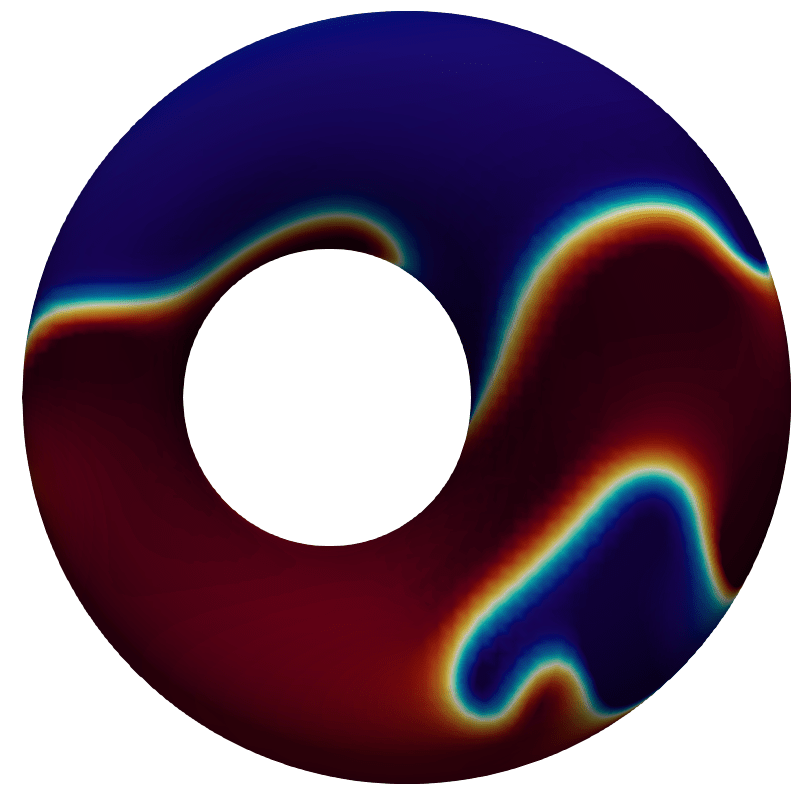}
\end{overpic}~
\begin{overpic}[width=0.14\textwidth,grid=false,tics=10]{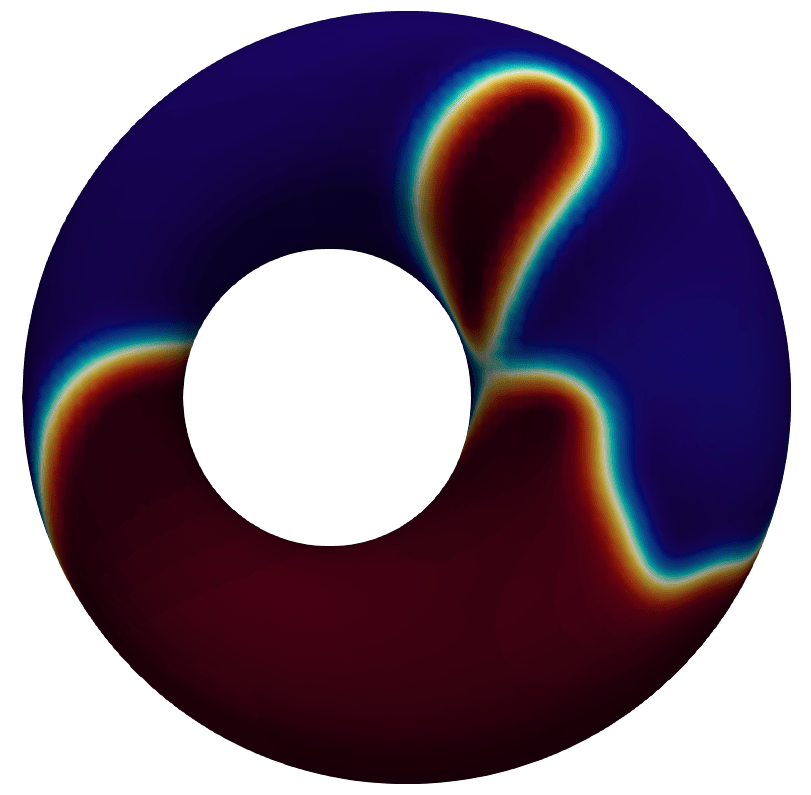}
\end{overpic}\\
\begin{overpic}[width=0.5\textwidth,grid=false,tics=10]{images/lab.png}
\end{overpic}
	\caption{RT instability on the torus: evolution of the order parameter for $\sigma_\gamma = 0.025$ and
	different values of viscosity: $\eta = 10^{-2}$ (top), $\eta = 10^{-1}$ (center), and $\eta = 1$ (bottom).
A full animation can be viewed following the link \href{https://youtu.be/FTqqFjvzEZg}{\underline{youtu.be/FTqqFjvzEZg}}
	}
	\label{rt_torus}
\end{figure}

\section{Conclusions}\label{sec:concl}

In this paper, we presented an extension of a well-known phase field model
for two-phase incompressible flow, and we applied and analyzed an unfitted finite element method
for its numerical approximation.
The advantage of our model is that it is thermodynamically consistent for a general
monotone relation of density and phase-field variable. Because of our interest in biomembranes,
this Navier--Stokes--Cahn--Hilliard type system is posed on an arbitrary-shaped closed smooth surface.
Although in this paper we assumed the surfaces to be rigid, our long term goal is to
study two-phase incompressible flow on evolving shapes. In fact, biological membranes exhibit
shape changes that need to be accounted for in a realistic model. This need dictated our choice
for the numerical approach, which is a versatile geometrically unfitted finite element method
called TraceFEM.

In order to reduce the computational cost, the discrete scheme we proposed decouples the fluid problem
(a linearized Navier--Stokes type problem) from the phase-field problem (a Cahn--Hilliard type problem with
constant mobility) at each time step. An attractive feature of our scheme is that
the numerical solution satisfies the same stability bound as the solution of the original system
under some restrictions on the discretization parameters.

We validated our implementation of the proposed numerical scheme with a benchmark problem
and applied it to simulate well-known two-phase fluid flows: the Kelvin--Helmholtz and Rayleigh--Taylor instabilities.
We investigated the effect of line tension on such instabilities.
For the Rayleigh--Taylor instability, we also assessed the effect of viscosity
and surface shape, which plays an important role in the evolution of the instability.

\section*{Acknowledgments}
This work was partially supported by US National Science Foundation (NSF) through grant DMS-1953535.
A.Z. and M.O. ~also acknowledge the support from NSF through DMS-2011444.
A.Q.~also acknowledges the support from NSF through DMS-1620384.

\bibliographystyle{abbrv}
\bibliography{main}{}

\end{document}